\documentclass[reqno,11pt]{article}
\usepackage{a4wide}
\usepackage{color,eucal,enumerate,mathrsfs,amsthm}
\usepackage[normalem]{ulem}
\usepackage{amsmath,amssymb,amsthm,epsfig,bbm}
\numberwithin{equation}{section}

\usepackage{amsfonts}
\usepackage{dsfont}
\usepackage{mathtools}
\usepackage{leftidx}
\usepackage{graphicx}
\usepackage{esint}
\usepackage{authblk}

\usepackage[pdfborder={0 0 0}]{hyperref}  

\usepackage[latin1]{inputenc}
\usepackage[english]{babel}

\newtheorem{Definition}{Definition}[section]
\newtheorem{Proposition}[Definition]{Proposition}
\newtheorem{Lemma}[Definition]{Lemma}
\newtheorem{Theorem}[Definition]{Theorem}
\newtheorem{Corollary}[Definition]{Corollary}
\theoremstyle{definition}

\newtheorem{Remark}[Definition]{Remark}
\newtheorem{Setting}[Definition]{Setting}

\allowdisplaybreaks

\renewcommand{\H}{\mathbb{H}}

\newcommand{\N}{\mathbb{N}}

\newcommand{\R}{\mathbb{R}}

\newcommand{\mm}{{\mbox{\boldmath$m$}}}

\newcommand{\ggamma}{{\mbox{\boldmath$\gamma$}}}

\newcommand{\ppi}{{\mbox{\boldmath$\pi$}}}

\newcommand{\sggamma}{{\mbox{\scriptsize\boldmath$\gamma$}}}

\newcommand{\sfd}{{\sf d}}

\newcommand{\sfh}{{\sf h}}

\newcommand{\Kliminf}{K\kern-3pt-\kern-2pt\mathop{\rm lim\,inf}\limits}  
\newcommand{\supp}{\mathop{\rm supp}\nolimits}   
\newcommand{\Lip}{\mathop{\rm Lip}\nolimits}          
\renewcommand{\d}{{\mathrm d}}
\newcommand{\dt}{{\d t}}
\newcommand{\ddt}{{\frac \d\dt}}

\newcommand{\restr}[1]{\lower3pt\hbox{$|_{#1}$}} 

\newcommand{\la}{\left<}                  
\newcommand{\ra}{\right>}
\newcommand{\eps}{\varepsilon}  
\newcommand{\nchi}{{\raise.3ex\hbox{$\chi$}}}
\newcommand{\weakto}{\rightharpoonup}

\newcommand{\limi}{\varliminf}
\newcommand{\lims}{\varlimsup}

\newcommand{\fr}{\penalty-20\null\hfill$\blacksquare$}                      

\newcommand{\prob}[1]{\mathscr P(#1)}                   
\newcommand{\probt}[1]{\mathscr P_2(#1)}                   
\newcommand{\e}{{\rm{e}}}                           

\renewcommand{\mm}{\mathfrak m}                                

\renewenvironment{proof}{\removelastskip\par\medskip   
\noindent{\em proof} \rm}{\penalty-20\null\hfill$\square$\par\medbreak}

\newcommand{\bd}{{\mathbf\Delta}}

\newcommand{\testi}[1]{{\rm Test}^{\infty}(#1)}
\newcommand{\testip}[1]{{\rm Test}^{\infty}_+(#1)}
\newcommand{\testipp}[1]{{\rm Test}^{\infty}_{>0}(#1)}

\newcommand{\X}{{\rm X}}

\newcommand{\h}{{\sfh}}

\renewcommand{\ae}{{\textrm{\rm{-a.e.}}}}
\newcommand{\aeon}{{\textrm{\rm{-a.e. on }}}}

\newcommand{\CD}{{\sf CD}}
\newcommand{\RCD}{{\sf RCD}}
\newcommand{\Ggamma}{{\mathbf\Gamma}}

\newcommand{\lip}{{\rm lip}}

\newcommand{\HS}{{\lower.3ex\hbox{\scriptsize{\sf HS}}}}
\renewcommand{\H}[1]{{\rm Hess}(#1)}
\renewcommand{\div}{{\rm div}}

\newcommand{\Y}{{\rm Y}}
\newcommand{\hr}{{\sf r}}
\newcommand{\hR}{{\sf R}}

\title{Second order differentiation formula on compact  $\RCD^*(K,N)$ spaces}
\author{Nicola Gigli \thanks{SISSA, Trieste. email: ngigli@sissa.it} \quad Luca Tamanini \thanks{SISSA, Trieste \& Universit\'e Paris Ouest, Nanterre. email: ltamanini@sissa.it, luca.tamanini@u-paris10.fr}}

\begin{document}

\maketitle

\begin{abstract}
Aim of this paper is to prove the second order differentiation formula along geodesics in compact $\RCD^*(K,N)$ spaces with $N < \infty$. This formula is new even in the context of Alexandrov spaces.

We establish this result by showing that $W_2$-geodesics can be approximated up to second order, in a sense which we shall make precise, by entropic interpolation. In turn this is achieved by proving new,  even in the smooth setting, estimates concerning entropic interpolations which we believe are interesting on their own. In particular we obtain:
\begin{itemize}
\item[-]  equiboundedness of the densities along the entropic interpolations,
\item[-] equi-Lipschitz continuity of the Schr\"odinger potentials,
\item[-] a uniform weighted $L^2$ control of the Hessian of such potentials.
\end{itemize}
Finally, the techniques used in this paper can be used to show that  the viscous solution of the Hamilton-Jacobi equation can be obtained  via a vanishing viscosity method, in accordance with the smooth case.
\end{abstract}

\tableofcontents

\section{Introduction}
In the last ten years there has been a great interest in the study of metric measure spaces with Ricci curvature bounded from below, see for instance \cite{Lott-Villani09},  \cite{Sturm06I}, \cite{Sturm06II}, \cite{Gigli-Kuwada-Ohta10}, \cite{AmbrosioGigliSavare11}, \cite{AmbrosioGigliSavare11-2},  \cite{Gigli12}, \cite{AmbrosioGigliSavare12}, \cite{Rajala12}, \cite{RajalaSturm12}, \cite{Gigli-Mosconi12},  \cite{Gigli13}, \cite{Gigli13over}, \cite{Ketterer13}, \cite{AmbrosioMondinoSavare13}, \cite{Mondino-Naber14}, \cite{CavMon15}, \cite{CavMil16}. The starting points of this research line have been the seminal papers \cite{Lott-Villani09} and \cite{Sturm06I}, \cite{Sturm06II} which linked lower Ricci bounds on metric measure spaces to properties of entropy-like functionals in connection with $W_2$-geometry.  Later (\cite{AmbrosioGigliSavare11}) it emerged  that also Sobolev calculus is linked to $W_2$-geometry and building on top of this the original definition of $\CD$ spaces by Lott-Sturm-Villani has evolved into that of $\RCD$ spaces (\cite{AmbrosioGigliSavare11-2}, \cite{Gigli12}).

An example of link between Sobolev calculus and $W_2$-geometry is the following statement, proved in \cite{Gigli13}:
\begin{Theorem}[First order differentiation formula]\label{thm:1}
Let $(\X,\sfd,\mm)$ be a $\RCD(K,\infty)$ space, $(\mu_t)$ a $W_2$-geodesic made of measures with bounded support and such that $\mu_t\leq C\mm$ for every $t\in[0,1]$ and some $C>0$. Then for every $f\in W^{1,2}(\X)$ the map 
\[
[0,1]\ni t\quad\mapsto\quad\int f\,\d\mu_t
\]
is $C^1$ and we have
\[
\frac\d{\d t}\int f\,\d\mu_t\restr{t=0}=-\int \d f(\nabla\varphi)\,\d\mu_0,
\]
where  $\varphi$ is any locally Lipschitz Kantorovich potential from $\mu_0$ to $\mu_1$.
\end{Theorem}
Recall that on $\RCD(K,\infty)$ spaces every $W_2$-geodesic $(\mu_t)$ between measures with bounded density and support is such that $\mu_t\leq C\mm$ for every $t\in[0,1]$ and some $C>0$ (\cite{RajalaSturm12}), so that the theorem also says that we can find `many' $C^1$ functions on $\RCD$ spaces. We remark that such   $C^1$ regularity - which was crucial in \cite{Gigli13} - is non-trivial even if the function $f$ is assumed to be Lipschitz and that  statements about $C^1$ smoothness are quite rare in metric geometry.

One might think at Theorem \ref{thm:1} as an `integrated' version of the basic formula
\[
\frac{\d}{\d t}f(\gamma_t)\restr{t=0}=\d f(\gamma'_0)
\]
valid in the smooth framework; at the technical level the proof of the claim has to do with the fact that the geodesic $(\mu_t)$ solves the continuity equation
\begin{equation}
\label{eq:cint}
\frac\d{\d t}\mu_t+\div(\nabla(-\varphi_t)\mu_t)=0,
\end{equation}
where the $\varphi_t$'s are appropriate choices of Kantorovich potentials (see also \cite{GigliHan13} in this direction).

\bigskip

In \cite{Gigli14}, the first author developed a second-order calculus on $\RCD$ spaces, in particular defining the space $H^{2,2}(\X)$ and for $f\in H^{2,2}(\X)$ the Hessian $\H f$, see \cite{Gigli14} and the preliminary section. It is then natural to ask whether an `integrated' version of the second order differentiation formula
\[
\frac{\d^2}{\d t^2}f(\gamma_t)\restr{t=0}=\H f(\gamma'_0,\gamma'_0)\qquad\text{for $\gamma$ geodesic}
\]
holds in this framework. In this paper we provide affirmative answer to this question, our main result being:
\begin{Theorem}[Second order differentiation formula]\label{thm:2}
Let $(\X,\sfd,\mm)$ be a compact $\RCD^*(K,N)$ space, $N<\infty$,  $(\mu_t)$ a $W_2$-geodesic such that $\mu_t\leq C\mm$ for every $t\in[0,1]$ and some $C>0$ and $f\in H^{2,2}(\X)$.

Then the function
\[
[0,1]\ni t\quad\mapsto\quad\int f\,\d\mu_t
\]
is $C^2$ and we have
\begin{equation}
\label{eq:sd}
\frac{\d^2}{\d t^2}\int f\,\d\mu_t\restr{t=0}=\int\H f(\nabla\varphi,\nabla\varphi)\,\d\mu_0,
\end{equation}
where  $\varphi$ is any Kantorovich potential from $\mu_0$ to $\mu_1$.
\end{Theorem}
See also Theorem \ref{thm:main} for an alternative, but equivalent, formulation of the result. We wish to stress that based on the kind of arguments used in our proof, we do not believe the compactness assumption to be crucial (but being our proof based on global analysis, to remove it is not a trivial task, especially in the case $K<0$), while on the other hand the finite dimensionality plays a key role (e.g.\ because we use the Li-Yau inequality).

Having at disposal such second order differentiation formula - perhaps without the restriction of working in compact spaces - is interesting not only at the theoretical level, but also for applications to the study of the geometry of $\RCD$ spaces. For instance, the proofs of both the splitting theorem and of the `volume cone implies metric cone' in this setting can be greatly simplified by using such formula. Also, one aspect of the theory of $\RCD$ spaces which is not yet clear is whether they have constant dimension: for Ricci-limit spaces this is known to be true by a result of Colding-Naber \cite{ColdingNaber12} which uses second order derivatives along geodesics in a crucial way. Thus our result is necessary to replicate Colding-Naber argument in the non-smooth setting (but not sufficient: they also use a calculus with Jacobi fields which as of today does not have  a non-smooth counterpart).

\bigskip

Let us discuss the strategy of the proof. Our starting point is a related second order differentiation formula obtained in \cite{Gigli14}, available under proper regularity assumptions:
\begin{Theorem}\label{thm:1i}
Let $(\mu_t)$ be a $W_2$-absolutely continuous curve solving the continuity equation
\[
\frac\d{\d t}\mu_t+\div(X_t\mu_t)=0,
\] 
for some vector fields $(X_t)\subset L^2(T\X)$ in the following sense: for every $f\in W^{1,2}(\X)$ the map $t\mapsto\int f\,\d\mu_t$  is absolutely continuous and it holds
\[
\frac\d{\d t}\int f\,\d\mu_t=\int\la\nabla f,X_t\ra\,\d\mu_t.
\]
Assume that 
\begin{itemize}
\item[(i)] $t \mapsto X_t \in L^2(T{\X})$ is absolutely continuous,
\item[(ii)] $\sup_t\{ \|X_t\|_{L^2} + \|X_t\|_{L^{\infty}}+ \|\nabla X_t\|_{L^2} \} < +\infty$.
\end{itemize} 
Then for $f\in H^{2,2}(\X)$ the map $t\mapsto\int f\,\d\mu_t$ is $C^{1,1}$ and the formula
\begin{equation}
\label{eq:secondsmooth}
\frac{\d^2}{\d t^2}\int f\d\mu_t = \int \H{f}(X_t,X_t) + \la\nabla f,\tfrac\d{\d t} X_t  + \nabla_{X_t} X_t\ra  \d\mu_t
\end{equation}
holds for a.e.\ $t \in [0,1]$.
\end{Theorem}
If the vector fields $X_t$ are of gradient type, so that $X_t=\nabla\phi_t$ for every $t$ and the `acceleration' $a_t$ is defined as
\[
\frac\d{\d t}\phi_t+\frac{|\nabla\phi_t|^2}2=:a_t
\]
then \eqref{eq:secondsmooth} reads as
\begin{equation}
\label{eq:secondsmooth2}
\frac{\d^2}{\d t^2}\int f\d\mu_t = \int \H{f}(\nabla\phi_t,\nabla\phi_t)\,\d\mu_t+\int  \la\nabla f,\nabla a_t\ra  \d\mu_t.
\end{equation}
In the case of geodesics, the functions  $\varphi_t$ appearing in \eqref{eq:cint} solve (in a sense which we will not make precise here) the Hamilton-Jacobi equation
\begin{equation}
\label{eq:hji}
\frac\d{\d t}\varphi_t=\frac{|\nabla\varphi_t|^2}{2},
\end{equation}
thus in this case the acceleration $a_t$ is identically 0 (notice the minus sign in \eqref{eq:cint}). Hence if the vector fields $(\nabla\varphi_t)$ satisfy the regularity requirements $(i),(ii)$ in the last theorem we would easily be able to establish Theorem \ref{thm:2}. However in general this is not the case; informally speaking this has to do with the fact that for solutions of the Hamilton-Jacobi equations we do not have sufficiently strong second order estimates.

In order to establish Theorem \ref{thm:2} it is therefore natural to look for suitable `smooth' approximation of geodesics for which we can apply Theorem \ref{thm:1i} above and then pass to the limit in formula \eqref{eq:secondsmooth}. Given that the lack of smoothness of $W_2$-geodesic is related to the lack of smoothness of solutions of \eqref{eq:hji}, also in line with the classical theory of viscous approximation for the Hamilton-Jacobi equation there is a quite natural thing to try:  solve, for $\eps>0$, the equation
\[
\frac\d{\d t}\varphi^\eps_t=\frac{|\nabla\varphi^\eps_t|^2}{2}+\frac\eps2\Delta\varphi^\eps_t,\qquad\qquad\varphi^\eps_0:=\varphi,
\]
where $\varphi$ is a given, fixed, Kantorovich potential for the geodesic $(\mu_t)$, and then solve
\[
\frac\d{\d t}\mu^\eps_t-\div(\nabla\varphi^\eps_t\mu^\eps_t)=0,\qquad\qquad\mu^\eps_0:=\mu_0.
\]
This plan can actually be pursued and following the ideas in this paper one can show   that if the space $(\X,\sfd,\mm)$ is compact and $\RCD^*(K,N)$ and the geodesic $(\mu_t)$ is made of measures with equibounded densities, then as $\eps\downarrow0$:
\begin{itemize}
\item[i)] the curves $(\mu_t^\eps)$ $W_2$-uniformly converge to the geodesic $(\mu_t)$ and the measures $\mu^\eps_t$ have equibounded densities.
\item[ii)] the functions $\varphi^\eps_t$ are equi-Lipschitz and converge both uniformly and in the $W^{1,2}$-topology to the only viscous solution $(\varphi_t)$ of \eqref{eq:hji} with $\varphi$ as initial datum; in particular the continuity equation \eqref{eq:cint} for the limit curve holds.
\end{itemize}
These convergence results are based on Hamilton's gradient estimates and the Li-Yau inequality and are sufficient to pass to the limit in the term with the Hessian in \eqref{eq:secondsmooth2}. For these curves the acceleration is given by $a_t^\eps=-\frac\eps2\Delta\varphi^\eps_t$ and thus we are left to prove  that the   quantity
\[
\eps\int\la\nabla f,\nabla\Delta\varphi^\eps_t \ra\,\d\mu^\eps_t
\]
goes to 0 in some sense. However, there appears to be \emph{no hope of obtaining this by PDE estimates}. The problem is that this kind of viscous approximation can produce in the limit a curve which is not a geodesic if $\varphi$ is not $c$-concave: shortly said, this happens as  soon as a shock appears in Hamilton-Jacobi. Since there is no hope for formula \eqref{eq:sd} to be true for non-geodesics, we see that there is little chance of obtaining it via such viscous approximation.

\bigskip

We therefore use another way of approximating geodesics:  the slowing down of entropic interpolations.  Let us briefly describe what this is in the familiar Euclidean setting.

Fix two probability measures $\mu_0=\rho_0\mathcal L^d$, $\mu_1=\rho_1\mathcal L^d$ on $\R^d$. The Schr\"odinger functional equations are
\begin{equation}
\label{eq:sch10}
\rho_0=f\,\h_1g\qquad\qquad\qquad\qquad\rho_1=g\,\h_1f,
\end{equation}
the unknown being the Borel functions $f,g:\R^d\to[0,\infty)$, where $\h_tf$ is the heat flow starting at $f$ evaluated at time $t$. It turns out that in great generality these equations admit a solution which is unique up to the trivial transformation $(f,g)\mapsto (cf,g/c)$ for some constant $c>0$. Such solution can be found in the following way: let $\hR$ be the measure on $(\R^d)^2$ whose density w.r.t.\ $\mathcal L^{2d}$ is given by the heat kernel $\hr_t(x,y)$ at time $t=1$ and minimize the Boltzmann-Shannon entropy $H(\ggamma\,|\,\hR)$ among all transport plans $\ggamma$ from $\mu_0$ to $\mu_1$. The Euler equation for the minimizer forces it to be of the form  $f\otimes g\,\hR$ for some Borel functions $f,g:\R^d\to[0,\infty)$, where $f\otimes g(x,y):=f(x)g(y)$ (we shall reprove this known result in Proposition \ref{pro:2}). Then the fact that  $f\otimes g\,\hR$ is a transport plan from $\mu_0$ to $\mu_1$ is equivalent to $(f,g)$ solving \eqref{eq:sch10}. 

Once we have found the solution of \eqref{eq:sch10} we can use it in conjunction with the heat flow to interpolate from $\rho_0$ to $\rho_1$ by defining
\[
\rho_t:=\h_tf\,\h_{1-t}g.
\]
This is called entropic interpolation. Now we slow down the heat flow: fix $\eps>0$ and by mimicking the above find $f^\eps,g^\eps$ such that
\begin{equation}
\label{eq:sch1}
\rho_0=f^\eps\,\h_{\eps/2}g^\eps\qquad\qquad\rho_1=g^\eps\,\h_{\eps/2}f^\eps,
\end{equation}
(the factor $1/2$ plays no special role, but is convenient in computations).
Then define 
\[
\rho^\eps_t:=\h_{t\eps/2}f^\eps\,\h_{(1-t)\eps/2}g^\eps.
\]

The  remarkable and non-trivial fact here is that as $\eps\downarrow0$ the curves of measures $(\rho^\eps_t\mathcal L^d)$ converge to the $W_2$-geodesic from $\mu_0$ to $\mu_1$. 

The first connections between Schr\"odinger equations and optimal transport have been obtained by Mikami in \cite{Mikami04} for the quadratic cost on $\R^d$; later Mikami-Thieullen \cite{MikamiThieullen08} showed that a link persists even for more general cost functions. The statement we have just made about convergence of entropic interpolations to displacement ones has been proved by L\'eonard in \cite{Leonard12}. Actually, L\'eonard worked in much higher generality: as it is perhaps clear from the presentation, the construction of entropic interpolation can be done in great generality, as only a  heat kernel is needed. He also provided a basic intuition about why such convergence is in place: the basic idea is that if the heat kernel admits the asymptotic expansion $\eps\log \hr_\eps(x,y)\sim -\frac{\sfd^2(x,y)}{2}$ (in the sense of Large Deviations), then the rescaled entropy functionals $\eps H(\cdot\,|\, \hR_\eps)$ converge to $\frac12\int \sfd^2(x,y)\,\d \cdot $ (in the sense of $\Gamma$-convergence). We refer to \cite{Leonard14} for a deeper discussion of this topic, historical remarks and much more.

\bigskip

Starting from these intuitions and results,  working in the setting of compact $\RCD^*(K,N)$ spaces we gain new information about the convergence of entropic interpolations to displacement ones. In order to state our results, it is convenient to introduce the Schr\"odinger potentials $\varphi^\eps_t,\psi^\eps_t$ as
\[
\varphi^\eps_t:=\eps\log \h_{t\eps/2}f^\eps\qquad\qquad\qquad\qquad\psi^\eps_t:=\eps\log \h_{(1-t)\eps/2}g^\eps.
\]
In the limit $\eps\downarrow0$ these will converge to forward and backward Kantorovich potentials along the limit geodesic $(\mu_t)$ (see below). In this direction, it is worth to notice that while for $\eps>0$ there is a tight link between potentials and densities, as we trivially have
\[
\varphi^\eps_t+\psi^\eps_t=\eps\log\rho^\eps_t,
\]
in the limit this becomes the well known (weaker) relation that is in place between  forward/backward Kantorovich potentials and measures $(\mu_t)$:
\[
\begin{split}
\varphi_t+\psi_t&=0\qquad\text{on }\supp(\mu_t),\\
\varphi_t+\psi_t&\leq 0\qquad\text{on }\X,
\end{split}
\]
see e.g.\ Remark 7.37 in \cite{Villani09} (paying attention to the different sign convention). By direct computation one can verify that $(\varphi^\eps_t),(\psi^\eps_t)$ solve the Hamilton-Jacobi-Bellman equations
\begin{equation}
\label{eq:hj2}
\frac{\d}{\d t}\varphi^{\varepsilon}_t  =  \frac{1}{2}|\nabla\varphi^{\varepsilon}_t|^2 + \frac{\varepsilon}{2}\Delta\varphi^{\varepsilon}_t\qquad\qquad\qquad\qquad-\frac{\d}{\d t}\psi^{\varepsilon}_t  =  \frac{1}{2}|\nabla\psi^{\varepsilon}_t|^2 + \frac{\varepsilon}{2}\Delta\psi^{\varepsilon}_t,
\end{equation}
thus introducing the functions 
\[
\vartheta^\eps_t:=\frac{\psi^\eps_t-\varphi^\eps_t}2
\]
it is not hard to check that it holds
\begin{equation}
\label{eq:cei}
\frac\d{\d t}\rho^\eps_t+\rm{div}(\nabla\vartheta^\eps_t\,\rho^\eps_t)=0
\end{equation}
and
\[
\frac\d{\d t}\vartheta^\eps_t+\frac{|\nabla\vartheta^\eps_t|^2}2=a^\eps_t,\qquad\qquad\text{where}\qquad a^\eps_t:= -\frac{\varepsilon^2}{8}\Big(2\Delta\log\rho^{\varepsilon}_t + |\nabla\log\rho^{\varepsilon}_t|^2\Big).
\]
With this said, our main results about entropic interpolations can be summarized  as follows. Under the assumptions that the metric measure space is compact and $\RCD^*(K,N)$, $N<\infty$, and that  $\rho_0,\rho_1$ belong to $L^\infty(\X)$ we have:
\begin{itemize}
\item[-]\noindent\underline{Zeroth order} 
\begin{itemize}
\item[--]\emph{bound} For some $C>0$ we have $\rho^\eps_t\leq C\mm$ for every $\eps\in(0,1)$ and $t\in[0,1]$ (Proposition \ref{pro:6}).

\item[--]\emph{convergence} The curves  $(\rho^\eps_t\mm)$ $W_2$-uniformly converge to the unique $W_2$-geodesic $(\mu_t)$ from $\mu_0$ to $\mu_1$ (Propositions \ref{lem:6} and \ref{pro:9}).
\end{itemize}
\item[-]\noindent\underline{First order}
\begin{itemize}
\item[--]\emph{bound} For any $t\in(0,1]$ the functions $\{\varphi^\eps_t\}_{\eps\in(0,1)}$ are equi-Lipschitz (Proposition \ref{pro:1}). Similarly for the $\psi$'s. 
\item[--]
\emph{convergence}   For every sequence $\eps_n\downarrow0$ there is a subsequence - not relabeled - such that for any $t\in(0,1]$ the functions $\varphi^\eps_t$ converge both uniformly and in $W^{1,2}(\X)$ to a function $\varphi_t$ such that  $-t\varphi_t$ is a Kantorovich potential from $\mu_t$ to $\mu_0$ (see Propositions \ref{lem:6}, \ref{pro:9} and \ref{thm:10} for the precise formulation of the results). Similarly for the $\psi$'s. 
\end{itemize}
\item[-]\noindent\underline{Second order}  For every $\delta\in(0,1/2)$ we have
\begin{itemize}
\item[--]\emph{bound}
\begin{equation}
\label{eq:2i}
\begin{split}
&\sup_{\eps\in(0,1)}\iint_\delta^{1-\delta} \big(|\H{\vartheta^\eps_t}|_\HS^2+\eps^2|\H{\log\rho^\eps_t}|_\HS^2\big)\rho^\eps_t\,\d t\,\d\mm<\infty,\\
&\sup_{\eps\in(0,1)}\iint_\delta^{1-\delta} \big(|\Delta{\vartheta^\eps_t}|^2+\eps^2|\Delta{\log\rho^\eps_t}|^2\big)\rho^\eps_t\,\d t\,\d\mm<\infty,
\end{split}
\end{equation}
(Lemma \ref{lem:7}).  Notice that since in general the Laplacian is not the trace of the Hessian, there is no direct link between these two bounds.
\item[--]\emph{convergence} For every function $h\in W^{1,2}(\X)$ with $\Delta h\in L^\infty(\X)$ it holds
\begin{equation}
\label{eq:2ii}
\lim_{\eps\downarrow0}\iint_\delta^{1-\delta}\la\nabla h,\nabla a^\eps_t\ra \rho^\eps_t\,\d t\,\d\mm=0,
\end{equation}
(Theorem \ref{thm:main}).
\end{itemize}
\end{itemize}
\bigskip

With the exception of the convergence $\rho^\eps_t\mm\to \mu_t$, all these results are new even on compact smooth manifolds (in fact, even in the flat torus). The zeroth and first order bounds are both consequences of the Hamilton-Jacobi-Bellman equations \eqref{eq:hj2} satisfied by the $\varphi$'s and $\psi$'s and can be obtained from Hamilton's gradient estimate and the Li-Yau inequality. The facts that the limit curve is the $W_2$-geodesic and that the limit potentials are Kantorovich potentials are consequence of the fact that we can pass to the limit in the continuity equation \eqref{eq:cei} and that the limit potentials satisfy the Hamilton-Jacobi equation. In this regard it is key that we approximate  at the same time both the `forward' potentials $\psi$ and the `backward' one $\varphi$: see the proof of Proposition \ref{pro:9} and recall that the simple viscous approximation may converge to curves which are not $W_2$-geodesics.

 These zeroth and first order convergences are sufficient to pass to the limit in the term with the Hessian in \eqref{eq:secondsmooth2}.

As said, also the viscous approximation could produce the same kind of convergence. The crucial advantage of dealing with entropic interpolations is thus in the second order convergence result \eqref{eq:2ii} which shows that the term with the acceleration in \eqref{eq:secondsmooth2} vanishes in the limit and thus eventually allows us to prove our main result Theorem \ref{thm:2}. In this direction, we informally point out  that being the geodesic equation a second order one, in searching for an approximation procedure it is natural to look for one producing some sort of second order convergence.

The limiting property \eqref{eq:2ii} is mostly a consequence - although perhaps non-trivial - of the bound \eqref{eq:2i} (see in particular Lemma \ref{le:vanish} and the proof of Theorem \ref{thm:main}), thus let us focus on how to get \eqref{eq:2i}. The starting point here is a formula due to L\'eonard \cite{Leonard13}, who realized that there is a connection between entropic interpolation and lower Ricci bounds: he computed the  second order derivative of the entropy along entropic interpolations obtaining 
\begin{equation}
\label{eq:leo}
\frac{\d^2}{\d t^2}H(\rho^\eps_t\mm\,|\,\mm)=\int \big(\Gamma_2(\vartheta^\eps_t)+\tfrac{\eps^2}4\Gamma_2(\log\rho^\eps_t)\big)\rho^\eps_t\,\d\mm,
\end{equation}
where $\Gamma_2$ is the `iterated carr\'e du champ' operator defined as
\[
\Gamma_2(f):=\Delta\frac{|\nabla f|^2}2-\la\nabla f,\nabla\Delta f\ra
\]
(in the setting of $\RCD$ spaces some care is needed when handling this object, because $\Gamma_2(f)$ is in general only a measure, but let us neglect this issue here).

Thus if, say, we are on a manifold with non-negative Ricci curvature, then the Bochner inequality 
\begin{equation}
\label{eq:bint}
\Gamma_2(f)\geq |\H f|^2_\HS
\end{equation}
grants that the entropy is convex along entropic interpolations. 

Now notice that if $f:[0,1]\to\R^+$ is convex, then for $t\in(0,1)$ the quantity $|f'(t)|$ can be bounded in terms of $f(0),f(1)$ and $t$ only. Thus since the value of $H(\rho^\eps_t\mm\,|\,\mm)$ at $t=0,1$ is independent on $\eps>0$, we have the uniform bound
\[
\sup_{\eps>0}\int_\delta^{1-\delta}\frac{\d^2}{\d t^2}H(\mu^\eps_t\,|\,\mm)\,\d t=\sup_{\eps>0}\bigg(\frac{\d}{\d t}H(\mu^\eps_t\,|\,\mm)\restr{t=1-\delta}-\frac{\d}{\d t}H(\mu^\eps_t\,|\,\mm)\restr{t=\delta}\bigg)<\infty
\]
which by \eqref{eq:bint} and \eqref{eq:leo} grants the first in \eqref{eq:2i}. The second is obtained in a similar way using the Bochner inequality in the form
\[
\Gamma_2(f)\geq \frac{(\Delta f)^2}N
\]
in place of \eqref{eq:bint}.

\bigskip

\noindent{\bf Acknowledgements} 

The authors wish to thank C. L\'eonard for the numerous inspiring conversations about entropic interpolation.

 This research has been supported by the MIUR SIR-grant `Nonsmooth Differential Geometry' (RBSI147UG4). The second author is also grateful to the UFI/UIF for the financial support of the Vinci Programme.

\section{Preliminaries}\label{sec:2}
\subsection{Sobolev calculus on $\RCD$ spaces}
We shall assume the reader to be familiar with the language of optimal transport, metric measure geometry, the notion of $\RCD$ spaces and the differential calculus on them. Here we shall only recall those facts that we shall use in the sequel, mostly to fix the notation and provide bibliographical references.

\bigskip

By $C([0,1],(\X,\sfd))$, or simply $C([0,1],\X)$, we denote the space of continuous curves with values on the metric space $(\X,\sfd)$ and for $t\in[0,1]$ the {\bf evaluation map} $\e_t:C([0,1],(\X,\sfd))\to \X$ is defined as $\e_t(\gamma):=\gamma_t$.  For the notion of {\bf absolutely continuous curve} in a metric space and of {\bf metric speed} see for instance Section 1.1 in \cite{AmbrosioGigliSavare08}.  The collection of absolutely continuous curves on $[0,1]$ is denoted $AC([0,1],(\X,\sfd))$, or simply by $AC([0,1],\X)$.

By $\prob \X$ we denote the space of Borel probability measures on $(\X,\sfd)$ and by $\probt \X\subset \prob \X$ the subclass of those with finite second moment.

\medskip

Let $(\X,\sfd,\mm)$ be a complete and separable metric measure space endowed with a Borel non-negative measure which is finite on bounded sets.

For the definition of {\bf test plans}, of the {\bf Sobolev class} $S^2(\X)$ and of {\bf minimal weak upper gradient} $|D f|$ see \cite{AmbrosioGigliSavare11} (and the previous works \cite{Cheeger00}, \cite{Shanmugalingam00} for alternative - but equivalent - definitions of Sobolev functions).

The Banach space $W^{1,2}(\X)$ is defined as $L^2(\X)\cap S^2(\X)$ and endowed with the norm $\|f\|_{W^{1,2}}^2:=\|f\|_{L^2}^2+\||Df|\|_{L^2}^2$ and the  {\bf Cheeger energy} is the convex and lower-semicontinuous functional $E:L^2(\X)\to[0,\infty]$ given by
\[
E(f):=\left\{\begin{array}{ll}
\displaystyle{\frac12\int|D f|^2\,\d\mm}&\qquad \text{for }f\in W^{1,2}(\X)\\
+\infty&\qquad\text{otherwise}
\end{array}\right.
\]
$(\X,\sfd,\mm)$ is {\bf infinitesimally Hilbertian} (see \cite{Gigli12}) if $W^{1,2}(\X)$ is Hilbert. In this case $E$ is a Dirichlet form and its infinitesimal generator $\Delta$, which is a closed self-adjoint operator on $L^2(\X)$, is called {\bf Laplacian} on $(\X,\sfd,\mm)$ and its domain denoted by $D(\Delta)\subset W^{1,2}(\X)$. The flow $(\h_t)$ associated to $E$ is called {\bf heat flow} (see \cite{AmbrosioGigliSavare11}), and  for any $f\in L^2(\X)$ the curve $t\mapsto\h_tf\in L^2(\X)$ is continuous an $[0,\infty)$, locally absolutely continuous on $(0,\infty)$ and the only solution of
\[
\ddt\h_tf=\Delta\h_tf\qquad\h_tf\to f\text{ as }t\downarrow0.
\]
If moreover $(\X,\sfd,\mm)$ is an $\RCD(K,\infty)$ space (see \cite{AmbrosioGigliSavare11-2}) there exists the {\bf heat kernel}, namely a function 
\begin{equation}
\label{eq:hk}
(0,\infty)\times \X^2\ni (t,x,y)\quad\mapsto\quad \hr_t[x](y)=\hr_t[y](x)\in (0,\infty)
\end{equation}
such that
\begin{equation}
\label{eq:rapprform}
\h_tf(x)=\int f(y)\hr_t[x](y)\,\d\mm(y)\qquad\forall t>0
\end{equation}
for every $f\in L^2(\X)$. For every $x\in \X$ and $t>0$, $r_t[x]$ is a probability density and thus \eqref{eq:rapprform} can be used to extend the heat flow to $L^1(\X)$ and shows that the  flow is {\bf mass preserving} and satisfies the {\bf maximum principle}, i.e.
\begin{equation}
\label{eq:maxprinc}
f\leq c\quad\mm-a.e.\qquad\qquad\Rightarrow \qquad\qquad\h_tf\leq c\quad\mm\ae,\ \forall t>0.
\end{equation}
For compact and finite-dimensional $\RCD^*(K,N)$ spaces (\cite{Gigli12}, \cite{Erbar-Kuwada-Sturm13}, \cite{AmbrosioMondinoSavare13}), the fact that the measure is doubling and the space supports a weak 1-2 Poincar\'e inequality (\cite{Sturm06II}, \cite{Rajala12-2}) grants via the results in \cite{Sturm96III}, \cite{AmbrosioGigliSavare11-2}  that the heat kernel is continuous and satisfies {\bf Gaussian estimates}, i.e.\ there is $C_1=C_1(K,N,{\rm Diam}(\X))$ and for every $\delta>0$ another constant $C_2=C_2(K,N,{\rm Diam}(\X),\delta)$ such that for every $x,y\in \X$ and $t>0$ it holds
\begin{equation}
\label{eq:gaussest}
\frac{1}{C_1\mm(B_{\sqrt t}(y))}\exp\Big(-C_1\frac{\sfd^2(x,y)}{t}\Big)\leq \hr_t[x](y)\leq \frac{C_2}{\mm(B_{\sqrt t}(y))}\exp\Big(-\frac{\sfd^2(x,y)}{(4+\delta)t}\Big).
\end{equation}

\bigskip

For general metric measure spaces, the {\bf differential} is a well defined linear map $\d$ from $S^2(\X)$ with values in the {\bf cotangent module} $L^2(T^*\X)$ (see \cite{Gigli14}) which is a closed operator when seen as unbounded operator on $L^2(\X)$. It satisfies the following calculus rules which we shall use extensively without further notice:
\begin{align*}
|\d f|&=|D f|\quad\mm\ae&&\forall f\in S^2(\X)\\
\d f&=\d g\qquad\mm\ae\ \text{\rm on}\ \{f=g\},&&\forall f,g\in S^2(\X)\\
\d(\varphi\circ f)&=\varphi'\circ f\,\d f&&\forall f\in S^2(\X),\ \varphi:\R\to \R\ \text{Lipschitz}\\
\d(fg)&=g\,\d f+f\,\d g&&\forall f,g\in L^\infty\cap S^2(\X).
\end{align*}
where it is part of the properties the fact that $\varphi\circ f,fg\in S^2(\X)$ for $\varphi,f,g$ as above.

If $(\X,\sfd,\mm)$ is infinitesimally Hilbertian, which from now on we shall always assume, the cotangent module is canonically isomorphic to  its dual, the {\bf tangent module} $L^2(T\X)$, and the isomorphism sends the differential $\d f$ to the gradient $\nabla f$. Elements of $L^2(T\X)$ are called vector fields. The {\bf divergence} of a vector field is defined as (minus) the adjoint of the differential, i.e.\ we say that  $v$ has a divergence, and write $v\in D({\rm div})$, provided there is a function $g\in L^2(\X)$ such that
\[
\int fg\,\d\mm=-\int \d f(v)\,\d\mm\qquad\forall f\in W^{1,2}(\X).
\]
In this case $g$ is unique and is denoted ${\rm div}(v)$. The formula
\[
{\rm div}(fv)=\d f(v)+f{\rm div}(v)\qquad\forall f\in W^{1,2}(\X),\ v\in D({\rm div}),\ \text{such that}\ |f|,|v|\in L^\infty(\X)
\]
holds, where it is intended in particular that $fv\in D({\rm div})$ for $f,v$ as above. It can also be verified that
\[
f\in D(\Delta)\ \text{if and only if}\ \nabla f\in D({\rm div})\text{ and in this case }\Delta f={\rm div}(\nabla f),
\]
in accordance with the smooth case. It is now not hard to see that the formulas
\[
\begin{split}
\Delta(\varphi\circ f)&=\varphi''\circ f|\d f|^2+\varphi'\circ f\Delta f\\
\Delta(fg)&=g\Delta f+f\Delta g+2\la\nabla f,\nabla g\ra
\end{split}
\]
hold, where in the first equality we assume that $f\in D(\Delta),\varphi\in C^2(\R)$ are such that $f,|\d f|\in L^\infty(\X)$ and $\varphi',\varphi''\in L^\infty(\R)$ and in the second that $f,g\in D(\Delta)\cap L^\infty(\X)$ and $|\d f|,|\d g|\in L^\infty(\X)$ and it is part of the claims that $\varphi\circ f,fg$ are in $D(\Delta)$.

Beside this notion of $L^2$-valued Laplacian, we shall also need that of measure-valued Laplacian (\cite{Gigli12}). A function $f\in W^{1,2}(\X)$ is said to have measure-valued Laplacian, and in this case we write $f\in D(\bd)$, provided there exists a Borel (signed) measure $\mu$ whose total variation is finite on bounded sets and such that
\[
\int g\,\d\mu=-\int\la\nabla g,\nabla f\ra\,\d\mm,\qquad\forall g\text{ Lipschitz with bounded support}.
\]
In this case $\mu$ is unique and denoted $\bd f$. This notion is compatible with the previous one in the sense that
\[
f\in D(\bd),\ \bd f\ll\mm\text{ and }\frac{\d\bd f}{\d\mm}\in L^2(\mm)\qquad\Leftrightarrow\qquad f\in D(\Delta)\text{ and in this case }\Delta f=\frac{\d\bd f}{\d\mm}.
\]

\bigskip

On $\RCD(K,\infty)$ spaces, the vector space of `test functions'  (see \cite{Savare13}) is defined as
\[
\begin{split}
\testi{\X} &:= \Big\{ f \in D(\Delta) \cap L^{\infty}({\X}) \ :\ |\nabla f| \in L^{\infty}(\X),\ \Delta f \in L^{\infty}\cap W^{1,2}(\X) \Big\}.
\end{split}
\]
This is an algebra dense in $W^{1,2}(\X)$ and such that
\begin{equation}
\label{eq:compos}
\varphi\circ f\in\testi \X\quad\forall f\in\testi \X,\ \varphi:\R\to\R \text{ which is $C^\infty$ on the image of $f$}
\end{equation}
(see \cite{Savare13}).   We shall also make use of the set
\[
\begin{split}
\testipp \X&:= \Big\{ f \in \testi \X\ :\  f\geq c\ \mm\ae \text{ for some }c>0 \Big\}.
\end{split}
\]
Combining  the Gaussian estimates  on compact $\RCD^*(K,N)$ spaces, $N<\infty$, with the results in \cite{Savare13} we see that
\begin{equation}
\label{eq:regflow}
\begin{split}
f\in L^1(\X),\  t>0\qquad&\Rightarrow\qquad \h_t(f)\in\testi \X,\\
f\in L^1(\X),\ f\geq 0,\ \int f\,\d\mm>0,\ t>0\qquad&\Rightarrow\qquad \h_t(f)\in\testipp \X.
\end{split}
\end{equation}
The fact that $\testi\X$ is an algebra is based on the property
\begin{equation}
\label{eq:reggrad}
\begin{split}
f\in\testi \X\qquad\Rightarrow\qquad &|\d f|^2\in W^{1,2}(\X)\quad\text{ with }\\
&\int|\d(|\d f|^2)|^2\,\d\mm\leq \||\d f|\|^2_{L^\infty}\Big(\||\d f|\|_{L^2}\||\d \Delta f|\|_{L^2}+|K|\||\d f|\|_{L^2}^2\Big)
\end{split}
\end{equation}
and actually  a further regularity property of test functions is that
\begin{equation}
\label{eq:lapmistest}
f\in\testi \X\qquad\Rightarrow\qquad |\d f|^2\in D(\bd),
\end{equation}
so that it is possible to introduce the {\bf measure-valued $\Gamma_2$ operator} (\cite{Savare13}) as
\[
\Ggamma_2(f):=\bd\frac{|\d f|^2}{2}-\la\nabla f,\nabla\Delta f\ra\mm\qquad\forall f\in\testi \X.
\]
By construction, the assignment $f\mapsto \Ggamma_2(f)$ is a quadratic form.

An important property of the heat flow on $\RCD(K,\infty)$ spaces is the {\bf Bakry-\'Emery contraction estimate} (see \cite{AmbrosioGigliSavare11-2}):
\begin{equation}
\label{eq:be}
|\d\h_tf|^2\leq e^{-2Kt}\h_t(|\d f|^2)\qquad\forall f\in W^{1,2}(\X),\ t\geq 0.
\end{equation}

We also recall that $\RCD(K,\infty)$ spaces have the {\bf Sobolev-to-Lipschitz} property (\cite{AmbrosioGigliSavare11-2}, \cite{Gigli13}), i.e.
\begin{equation}
\label{eq:sobtolip}
f\in W^{1,2}(\X),\ |\d f|\in L^\infty(\X)\qquad\Rightarrow\qquad \exists \tilde f=f\ \mm-a.e.\ \text{ with }\Lip(\tilde f)\leq\||\d f|\|_{L^\infty},
\end{equation}
and thus we shall typically identify Sobolev functions with bounded differentials with their Lipschitz representative; in particular this will be the case for functions in $\testi \X$.

\bigskip

The existence of the space of test functions and the language of $L^2$-normed $L^\infty$-modules allow to introduce the space $W^{2,2}(\X)$ as follows (see \cite{Gigli14}). We first consider the  tensor product $L^2((T^*)^{\otimes 2}\X)$ of $L^2(T^*\X)$ with itself. The pointwise norm on such module is denoted $|\cdot|_\HS$ to remind that in the smooth case it coincides with the Hilbert-Schmidt one. Then we say that a function $f\in W^{1,2}(\X) $ belongs to $W^{2,2}(\X)$ provided there exists $A\in L^2((T^*)^{\otimes 2}\X)$ symmetric, i.e.\ such that $A(v_1,v_2)=A(v_2,v_1)$ $\mm$-a.e.\ for every $v_1,v_2\in L^2(T\X)$, for which it holds
\[
\int h A(\nabla g,\nabla g)\,\d\mm=\int-\la\nabla f,\nabla g\ra{\rm div}(h\nabla g)-h\langle\nabla f,\nabla\frac{|\nabla g|^2}2\rangle\,\d\mm \qquad\forall g,h\in\testi \X.
\]
In this case $A$ is unique, called {\bf Hessian} of $f$ and denoted by $\H f$. The space $W^{2,2}(\X)$ endowed with the norm
\[
\|f\|^2_{W^{2,2}(\X)}:=\|f\|^2_{L^2(\X)}+\|\d f\|^2_{L^2(T^*\X)}+\|\H f\|^2_{L^2((T^*)^{\otimes 2}\X)}
\]
is a complete separable Hilbert space which contains $\testi \X$ and in particular is dense in $W^{1,2}(\X)$. It is proved in \cite{Gigli14} that $D(\Delta)\subset W^{2,2}(\X)$ with
\begin{equation}
\label{eq:heslap}
\int|\H f|_{\HS}^2\,\d\mm\leq \int (\Delta f)^2-K|\nabla f|^2\,\d\mm\qquad\forall f\in D(\Delta).
\end{equation}
The space $H^{2,2}(\X)$ is defined as the closure of $D(\Delta)$ in $W^{2,2}(\X)$; it is unknown whether it coincides with $W^{2,2}(\X)$ or not. 

We shall need the following form of Leibniz rule (\cite{Gigli14}):
\begin{equation}
\label{eq:leibh}
\d\la\nabla f,\nabla g\ra=\H f(\nabla g,\cdot)+\H g(\nabla f,\cdot)\qquad\forall f,g\in\testi \X.
\end{equation}

The {\bf Bochner inequality} on $\RCD(K,\infty)$ spaces takes the form of an inequality between measures (\cite{Gigli14} - see also the previous contributions \cite{Savare13}, \cite{Sturm14}):
\begin{equation}
\label{eq:bochhess}
\Ggamma_2(f)\geq \big(|\H f|^2_\HS+K|\d f|^2\big)\mm\qquad\forall f\in\testi \X,
\end{equation}
and if the space is $\RCD^*(K,N)$ for some finite $N$ it also holds (\cite{Erbar-Kuwada-Sturm13}, \cite{AmbrosioMondinoSavare13}):
\begin{equation}
\label{eq:bochlap}
\Ggamma_2(f)\geq\Big( \frac{(\Delta f)^2}N+K|\d f|^2\Big)\mm\qquad\forall f\in\testi \X.
\end{equation}
Notice that since the Laplacian is in general not the trace of the Hessian, the former does not trivially imply the latter (in connection to this, see \cite{Han14}).

\bigskip

We conclude the section recalling the notion of Regular Lagrangian Flow, introduced by Ambrosio-Trevisan in \cite{Ambrosio-Trevisan14} as the generalization to $\RCD$ spaces of the analogous concept existing on $\R^d$ as proposed by Ambrosio in \cite{Ambrosio04}:

\begin{Definition}[Regular Lagrangian Flow]
Given $(v_t)\in L^1([0,1],L^2(T\X))$, the function $F:[0,1]\times \X\to \X$ is  a Regular Lagrangian Flow for $(v_t)$ provided:
\begin{itemize}
\item[i)] $[0,1]\ni t \mapsto F_t(x)$ is continuous for every $x\in \X$
\item[ii)] for every $f\in\testi \X$ and $\mm$-a.e.\ $x$ the map $t\mapsto f(F_t(x))$ belongs to $W^{1,1}([0,1])$ and
\[
\ddt  f(F_t(x))=\d f(v_t)(F_t(x))\qquad {\rm a.e.}\ t\in[0,1].
\]
\item[iii)] it holds
\[
(F_t)_*\mm\leq C\mm\qquad\forall t\in[0,1]
\]
for some constant $C>0$.
\end{itemize}
\end{Definition}
In \cite{Ambrosio-Trevisan14} the authors prove that under suitable assumptions on the $v_t$'s, Regular Lagrangian Flows exist and are unique. We shall use the following formulation of their result (weaker than the one provided in \cite{Ambrosio-Trevisan14}):
\begin{Theorem}\label{thm:RLF}
Let $(\X,\sfd,\mm)$ be a $\RCD(K,\infty)$ space and $(\varphi_t)\in L^1([0,1],W^{1,2}(\X))$ be such that $\varphi_t\in D(\Delta)$ for a.e.\ $t$ and
\[
\Delta\varphi_t\in L^1([0,1],L^2(\X))\qquad (\Delta\varphi_t)^-\in L^1([0,1],L^\infty(\X)).
\]
Then there exists a unique, up to $\mm$-a.e.\ equality, Regular Lagrangian Flow $F$ for $(\nabla\varphi_t)$.

For such flow, the quantitative bound
\begin{equation}
\label{eq:quantm}
(F_t)_*\mm\leq \exp\Big(\int_0^1\|(\Delta\varphi_t)^-\|_{L^\infty(\X)}\,\d t\Big)\mm
\end{equation}
holds for every $t\in[0,1]$ and for $\mm$-a.e.\ $x$ the curve $t\mapsto F_t(x)$ is absolutely continuous and its metric speed ${\rm ms}_t({F_{\cdot}}(x))$ at time $t$ satisfies
\begin{equation}
\label{eq:quants}
{\rm ms}_t({F_{\cdot}}(x))=|\nabla\varphi_t|(F_t(x))\qquad {\rm a.e.}\ t\in[0,1].
\end{equation}
\end{Theorem}
To be precise,  \eqref{eq:quants} is not explicitly stated in \cite{Ambrosio-Trevisan14}; its proof is anyway not hard and can be obtained, for instance, following the arguments in \cite{Gigli14}.

\subsection{Optimal transport on $\RCD$ spaces}
It is well known that on $\R^d$, curves of measures which are $W_2$-absolutely continuous are in correspondence with appropriate solutions of the {\bf continuity equation} (\cite{AmbrosioGigliSavare08}). It has been proved in \cite{GigliHan13} that the same connection holds on arbitrary metric measure spaces $(\X,\sfd,\mm)$, provided the measures are controlled by $C\mm$ for some $C>0$, the formulation of such result which we shall need is:
\begin{Theorem}[Continuity equation and $W_2$-AC\ curves]\label{thm:GH} Let $(\X,\sfd,\mm)$ be infinitesimally Hilbertian,   $(\mu_t)\subset \prob \X$ be weakly continuous and $t\mapsto\phi_t\in W^{1,2}(\X)$ be Borel, possibly defined only for a.e.\ $t\in[0,1]$. Assume that:
\begin{subequations}
\begin{align}
\label{eq:bih}
\mu_t&\leq C\mm\qquad\text{ $\forall t\in[0,1]$ for some $C>0$}\\
\label{eq:bkh}
\int_0^1\int |\nabla\phi_t|^2\,\d\mu_t\,\d t&<\infty
\end{align}
\end{subequations}
and that  the continuity equation
\[
\ddt\mu_t+{\rm div}(\nabla\phi_t\mu_t)=0,
\]
is satisfied in the following sense: for any $f\in W^{1,2}(\X)$ the map $[0,1]\ni t\mapsto\int f\,\d\mu_t$ is absolutely continuous and it holds
\[
\ddt\int f\,\d\mu_t=\int \d f(\nabla\phi_t)\,\d\mu_t\qquad {\rm a.e.}\ t.
\]
Then $(\mu_t)\in AC([0,1],(\prob \X,W_2))$ and
\[
|\dot\mu_t|^2=\int|\nabla\phi_t|^2\,\d\mu_t\qquad{\rm a.e.}\ t\in[0,1].
\]
\end{Theorem}
Recall that given $f:\X\to\R$ the  upper and lower slopes $|D^+f|,|D^-f|:\X\to[0,\infty]$ are defined as 0 on isolated points and otherwise
\[
|D^+f|(x):=\lims_{y\to x}\frac{(f(y)-f(x))^+}{\sfd(x,y)}\qquad\qquad|D^-f|(x):=\lims_{y\to x}\frac{(f(y)-f(x))^-}{\sfd(x,y)}.
\]
Similarly, the {\bf local Lipschitz constant} $\lip(f):\X\to[0,\infty]$ is defined as 0 on isolated points and otherwise as
\[
\lip f(x) :=\max\{|D^+f|(x),|D^-f|(x)\}= \limsup_{y \to x}\frac{|f(x) - f(y)|}{d(x,y)}.
\]
We also recall that the $c$-transform $\varphi^c:X\to\R\cup\{-\infty\}$ of  a function $\varphi:\X\to\R\cup\{-\infty\}$ is defined as
\[
\varphi^c(x):=\inf_{y\in X}\frac{\sfd^2(x,y)}{2}-\varphi(y)
\]
and that $\varphi$ is said to be {\bf $c$-concave} provided $\varphi=\psi^c$ for some $\psi$. Also, given $\mu_0,\mu_1\in\probt X$, a function $\varphi:\X\to\R\cup\{-\infty\}$ is called {\bf Kantorovich potential} from $\mu_0$ to $\mu_1$ provided it is $c$-concave and
\[
\int\varphi\,\d\mu_0+\int\varphi^c\,\d\mu_1=\frac12W_2^2(\mu_0,\mu_1).
\]
It is worth recalling that on general complete and separable metric spaces $(X,\sfd)$ we have that for $\mu_0,\mu_1\in\prob \X$ with bounded support there exists a Kantorovich potential from $\mu_0$ to $\mu_1$ which is Lipschitz and bounded.

This can be obtained starting from an arbitrary Kantorovich potential $\psi$ and then defining
\[
\varphi(x):=\min\Big\{C,\inf_{y\in X}\frac{\sfd^2(x,y)}2-\psi^c(y)\Big\}
\]
for $C$ sufficiently big.

With this said, we recall the following version of   Brenier-McCann theorem on $\RCD$ spaces ($(i)$ comes from \cite{Gigli12a} and \cite{RajalaSturm12}, $(ii)$ from \cite{AmbrosioGigliSavare11-2} and \cite{Gigli12}, $(iii)$ from \cite{AmbrosioGigliSavare11} and  $(iv)$ from \cite{GigliRajalaSturm13}).
\begin{Theorem}\label{thm:bm} Let $(\X,\sfd,\mm)$ be a $\RCD(K,\infty)$ space and $\mu_0,\mu_1\in\probt \X$ with bounded support and such that $\mu_0,\mu_1\leq C\mm$ for some $C>0$.  Also, let $\varphi$ be a  Kantorovich potential for the couple $(\mu_0,\mu_1)$ which is locally Lipschitz on a neighbourhood of $\supp(\mu_0)$. Then:
\begin{itemize}
\item[i)] There exists a unique geodesic $(\mu_t)$ from $\mu_0$ to $\mu_1$, it satifies
\begin{equation}
\label{eq:linftyrcd}
\mu_t\leq C'\mm\qquad\forall t\in[0,1]\text{ for some }C'>0
\end{equation}
and there is a unique \emph{lifting} $\ppi$ of it, i.e.\ a unique measure $\ppi\in\prob{C([0,1],X)}$ such that $(\e_t)_*\ppi=\mu_t$ for every $t\in[0,1]$ and $\iint_0^1|\dot\gamma_t|^2\,\d t\,\d\ppi(\gamma)=W_2^2(\mu_0,\mu_1)$.

\item[ii)] For every $f\in W^{1,2}(\X)$ the map $t\mapsto \int f\,\d\mu_t$ is differentiable at $t=0$ and
\[
\ddt\int f\,\d\mu_t\restr{t=0}=-\int \d f(\nabla\varphi)\,\d\mu_0.
\]
\item[iii)] The identity 
\[
|\d\varphi|(\gamma_0)=|D^+\varphi|(\gamma_0)=\sfd(\gamma_0,\gamma_1)
\]
holds for $\ppi$-a.e.\ $\gamma$.
\item[iv)] If the space is $\RCD^*(K,N)$ for some $N<\infty$, then $(i),(ii),(iii)$ holds with $\mu_1$ only assumed to be with bounded support, with the caveat that \eqref{eq:linftyrcd} holds in the form: for every $\delta\in(0,1/2)$ there is $C_\delta>0$ so that  $\mu_t\leq C'_\delta\mm$ for every $t\in[0,1-\delta]$.
\end{itemize}
\end{Theorem}
A property related to the above is the fact that although the Kantorovich potentials are not uniquely determined by the initial and final measures, their gradients are. This is expressed by the following result, which also says that if we sit in the intermediate point of a geodesic and move to one extreme or the other, then the two corresponding velocities are one the opposite of the other (see Lemma 5.8 and Lemma 5.9 in \cite{Gigli13} for the proof):
\begin{Lemma}\label{le:gradpot}
Let $(\X,\sfd,\mathfrak{m})$ be a  $\RCD(K,\infty)$ space with $K \in \mathbb{R}$ and  $(\mu_t)\subset \probt \X$ a $W_2$-geodesic such that $\mu_t\leq C\mm$ for every $t\in[0,1]$ for some $C>0$. For $t\in[0,1]$ let $\phi_t,\phi_t':\X\to\R$ be locally Lipschitz functions such that for some $s,s'\neq t$ the functions $-(s-t)\phi_t$ and $-(s'-t)\phi_t'$ are Kantorovich potentials from $\mu_t$ to $\mu_s$ and from $\mu_t$ to $\mu_{s'}$ respectively.

Then
\[
\nabla\phi_t=\nabla\phi_{t'}\qquad\mu_t\ae.
\]
\end{Lemma}
On $\RCD$ spaces, $W_2$-geodesics  made of measures with bounded density also have the weak continuity property of the densities expressed by the following lemma. The proof  follows by a simple argument involving Young's measures and the continuity of the entropy along a geodesic (see Corollary 5.7 in \cite{Gigli13}):
\begin{Lemma}\label{lem:8}
Let $(\X,\sfd,\mathfrak{m})$ be a  $\RCD(K,\infty)$ space with $K \in \mathbb{R}$ and   $(\mu_t)\subset \probt X$ a $W_2$-geodesic such that $\mu_t\leq C\mm$ for every $t\in[0,1]$ for some $C>0$. Let $\rho_t$ be the density of $\mu_t$.

Then for any $t \in [0,1]$ and any sequence $(t_n)_{n \in \N} \subset [0,1]$ converging to $t$ there exists a subsequence $(t_{n_k})_{k \in \N}$ such that
\[
\rho_{t_{n_k}} \to \rho_t, \quad \mm\ae
\]
as $k \to \infty$.
\end{Lemma}
We conclude recalling some properties of the {\bf Hopf-Lax semigroup} in metric spaces, also in connection with optimal transport. For $f:X\to\R$ and $t>0$ the function $Q_tf:X\to\R$ is defined as
\begin{equation}
\label{eq:hli}
Q_tf(x):=\inf_{y\in \X}\frac{\sfd^2(x,y)}{2t}+f(y).
\end{equation}
Then we have the following result (\cite{AmbrosioGigliSavare11} - see also \cite{AmbrosioGigliSavare-compact}):
\begin{Proposition}\label{pro:11}
Let $(\X,\sfd)$ be a compact geodesic metric  space and $f:\X\to\R$ Lipschitz. Then the map  
$[0,\infty)\ni t\mapsto Q_tf\in C(\X)$ is Lipschitz w.r.t.\ the $\sup$ norm and for every $x\in \X$  we have
\begin{equation}\label{eq:55}
\ddt Q_tf(x) + \frac{1}{2}\Big(\lip Q_tf(x)\Big)^2 = 0\qquad {\rm a.e.}\ t>0.
\end{equation}
\end{Proposition}

\section{The Schr\"odinger problem}\label{sec:3}

Let $(\X,\tau)$ be a Polish space, $\mu_0,\mu_1\in\prob \X$ and $\hR\in\prob{\X^2}$ be given measures. Recall that $\ggamma\in\prob{\X^2}$ is called  transport plan  for $\mu_0,\mu_1$ provided $\pi^0_*\ggamma=\mu_0$ and $\pi^1_*\ggamma=\mu_1$, where $\pi^0,\pi^1:\X^2\to \X$ are the canonical projections. We are interested in finding  a transport plan of the form
\[
\ggamma=f\otimes g\,\hR
\]
for certain Borel functions $f,g:\X\to[0,\infty)$, where $f\otimes g(x,y):=f(x)g(y)$. As we shall see in this short section, in great generality this problem can be solved in a unique way and the plan $\ggamma$ can be found as the minimum of
\[
\ggamma'\quad\mapsto\quad H(\ggamma'\,|\, \hR)
\]
among all transport plans from $\mu_0$ to $\mu_1$, where $H(\cdot\,|\,\cdot)$ is the Boltzmann-Shannon entropy defined as
\[
H(\sigma\,|\,\nu):=\left\{\begin{array}{ll}
\displaystyle{\int\rho\log(\rho)\,\d\nu}&\qquad\text{ if }\sigma=\rho\nu,\\
+\infty&\qquad\text{ if }\sigma\not\ll\nu.
\end{array}\right.
\]
For appropriate choice of the reference measure $\hR$ (which will also be our choice in the following), this minimization problem is called Schr\"odinger problem, we refer to \cite{Leonard14} for a survey on the topic.

The following proposition collects the basic properties of the minimizer of the Schr\"odinger problem; points $(i)$ and $(ii)$ of the statement are already  known in the literature on the subject (see in particular \cite{Leonard01}, \cite{BLN94} and \cite{RuscThom93}), but for completeness we give the full proofs.

\begin{Proposition}\label{pro:2}
Let $(\X,\tau,\mm)$ be a Polish space equipped with a probability measure and $\hR \in \mathscr{P}(\X^2)$ be such that 
\[
\mm\otimes\mm\ll \hR\ll\mm\otimes\mm\qquad\text{ and }\qquad H(\mm \otimes \mm \,|\, \hR)<\infty. 
\]
Let $\mu_0=\rho_0\mm$ and $\mu_1=\rho_1\mm$ be Borel probability measures with bounded densities.

Then:
\begin{itemize}
\item[i)] There exists a unique minimizer $\ggamma$ of $H(\cdot\,|\,\hR)$ among all transport plans from $\mu_0$ to $\mu_1$.
\item[ii)] $\ggamma=f\otimes g\hR$ for appropriate Borel functions $f,g:\X\to[0,\infty)$ which are  unique up to the trivial transformation $(f,g)\to (cf,g/c)$ for some $c>0$.
\item[iii)] Assume in addition that
\begin{equation}
\label{eq:k8}
c\mm\otimes\mm\leq \hR\leq C\mm\otimes\mm
\end{equation}
for suitable $c,C>0$. Then $f,g\in L^\infty(\X,\mm)$ and $\ggamma$ is the only transport plan which can be written as $f'\otimes g' \hR$ for Borel $f',g':\X\to[0,\infty)$.
\end{itemize}
\end{Proposition}
\begin{proof}\\
\noindent{\bf{(i)}} Existence follows by the direct method of calculus of variations: the class of transport plans is  not empty and narrowly compact and $H(\,\cdot\,|\,\hR)$ is narrowly lower semicontinuous (see e.g.\ \cite{AmbrosioGigliSavare08}). Since $H(\,\cdot\,|\,\hR)$ is strictly convex, uniqueness will follow if we show that there is a transport plan with finite entropy. We consider  $\mu_0\otimes\mu_1$ and notice that  by direct computation we have
\[
\begin{split}
H(\mu_0\otimes\mu_1 \,|\, \hR) & =  H(\mu_0 \,|\,\mm) + H(\mu_1 \,|\, \mm) + \int \log\bigg(\frac{\d(\mm \otimes \mm)}{\d \hR}\bigg)\rho_0\otimes \rho_1\d (\mm \otimes \mm)  \\ 
& \leq H(\mu_0 \,|\, \mm) + H(\mu_1 \,|\, \mm) + \|\rho_0\|_{L^\infty}\|\rho_1\|_{L^\infty}\int \frac{\d(\mm \otimes \mm)}{\d \hR}\Big|\log\bigg(\frac{\d(\mm \otimes \mm)}{\d \hR}\bigg)\Big|\d \hR
\end{split}
\]
and using the trivial fact that $ z|\log(z)|=z\log(z)+2 (z\log(z))^-\leq z\log(z)+2e^{-1}$ valid for any $z\geq 0$ we conclude that
\[
H(\mu_0\otimes\mu_1 \,|\, \hR)\leq H(\mu_0 \,|\, \mm) + H(\mu_1 \,|\, \mm) +\|\rho_0\|_{L^\infty}\|\rho_1\|_{L^\infty}\big(H(\mm\otimes\mm\,|\, \hR)+2e^{-1} \big)
\] 
and our assumptions grant that the right hand side is finite. 

\medskip

\noindent{\bf{(ii)}} The uniqueness part of the claim is trivial, so we concentrate on existence. Finiteness of the entropy in particular grants that $\ggamma\ll \hR$. Put $p:=\frac{\d\sggamma}{\d \hR}$ and let $P_0:=\{\rho_0>0\}$, $P_1:=\{\rho_1>0\}$. We start claiming that 
\begin{equation}
\label{eq:denspos}
p>0\quad \mm\otimes\mm\aeon P_0\times P_1.
\end{equation}
Notice that since $\mm\otimes\mm$ and $\hR$ are mutually absolutely continuous, the claim makes sense and arguing by contradiction we shall assume that $\hR(Z)>0$, where  $Z:=(P_0\times P_1)\cap \{p=0\}$.

Let $s:=\frac{\d(\mu_0\otimes\mu_1)}{\d \hR }$ and  for $\lambda\in(0,1)$ let us define $\Phi(\lambda):\X^2\to\R$ by
\[
\Phi(\lambda) := \frac{u(p + \lambda(s-p)) - u(p)}{\lambda},\qquad\text{where}\quad u(z):=z\log(z).
\]
The convexity of $u$ grants that $\Phi(\lambda)\leq u(s) - u(p)\in L^1(\X^2,\hR)$ (recall that we proved that $H(\mu_0\otimes\mu_1|\hR)<\infty$) and   that  $\Phi(\lambda)$ is monotone decreasing as $\lambda \downarrow 0$. Moreover, on  $Z$ we have $\Phi(\lambda) \downarrow -\infty$ $\hR$-a.e.\ as $\lambda \downarrow 0$, thus  the monotone convergence theorem ensures that
\[
\lim_{\lambda \downarrow 0} \frac{H(\ggamma + \lambda(\mu_0\otimes\mu_1 - \ggamma) \,|\, \hR) - H(\ggamma \,|\, \hR)}{\lambda} = -\infty.
\]
Since $\ggamma + \lambda(\mu_0\otimes\mu_1 - \ggamma)$ is a transport plan from $\mu_0$ to $\mu_1$ for $\lambda\in(0,1)$, this is in contradiction with the minimality of $\ggamma$ which grants that the left hand side is non-negative, hence $Z$ is $\hR$-negligible, as desired.

Let us now pick $h\in L^\infty(\X^2,\ggamma)$  such that $\pi^0_*( h\ggamma)=\pi^1_*(h\ggamma)=0$ and  $\eps\in(0, \|h\|^{-1}_{L^\infty(\sggamma)})$. Then $(1+\eps h)\ggamma$ is a transport plan from $\mu_0$ to $\mu_1$ and noticing that $hp$ is well defined $\hR$-a.e.\  we have
\[
\begin{split}
\|u((1+\eps h)p)\|_{L^1(\hR)}& =  \int| (1 + \varepsilon h)p\log((1 + \varepsilon h)p)|\,\d \hR  \\ 
& \leq  \int (1 + \varepsilon h)p\,|\log p|\,\d \hR + \int (1 + \varepsilon h)\,|\log(1 + \varepsilon h)|\,\d \ggamma \\ 
& \leq  \|1 + \varepsilon h\|_{L^{\infty}(\sggamma)} \|p\log p\|_{L^1(\hR)} + \|(1 + \varepsilon h)\log(1 + \varepsilon h)\|_{L^{\infty}(\sggamma)} ,
\end{split}
\]
so that $u((1+\eps h)p)\in L^1(\hR)$. Then again by the monotone convergence theorem we get
\[
\lim_{\eps\downarrow 0}\frac{H((1 + \varepsilon h)\ggamma \,|\, \hR)-H(\ggamma \,|\, \hR)}{\eps}=\int \lim_{{\eps\downarrow 0}}\frac{u((1+\eps h)p)-u(p)}\eps\,\d \hR=\int hp(\log p+1)\,\d \hR.
\] 
By the minimality of $\ggamma$ we know that the left hand side in this last identity is non-negative, thus after running the same computation with $-h$ in place of $h$ and noticing that the choice of $h$ grants that $\int hp\,\d \hR=\int h\,\d\ggamma=0$ we obtain
\begin{equation}
\label{eq:kyoto1}
\int hp\log(p)\,\d \hR=0\qquad\forall h\in L^\infty(\ggamma)\ \text{such that}\ \pi^0_*( h\ggamma)=\pi^1_*(h\ggamma)=0.
\end{equation}
The rest of the argument is better understood by introducing the spaces $V,{}^\perp W\subset L^1(\ggamma)$ and $V^\perp,W\subset L^\infty(\ggamma)$ as follows
\[
\begin{split}
V&:=\big\{f\in L^1(\ggamma)\ :\ f=\varphi\oplus \psi\text{ for some }\varphi\in L^0(\mm\restr{P_0}),\ \psi\in L^0(\mm\restr{P_1})\big\},\\
W&:=\big\{h\in L^\infty(\ggamma)\ :\ \pi^0_*( h\ggamma)=\pi^1_*(h\ggamma)=0\big\},\\
V^\perp&:=\big\{h\in L^\infty(\ggamma)\ :\ \int fh\,\d\ggamma=0\ \forall f\in V\big\},\\
{}^\perp W&:=\big\{f\in L^1(\ggamma):\int fh\,\d\ggamma=0\ \forall h\in W\big\},
\end{split}
\]
where here and in the following the function $\varphi\oplus\psi$ is defined as $\varphi\oplus\psi(x,y):=\varphi(x)+\psi(y)$. Notice that the Euler equation \eqref{eq:kyoto1} reads as $\log(p)\in {}^\perp W$ and our thesis as $\log(p)\in V$; hence to conclude it is sufficient to show that ${}^\perp W\subset V$.

\noindent{\underline{Claim 1}: $V$ is a closed subspace of $L^1(\ggamma)$}.

We start claiming that  $f\in V$ if and only if $f\in L^1(\ggamma)$ and
\begin{equation}
\label{eq:kyoto2}
f(x,y)+f(x',y')=f(x,y')+f(x',y)\quad \mm\otimes\mm\otimes\mm\otimes\mm\ae\ (x,x',y,y')\in P_0^2\times P_1^2.
\end{equation}
Indeed the `only if' follows trivially from $\ggamma\ll\mm\otimes\mm$ and the definition of $V$. For the `if' we apply Fubini's theorem to get the existence of $x'\in P_0$ and $y'\in P_1$ such that
\[
f(x,y)+f(x',y')=f(x,y')+f(x',y)\quad \mm\otimes\mm\ae\ x,y\in P_0\times P_1.
\]
Thus  $f=f(\cdot,y')\oplus(f(x',\cdot)-f(x',y'))$, as desired. 

Now notice that since that \eqref{eq:denspos} grants that   $(\mm\times\mm)\restr{P_0\times P_1}\ll\ggamma$, we see that the condition \eqref{eq:kyoto2} is closed w.r.t.\ $L^1(\ggamma)$-convergence.

\noindent{\underline{Claim 2}: $V^\perp\subset W$}.

Let $h\in L^\infty(\ggamma)\setminus W$, so that either the first or second marginal of $h\ggamma$ is non-zero. Say the first. Thus since $\pi^0_*\ggamma=\mu_0$ we have $\pi^0_*(h\ggamma)=f_0\mu_0$ for some $f_0\in L^\infty(\mu_0)\setminus\{0\}$. Then the function $f:=f_0\oplus 0=f_0\circ\pi^0$ belongs to $V$ and we have
\[
\int hf\,\d\ggamma=\int f_0\circ\pi^0\,\d(h\ggamma)=\int f_0\,\d\pi^0_*(h\ggamma)=\int f_0^2\,\d\mu_0>0,
\]
so that $h\notin V^\perp$.

\noindent{\underline{Claim 3}: ${}^\perp W\subset V$}.

Let $f\in L^1(\ggamma)\setminus V$, use the fact that  $V$ is closed and  the Hahn-Banach theorem to find $h\in L^\infty(\ggamma)\sim L^1(\ggamma)^*$ such that $\int fh\,\d\ggamma\neq 0$ and $\int \tilde fh\,\d\ggamma=0$ for every $\tilde f\in V$. Thus $h\in V^\perp$ and hence by the previous step $h\in W$. The fact that $\int fh\,\d\ggamma\neq 0$ shows that $f\notin {}^\perp W$, as desired.

\medskip

\noindent{\bf (iii)} Let $\sigma$ be a transport plan from $\mu_0$ to $\mu_1$ such that  $\sigma = f' \otimes g' \hR$ for suitable non-negative Borel functions $f',g'$. We claim that in this case it holds $f',g'\in L^\infty(\mm)$, leading in particular to the claim in the statement about $\ggamma$.

By disintegrating $\hR$ w.r.t.\ $\pi^0$, from $\pi^0_*(f'\otimes g'\hR)=\rho_0\mm$ we get that
\begin{equation}\label{eq:44}
f'(x)\int g'(y)\,\d \hR_x(y) = \rho_0(x) < +\infty, \quad \textrm{for }\mm\ae\ x 
\end{equation}
whence $g' \in L^1(\hR_x)$ for $\mm$-a.e.\ $x \in P_0$. Since from \eqref{eq:k8} we have that $\hR_x\geq c\mm$ for $\mm$-a.e.\ $x$, we see that $g'\in L^1(\mm)$ with 
\[
c \|g'\|_{L^1(\mm)}\leq \int g'(y)\,\d \hR_x( y)\quad \textrm{for }\mm\ae\ x 
\]
and thus  \eqref{eq:44} yields
\[
f' \leq \frac{\|\rho_0\|_{L^{\infty}(\mm)}}{c\|g'\|_{L^1(\mm)}}, \quad \mm\ae,
\]
which is the desired $L^\infty$ bound on  $f'$. By interchanging the roles of $f'$ and $g'$, the same conclusion follows for $g'$.

For the uniqueness of $\ggamma$, put $\varphi:=\log f'$, $\psi:=\log g'$ and notice that, by what we have just proved, they are bounded from above. Therefore from
\[
0\leq H(\sigma\,|\,\hR)=\int \varphi\oplus\psi \,\d\sigma
\]
we infer that
\begin{equation}
\label{eq:kyoto3}
\varphi\circ\pi^0,\psi\circ\pi^1\in L^1(\sigma).
\end{equation}
Putting for brevity $p':=f'\otimes g'$ and  arguing as before to justify the passage to the limit inside the integral we  get
\[
\begin{split}
\frac{\d}{\d\lambda} H\big((1-\lambda)\sigma+\lambda\ggamma\,|\,\hR\big)\restr{\lambda=0^+}&=\int(p-p')\log(p')\,\d \hR \\
&=\int\varphi\oplus\psi\,\d(\ggamma-\sigma)\\
(\text{by }\eqref{eq:kyoto3})\qquad&=\int\varphi\,\d\pi^0_*(\ggamma-\sigma)+\int\psi\,\d\pi^1_*(\ggamma-\sigma)\\
(\text{because }\sigma\text{ and }\ggamma\text{ have the same marginals})\qquad&=0.
\end{split}
\]
This equality and the convexity of $H(\cdot\,|\,\hR)$ yield $H(\sigma\,|\,\hR)\leq H(\ggamma\,|\,\hR)$ and being $\ggamma$  the unique minimum of $H(\cdot\,|\,\hR)$  among transport plans from $\mu_0$ to $\mu_1$, we conclude that  $\sigma=\ggamma$.
\end{proof}

The above result is valid in the very general framework of Polish spaces. We shall now restate it in the form we shall need in the context  of $\RCD$ spaces and show that  additional regularity assumptions on $\rho_0,\rho_1$ reflect into the regularity of $f,g$. 

Recall that on $\RCD$ spaces there is a well defined heat kernel $\hr_\eps[x](y)$ (see \eqref{eq:hk} and \eqref{eq:rapprform}). The choice of working with $\hr_{\eps/2}$ is convenient for the computations we will do later on.
\begin{Theorem}\label{thm:5}
Let $({\X},\sfd,\mm)$ be a compact $\RCD^*(K,N)$ space with $K \in \mathbb{R}$, $N \in [1,\infty)$ and $\mm\in\prob \X$. For $\eps>0$ define $\hR^{\eps/2}\in\prob{\X^2}$ as
\[
\d \hR^{\eps/2}(x,y):= r_{\varepsilon/2}[x](y)\,\d\mm(x)\,\d\mm(y).
\]
Also, let  $\mu_0,\mu_1\in\prob \X$ be Borel probability measures with bounded densities.

Then there exist and are uniquely $\mm$-a.e.\ determined (up to multiplicative constants) two Borel non-negative functions $f^{\varepsilon}, g^{\varepsilon} : \X \to [0,\infty)$ such that
$ f^{\varepsilon}\otimes g^{\varepsilon}\hR^{\eps/2}$ is a transport plan from $\mu_0$ to $\mu_1$.  In addition, $f^{\varepsilon},g^{\varepsilon}$ belong to $L^{\infty}(\mm)$.

Moreover, if the densities of $\mu_0,\mu_1$ belong to $ \testipp{\X}$, then $f^\eps,g^\eps\in \testipp \X$ as well.
\end{Theorem}
\begin{proof} The first part of the statement follows directly from  Proposition \ref{pro:2}  and the fact that the Gaussian estimates \eqref{eq:gaussest} on the heat kernel grant that there are  constants $0 < c_{\varepsilon} \leq C_{\varepsilon} < +\infty$ such that
\[
c_{\varepsilon} \mm \otimes \mm \leq \hR^{\varepsilon} \leq C_{\varepsilon} \mm \otimes \mm.
\]
For the second part, notice that thanks to the representation formula \eqref{eq:rapprform}, the fact that $\pi^0_*(f^\eps\otimes g^\eps \hR^{\eps/2})=\rho_0\mm$ reads as
\[
f^\eps \h_{\eps/2}(g^\eps)=\rho_0.
\]
Now notice that by \eqref{eq:regflow} we have  $\h_{\eps/2}(g^\eps)\in \testipp \X$, and thus from \eqref{eq:compos} applied with $\varphi(z):=z^{-1}$ we deduce that  $\frac1{\h_{\eps/2}(g^\eps)}\in\testipp \X$. Since $\testi \X$ is an algebra we conclude that $f^\eps=\frac{\rho_0}{\h_\eps(g^\eps)}\in\testipp \X$. The same applies to $g^\eps$.
\end{proof}

\section{Old estimates in a new setting}\label{sec:4}

Aim of this part is to adapt two results already known in the Riemannian framework to the context of (compact) $\RCD$ spaces:  Hamilton's gradient estimate and Li-Yau inequality. As we learnt while already working on this manuscript, the former has already been proved on proper  $\RCD$ spaces by Jiang-Zhang in \cite{JiangZhang16}; since we have the additional compactness assumption, the proof simplifies a bit and for completeness we present it. In this direction, we also prove a bound which seems new in the non-smooth context, namely a uniform bound on $|\nabla\log\h_tu|$ in the special case  $|\nabla\log u|\in L^\infty$, see Proposition \ref{pro:kyoto}. On the other hand, to the best of our knowledge the Li-Yau inequality is only known on $\RCD^*(0,N)$ spaces from \cite{GarofaloMondino14} and \cite{Jiang15}. Here we generalize such result to negative Ricci bounds in the case of compact spaces: the bound that we obtain is quite rough, but sufficient for our purposes.

\subsection{Comparison principles}

The proofs of Hamilton's gradient estimate and of the Li-Yau inequality are based on the following two comparison principles, valid in general infinitesimally Hilbertian spaces $(\Y,\sfd_\Y,\mm_\Y)$.

 To formulate the result we need to introduce the dual of $W^{1,2}(\Y)$, which we shall denote $W^{-1,2}(\Y)$. As usual, the fact that $W^{1,2}(\Y)$ embeds in $L^2(\Y)$ with dense image allows to see $L^2(\Y)$ as a dense subset of $W^{-1,2}(\Y)$, where $f\in L^2(\Y)$ is identified with the mapping $W^{1,2}(\Y)\ni g\mapsto \int fg\,\d\mm_\Y$.
 
Notice also that even in this generality, a regularization via the heat flow shows that $D(\Delta)$ is dense in $W^{1,2}(\Y)$ and, with the use of the maximum principle \eqref{eq:maxprinc}, that non-negative functions in $D(\Delta)$ are $W^{1,2}$-dense in the space of non-negative functions in $W^{1,2}$.

\begin{Proposition}\label{pro:3}
Let $(\Y,\sfd_\Y,\mm_\Y)$ be an infinitesimally Hilbertian space. Then the following two comparison principles hold:
\begin{itemize}
\item[(i)]  let $(F_t) , (G_t)  \in   AC_{loc}([0,\infty),L^2(\Y))$ be respectively a weak super- and weak sub- solution of the heat equation, i.e.\ such that for all $h \in D(\Delta)$ non-negative and a.e.\ $t > 0$ it holds
\[
\frac{\d}{\d t}\int hF_t\d\mm_\Y \geq \int \Delta h F_t\d\mm_\Y, \qquad \qquad \frac{\d}{\d t}\int hG_t\d\mm_\Y \leq \int \Delta h G_t\d\mm_\Y.
\]
Assume that $F_0 \geq G_0$ $\mm$-a.e. Then $F_t \geq G_t$ $\mm$-a.e.\ for every $t > 0$.
\item[(ii)] Let  $a_0,a_1,a_2 \in \R$ and  $(v_t)\in  L^1_{loc}([0,\infty),W^{1,2}(\Y))$ with $v_t\in D(\Delta)$ for a.e.\ $t$ and $\|\Delta v_t\|_{L^\infty}\in L^1_{loc}([0,\infty))$ and let $(F_t),(G_t)\in L^\infty_{loc}([0,\infty),L^{\infty}(\Y))\cap L^\infty_{loc}([0,\infty),W^{1,2}(\Y))\cap AC_{loc}([0,\infty),W^{-1,2}(\Y))$ be respectively a weak super- and weak sub- solution of 
\begin{equation}\label{eq:16}
\frac{\d}{\d t} u_t = \Delta u_t + a_0u_t^2 + a_1u_t + \langle \nabla u_t,\nabla v_t \rangle + a_2
\end{equation}
 in the following sense: for all $h \in D(\Delta) $ non-negative  and a.e.\ $t > 0$ it holds
\begin{eqnarray}
&& \frac{\d}{\d t}\int h F_t\d\mm_\Y \geq \int  \Delta h F_t\d\mm_\Y + \int  h\Big(a_0F_t^2 + a_1F_t + \langle \nabla F_t,\nabla v_t \rangle + a_2\Big)\d\mm_\Y, \nonumber \\
&& \frac{\d}{\d t}\int  h G_t\d\mm_\Y \leq \int  \Delta h G_t\d\mm_\Y + \int  h\Big(a_0G_t^2 + a_1G_t + \langle \nabla G_t,\nabla v_t \rangle + a_2\Big)\d\mm_\Y. \nonumber
\end{eqnarray}
Assume that $F_0 \geq G_0$ $\mm_\Y$-a.e.. Then $F_t \geq G_t$ $\mm_\Y$-a.e.\ for every $t > 0$.
\end{itemize}
\end{Proposition}
\begin{proof}

\noindent{\bf(i)} By linearity it is not restrictive to assume $G_t \equiv 0$ for all $t \geq 0$. Fix $\varepsilon > 0$, notice that $t \mapsto \h_{\varepsilon}F_t$ belongs to $  AC_{loc}([0,\infty),L^2(\Y))$ with values in $D(\Delta)$. Then pick $h\in D(\Delta)$ non-negative, notice that $\h_\eps h$ is non-negative as well to get
\[
\int h\frac{\d}{\d t}\h_{\varepsilon} F_t\d\mm_\Y= \frac{\d}{\d t}\int  h\h_{\varepsilon}F_t\d\mm_Y = \frac{\d}{\d t}\int  (\h_{\varepsilon} h) F_t\d\mm_\Y \geq \int  \Delta\h_{\varepsilon} h \,F_t\d\mm_\Y = \int h\Delta\h_{\varepsilon} F_t\d\mm_\Y.
\]
Since this is true for all $h\in D(\Delta)$ non-negative and, by what we said before, this class of functions is $L^2$-dense in the set of non-negative $L^2$-functions, we deduce that  for a.e.\ $t > 0$ it holds
\begin{equation}
\label{eq:pmax}
\frac{\d}{\d t}\h_{\varepsilon}F_t \geq \Delta\h_{\varepsilon}F_t, \quad \mm_Y\ae.
\end{equation}
Now notice that being $F_0 \geq 0$, by the maximum principle \eqref{eq:maxprinc} we see that $\h_{\varepsilon}F_0 \geq 0$ too and we claim that from this fact and \eqref{eq:pmax} it follows that  $\h_\eps(F_t)\geq 0$ for every $t\geq 0$. Thus  let us consider
\[
\Phi(t) := \frac12\int {|\phi(\h_{\varepsilon}F_t)|^2}\d\mm_\Y,
\]
where $\phi(z):=z^-=\max\{0,-z\}$.  Observe that $\Phi\in  AC_{loc}([0,\infty))$, that $\Phi(0) = 0$ and compute
\begin{equation}
\label{eq:step1}
\begin{split}
\Phi'(t)&=\int\phi(\h_\eps F_t)\frac\d{\d t}\phi(\h_\eps F_t)\,\d\mm_\Y=\int\phi'(\h_\eps F_t)\,\phi(\h_\eps F_t)\frac{\d}{\d t}\h_\eps F_t\,\d\mm_\Y=-\int\phi(\h_\eps F_t)\frac{\d}{\d t}\h_\eps F_t \,\d\mm_\Y
\end{split}
\end{equation}
and therefore taking \eqref{eq:pmax} into account we see that
\[
\begin{split}
\Phi'(t)&\leq-\int\,\phi(\h_\eps F_t)\Delta\h_\eps F_t\,\d\mm_\Y=\int \la\nabla\phi(\h_\eps F_t),\nabla \h_\eps F_t \ra\,\d\mm_\Y=-\int |\nabla \phi(\h_\eps F_t)|^2\,\d\mm_\Y\leq 0.
\end{split}
\]
Thus $\Phi(t)=0$ for every $t\geq 0$, i.e.\  $\h_{\varepsilon}F_t \geq 0$ for all $t \geq 0$. Letting $\varepsilon \downarrow 0$ we conclude.

\noindent{\bf(ii)} 
Since $(F_t)\in  L^\infty_{loc}([0,\infty),W^{1,2}(\Y))$, the fact that it is a supersolution of \eqref{eq:16} can be written as
\begin{equation}
\label{eq:16f}
\frac{\d}{\d t}\int h F_t\d\mm_\Y \geq -\int \la\nabla h,\nabla F_t\ra \d\mm_\Y + \int h\Big(a_0F_t^2 + a_1F_t + \langle \nabla F_t,\nabla v_t \rangle + a_2\Big)\d\mm_\Y
\end{equation}
for every $h\in D(\Delta)$ non-negative. Recalling that the class of such functions is $W^{1,2}$-dense in the one  of non-negative $W^{1,2}$ functions, passing through the integral formulation - in time - of \eqref{eq:16f} it is immediate to see that \eqref{eq:16f} also holds for any $h\in W^{1,2}(\Y)$ non-negative. Using the fact that $W^{-1,2}(\Y)$ has the Radon-Nikodym property (because it is Hilbert) we see that $(F_t)$ seen as curve with values in $W^{-1,2}(\Y)$ must be differentiable at a.e.\ $t$ and it is then clear that for any point of differentiability $t$, the inequality \eqref{eq:16f} holds for any $h\in W^{1,2}(\Y)$ non-negative, i.e.\ that the set of $t$'s for which \eqref{eq:16f} holds is independent on $h$. The analogous property holds for $(G_t)$.

Now we apply Lemma \ref{le:kyoto} below to $h_t:=G_t-F_t$ to get that $\Phi(t):=\frac12\int|(G_t-F_t)^+|^2\,\d\mm_Y$ is absolutely continuous and 
\[
\Phi'(t)= \int (G_t-F_t)^+\frac{\d}{\d t}(G_t-F_t)\,\d\mm_Y,
\]
where the right hand side is intended as the coupling of $\frac{\d}{\d t}(G_t-F_t)\in W^{-1,2}(\Y)$ and the function $ (G_t-F_t)^+\in W^{1,2}(\Y)$. Fix $t$ which is a differentiability point of both $(F_t)$ and $(G_t)$,  pick $h:=(G_t-F_t)^+$ in \eqref{eq:16f} and in the analogous inequality for $(G_t)$ to obtain
\[
\begin{split}
\Phi'(t)&\leq \int-\la\nabla((G_t-F_t)^+),\nabla(G_t-F_t)\ra\\
&\qquad\qquad+ (G_t-F_t)^+\Big(a_0(G_t^2-F_t^2)+a_1(G_t-F_t)+\la\nabla(G_t-F_t),\nabla v_t\ra\Big)\,\d\mm_\Y
\end{split}
\]
and since $\la\nabla h^+,\nabla h\ra=|\nabla h^+|^2$ and $h^+\nabla h=\frac12\nabla(h^+)^2 $ for any $h\in W^{1,2}$, we have
\[
\begin{split}
\Phi'(t)&\leq\int -|\nabla((G_t-F_t)^+)|^2+|(G_t-F_t)^+|^2\big(a_0(G_t+F_t)+a_1-\tfrac12\Delta v_t\big)\,\d\mm_\Y\\
&\leq 2\Phi(t)\big(|a_0|\|G_t+F_t\|_{L^\infty}+|a_1|+\tfrac12\|\Delta v_t\|_{L^\infty}\big).
\end{split}
\]
Since the assumption $F_0\geq G_0$ gives $\Phi(0)=0$, by Gronwall's lemma we conclude that $\Phi(t)=0$ for any $t\geq 0$, which is the thesis.
\end{proof}

\begin{Lemma}\label{le:kyoto}
Let $(h_t)\in L^\infty_{loc}([0,\infty),W^{1,2}(\Y))\cap AC_{loc}([0,\infty),W^{-1,2}(\Y))$. 

Then $t\mapsto\frac12\int|(h_t)^+|^2\,\d\mm_\Y$ is locally absolutely continuous on $[0,\infty)$ and it holds
\begin{equation}
\label{eq:tesilemma}
\frac\d{\d t}\frac12\int|(h_t)^+|^2\,\d\mm_\Y=\int (h_t)^+\frac{\d}{\d t}h_t\,\d\mm_\Y,\qquad {\rm a.e.}\ t,
\end{equation}
where the integral in the right hand side is intended as the coupling of $(h_t)^+\in W^{1,2}(\Y)$ with $\frac\d{\d t}h_t\in W^{-1,2}(\Y)$.
\end{Lemma}
\begin{proof}
If $(h_t)\in AC_{loc}([0,\infty),L^2(\Y))$, the claim follows easily with the same computations done in \eqref{eq:step1}. The general case follows by approximation via the heat flow. Fix $\eps>0$ and notice that the fact that $\h_\eps$ is a contraction in $W^{1,2}$ and a bounded operator from $L^2$ to $W^{1,2}$ yield  the inequalities
\[
\begin{split}
\|\h_\eps f\|_{L^2}&=\sup_{\|g\|_{L^2}\leq 1}\int \h_\eps f\,g\,\d\mm_\Y\leq\sup_{\|g\|_{L^2}\leq 1}\|\h_\eps g\|_{W^{1,2}}\|f\|_{W^{-1,2}}\leq C_\eps\|f\|_{W^{-1,2}}\\
\|\h_\eps f\|_{W^{-1,2}}&=\sup_{\|g\|_{W^{1,2}}\leq 1}\int \h_\eps f\,g\,\d\mm_\Y\leq\sup_{\|g\|_{W^{1,2}}\leq 1}\|\h_\eps g\|_{W^{1,2}}\|f\|_{W^{-1,2}}\leq \|f\|_{W^{-1,2}},
\end{split}
\]
for all $f\in L^2$, which together with the density of $L^2$ in $W^{-1,2}$ ensures that $\h_\eps$ can be uniquely extended to a linear bounded operator from $W^{-1,2}$ to $L^2$ which is also a contraction when seen with values in $W^{-1,2}$. It is then clear that $\h_\eps f\to f$ in $W^{-1,2}$ as $\eps\downarrow0$ for any $f\in W^{-1,2}$. It follows that for $(h_t)$ as in the assumption, $(\h_\eps h_t)\in AC_{loc}([0,\infty),L^2(\Y))$, so that by what previously said the thesis holds for such curve and writing the identity  \eqref{eq:tesilemma} in integral form we have
\[
\frac12\int | (\h_\eps h_{t_1})^+|^2- | (\h_\eps h_{t_0})^+|^2\,\d\mm_\Y=\int_{t_0}^{t_1}\int  \big(\h_{\eps}h_t\big)^+\h_\eps\Big(\frac{\d}{\d t}h_t\Big)\,\d\mm_\Y\,\d t\qquad\forall 0\leq t_0\leq t_1.
\] 
Letting $\eps\downarrow0$, using the continuity at $\eps=0$ of $\h_\eps$ seen as operator on all the spaces $W^{1,2},L^2,W^{-1,2}$ and the continuity of $h\mapsto h^+$ as map from $W^{1,2}$ with the strong topology to $W^{1,2}$ with the weak one (which follows from the continuity of the same operator in $L^2$ together with the fact that it decreases the $W^{1,2}$ norm), we obtain
\begin{equation}
\label{eq:tesik}
\frac12\int | h_{t_1}^+|^2- | h_{t_0}^+|^2\,\d\mm_\Y=\int_{t_0}^{t_1}\int  (h_t)^+ \frac{\d}{\d t}h_t \,\d\mm_\Y\,\d t\qquad\forall 0\leq t_0\leq t_1.
\end{equation}
Now the bound
\[
\bigg|\int_{t_0}^{t_1}\int  (h_t)^+ \frac{\d}{\d t}h_t \,\d\mm_\Y\,\d t\bigg|\leq \|(h_t)\|_{L^\infty([t_0,t_1],W^{1,2})}\int_{t_0}^{t_1}\Big\|\frac{\d}{\d t}h_t\Big\|_{W^{-1,2}}\,\d t
\] 
   grants the local absolute continuity of $t\mapsto \frac12\int |h_t^+|^2\,\d\mm_\Y$ and the conclusion follows by differentiating \eqref{eq:tesik}.
\end{proof}
\subsection{Hamilton's gradient estimates and related inequalities}
We  start proving Hamilton's gradient estimate on compact $\RCD(K,\infty)$ spaces, with a proof which closely follows the original one in \cite{Hamilton93}. As already said, in fact the same result is know to be true - from \cite{JiangZhang16} - on the more general class of proper $\RCD(K,\infty)$ spaces, but given that the compactness assumption slightly simplifies the argument, for completeness we provide the proof. 
\begin{Proposition}\label{pro:4}
Let $({\X},\sfd,\mm)$ be a compact $\RCD^*(K,\infty)$ space with $K\in\R$ and  let $u_0 \in L^{\infty}({\mm})$ be such that  $u_0 \geq c$ for some positive constant $c$. Put  $u_t := \h_t u_0$ for all $t > 0$. Then
\begin{equation}\label{eq:12}
t|\nabla\log u_t|^2 \leq (1 + 2K^-t)\log\bigg(\frac{\|u_0\|_{L^{\infty}(\mm)}}{u_t}\bigg), \quad \mm\ae
\end{equation}
for all $t > 0$, where $K^-:=\max\{0,-K\}$.
\end{Proposition}

\begin{proof} Let us assume for the moment that $u_0\in\testipp\X$. Set $M := \|u_0\|_{L^{\infty}(\mm)}$ and define for $t \geq 0$
\[
v_t:= \varphi_t\frac{|\nabla u_{t }|^2}{u_{t}} - u_{t }\log\frac{M}{u_{t }},\qquad\text{with }\quad\varphi_t := \frac{t}{1 + 2K^-t}.
\]
Notice that by  the maximum principle \eqref{eq:maxprinc} we know that $c\leq u_t \leq M$ for all $t \geq 0$, thus the definition of $v_t$ is well posed. 

Our thesis is equivalent to the fact that $v_t\leq 0$ and we shall prove this via the comparison principle for the heat flow stated in point $(i)$ of Proposition \ref{pro:3}. The fact that $(u_t)\in AC_{loc}([0,\infty),W^{1,2}(\X))$ and - by the maximum principle \eqref{eq:maxprinc} and the Bakry-\'Emery inequality \eqref{eq:be} - that $(\log(u_t)),(|\nabla u_t|)\in L^\infty_{loc}([0,\infty),L^\infty(\X))$ grant that  $(v_t)\in AC_{loc}([0,\infty), L^2(\X))$. Since  by construction we have $v_0\leq 0$,  we are left to prove that for any $h\in D(\Delta)$ non-negative it holds
\[
\int  h\frac{\d}{\d t} v_t\,\d\mm \leq \int  v_t \Delta h\d\mm \qquad {\rm a.e.}\ t.
\]
We have $u_t\in D(\Delta)$ and, by  \eqref{eq:lapmistest}, that $|\nabla u_t|^2\in D(\bd)$ for any $t\geq 0$, thus since as said   $0<c\leq u_t\leq M$ for  all $t\geq 0$, we deduce that  $v_t\in D(\bd)$ for any $t\geq 0$. Hence our thesis can be rewritten as
\[
\Big(\frac\d{\d t}v_t\Big)\mm\leq \bd v_t \qquad {\rm a.e.}\ t.
\]
The conclusion now follows by  direct computation. We have
\begin{equation}
\label{eq:h1}
\frac{\d}{\d t}v_t=\varphi'_t\frac{|\nabla u_{t }|^2}{u_{t }}+\varphi_t\Big(\frac2{u_t}\la\nabla u_t,\nabla\Delta u_t\ra-\Delta u_t\frac{|\nabla u_t|^2}{u_t^2}\Big)-\Delta u_t\log\frac M{u_t}+\Delta u_t
\end{equation}
and
\begin{equation}
\label{eq:h2}
\Delta \Big(u_t\log\frac M{u_t}\Big)=(\Delta u_t)\log\frac M{u_t}-\Delta u_t-\frac{|\nabla u_t|^2}{u_t}.
\end{equation}
Moreover
\[
\begin{split}
\bd\frac{|\nabla u_t|^2}{u_t}=\frac1{u_t}\bd|\nabla u_t|^2+\Big(|\nabla u_t|^2\Delta(u_t^{-1})+2\la\nabla|\nabla u_t|^2,\nabla(u_t^{-1})\ra\Big)\mm\\
\end{split}
\]
so that using the Bochner inequality \eqref{eq:bochhess} we obtain
\begin{equation}
\label{eq:h3}
\begin{split}
\bd\frac{|\nabla u_t|^2}{u_t}&\geq \Big(\frac2{u_t}|\H {u_t}|_\HS^2+\frac2{u_t}\la\nabla u_t,\nabla \Delta u_t\ra+\frac{2K}{u_t}|\nabla u_t|^2\\
&\qquad\qquad-\Delta u_t\frac{|\nabla u_t|^2}{u_t^2}+2\frac{|\nabla u_t|^4}{u_t^3}-\frac2{u_t^2}\la\nabla u_t,\nabla|\nabla u_t|^2\ra\Big)\mm.
\end{split}
\end{equation}
Putting together \eqref{eq:h1}, \eqref{eq:h2} and \eqref{eq:h3} and using the identity
\[
\Big|\H{u_t}-\frac{\nabla u_t\otimes\nabla u_t}{u_t}\Big|^2_\HS=|\H{u_t}|_\HS^2+\frac{|\nabla u_t|^4}{u_t^2}-\frac{\la\nabla u_t,\nabla|\nabla u_t|^2\ra}{u_t}
\]
we obtain
\[
\begin{split}
\Big(\frac\d{\d t}v_t\Big)\mm- \bd v_t\leq  \Big(\frac{|\nabla u_t|^2}{u_t}\big(\varphi'_t-2K\varphi-1\big)-\frac2{u_t}\Big|\H{u_t}-\frac{\nabla u_t\otimes\nabla u_t}{u_t}\Big|^2_\HS\Big)\mm
\end{split}
\]
and the conclusion follows noticing that by the definition of $\varphi_t$ we have
\[
\varphi'_t-2K\varphi_t-1\leq 0\qquad\forall t\geq 0.
\]
For the general case, recall that by \eqref{eq:regflow} and our assumption on $u_0$ we have that $u_\eps\in\testipp\X$ for every $\eps>0$  and notice that what we have just proved grants that
\[
t|\nabla\log u_{t+\eps}|^2\leq (1+2K^-t)\log\Big(\frac{\|u_\eps\|_{L^\infty}}{u_{t+\eps}}\Big),\qquad\mm\ae,\quad\forall t\geq 0.
\]
By the maximum principle \eqref{eq:maxprinc} we have that $\|u_\eps\|_{L^\infty}\leq\|u_0\|_{L^\infty}$, then the conclusion easily follows letting $\eps\downarrow0$ and using the continuity of $\eps\mapsto u_\eps,|\nabla u_\eps|\in L^2(\mm)$.
\end{proof}
In the compact finite-dimensional case, thanks to the Gaussian estimates for the heat kernel we can now easily obtain a bound independent on the  $L^\infty$ norm of the initial datum present in inequality \eqref{eq:12}:
\begin{Theorem}\label{cor:1}
Let $({\X},\sfd,\mm)$ be a compact $\RCD^*(K,N)$ space with $K \in\R$ and $N \in [1,\infty)$. Then there is a constant $C$ depending on $K,N$ and $D:={\rm diam}(\X)$ only such that for any $u_0 \in L^1(\mm)$  non-negative and not identically 0 the inequality
\begin{equation}\label{eq:23}
|\nabla\log (u_t)|^2 \leq C\bigg(1  + \frac{1}{t^2}\bigg), \quad \mm\ae
\end{equation}
holds for all $t > 0$, where $u_t:=\h_t u$. In particular, for every $\delta>0$ there is a constant $C_\delta>0$ depending on $K,N,D,\delta$ only such that
\begin{equation}
\label{eq:ham}
\sup_{\eps\in(0,1)}\eps \|\nabla\log(u_{\eps t})\|_{L^\infty}\leq C_\delta\qquad\forall t\geq \delta.
\end{equation}
\end{Theorem}
\begin{proof} Recall  the representation formula \eqref{eq:rapprform}:
\[
u_t(x)=\int u(y)r_t[y](x)\,\d\mm(y)\qquad\forall x\in \X
\]
and that for the transition probability densities $r_t[y](x)$ we have the Gaussian estimates \eqref{eq:gaussest}
\[
\frac{C_0}{\mm(B_{\sqrt t}(y))}e^{-C_1\frac{D^2}t}\leq r_t[y](x)\leq  \frac{C_2}{\mm(B_{\sqrt t}(y))}\qquad\forall x,y\in \X,
\]
for appropriate constants $C_0,C_1,C_2$ depending only on $K,N$. Therefore we have
\[
\begin{split}
\|u_t\|_{L^\infty}&=\sup_x u_t(x)\leq C_2\int\frac{u(y)}{\mm(B_{\sqrt t}(y))}\,\d\mm(y),\\
\inf_x u_{2t}(x)&\geq C_0e^{-C_1\frac{D^2}t}\int\frac{u(y)}{\mm(B_{\sqrt {2t}}(y))}\,\d\mm(y)>0.
\end{split}
\]
By the Bishop-Gromov inequality we know that for some constant $C_3>0$ it holds
\[
\mm(B_{\sqrt{2t}}(y))\leq C_3\mm(B_{\sqrt{t}}(y))\qquad\forall y\in \X,\ t>0,
\]
hence the above yields
\[
\frac{\|u_t\|_{L^\infty}}{u_{2t}(x)}\leq \frac{C_2 C_3}{C_0}e^{C_1\frac{D^2}t}\qquad\forall x\in\X,\ t>0.
\]
We now apply Proposition \ref{pro:4}  with $u_t$ in place of $u_0$ (notice that the assumptions are fulfilled) to get
\[
\begin{split}
t|\nabla\log(u_{2t})|^2&\leq (1+2K^-t)\log\Big(\frac{\|u_t\|_{L^\infty}}{u_{2t}}\Big)\leq (1+2K^-t)\Big(\log\big(\frac{C_2C_3}{C_0}\big) +C_1\frac{D^2}t\Big) \qquad\mm\ae,
\end{split}
\]
which is (equivalent to) the bound \eqref{eq:23}. The last statement is now obvious.
\end{proof}
In inequality \eqref{eq:23}, the right hand side blows-up at $t=0$ and thus it gives no control for small $t$'s. In the next simple proposition we show that if the initial datum is good enough, then we have a control for all $t$'s:
\begin{Proposition}\label{pro:kyoto}
Let $(\X,\sfd,\mm)$ be a compact $\RCD^*(K,\infty)$ space with $K \in\R$  and let $u_0:\X\to(0,\infty)$ be such that $\log u_0$ is Lipschitz. Put  $u_t := \h_t u_0$ for all $t > 0$. Then
\[
|\nabla\log u_t|\leq e^{-Kt}\||\nabla\log u_0|\|_{L^\infty}\qquad \mm\ae.
\]
\end{Proposition}
\begin{proof} Assume for a moment that $u_0\in\testipp\X$ and put $\varphi_t:=\log u_t\in\testi\X$ so that, also recalling the calculus rules stated in the preliminary section, we have $(\varphi_t)\in AC_{loc}([0,\infty),L^{2}(\X))$ and
\begin{equation}
\label{eq:derv}
\ddt\varphi_t=|\nabla\varphi_t|^2+\Delta\varphi_t.
\end{equation}
By the maximum principle \eqref{eq:maxprinc} we know that $u_t(x)\in [c,C]$ for any $t,x$, for some $[c,C]\subset (0,\infty)$ and from this fact and the chain rule for the differential and Laplacian it easily follows that $(\Delta\varphi_t)\in L^\infty_{loc}([0,\infty),W^{1,2}(\X))$ and $(|\nabla\varphi_t|)\in  L^\infty_{loc}([0,\infty),L^\infty(\X))$. Hence taking \eqref{eq:reggrad}  into account we see that $|\nabla\varphi_t|^2\in L^\infty_{loc}([0,\infty),W^{1,2}(\X))$ as well. Therefore from \eqref{eq:derv} we deduce that $(\varphi_t)\in AC_{loc}([0,\infty),W^{1,2}(\X)) $ so that putting
\[
F_t:=|\nabla\varphi_t|^2,
\]
we have that $(F_t)$ satisfies the regularity assumptions needed in point $(ii)$ of Proposition \ref{pro:3} (notice that trivially $AC_{loc}([0,\infty),L^2(\X))\subset AC_{loc}([0,\infty),W^{-1,2}(\X))$). Moreover, from \eqref{eq:derv} we get
\[
\ddt F_t=2\la\nabla\varphi_t,\nabla F_t\ra+2\la\nabla\varphi_t,\nabla\Delta\varphi_t\ra
\]
and therefore from the Bochner inequality \eqref{eq:bochhess} written for $\varphi_t$ - neglecting the term with the Hessian - we see that for any $h\in\testip \X$ it holds
\[
\ddt\int h F_t\,\d\mm\leq \int\Delta hF_t+2h\Big(\la\nabla\varphi_t,\nabla F_t\ra-K F_t\Big)\,\d\mm,
\]
showing that $(F_t)$ is a weak subsolution of \eqref{eq:16} with
\[
a_0=0\qquad a_1=-2K\qquad a_2=0\qquad v_t=2\varphi_t.
\]
On the other hand, the function
 \[
 G_t(x):=e^{-2Kt}\|F_0\|_{L^\infty}
 \]
is a solution of \eqref{eq:16} and  $F_0\leq G_0$ $\mm$-a.e..  Since from the chain rule for the Laplacian and the maximum principle \eqref{eq:maxprinc} we have $\Delta\varphi_t\in L^\infty_{loc}([0,\infty),L^\infty(\X))$, we see that we are in position to apply point $(ii)$ of Proposition \ref{pro:3} and  deduce that $F_t\leq G_t$ $\mm$-a.e.\ for every $t>0$, which is the thesis.

For the case of general $u_0$ as in the assumptions, we put $u^\eps_0:=e^{\h_\eps(\log(u_0))}$ and notice that  by the Bakry-\'Emery estimate \eqref{eq:be}, it holds
\[
\lims_{\eps\downarrow0}\||\nabla\log u_0^\eps|\|_{L^\infty}\leq\||\nabla\log u_0|\|_{L^\infty}.
\] 
Then put $\varphi^\eps_t:=\log \h_tu^\eps_0$ and notice that this last inequality together with what previously proved grants that
\[
\lims_{\eps\downarrow0}\||\nabla \varphi_t^\eps|\|_{L^\infty}\leq e^{-Kt}\||\nabla\log u_0|\|_{L^\infty}.
\]
Conclude noticing that $\varphi^\eps_t\to \log u_t$ $\mm$-a.e.\ as $\eps\downarrow0$ and use the closure of the differential.
\end{proof}

\subsection{A Li-Yau type inequality}
We now prove a version of Li-Yau inequality valid on general compact $\RCD^*(K,N)$ spaces, where $K$ is possibly negative: the bound \eqref{eq:liyau} that we obtain is not sharp (as it is seen by letting $K\uparrow 0$ in the estimate \eqref{eq:21} provided in the proof) but sufficient for our needs.
\begin{Theorem}\label{lem:4}
Let $(\X,\sfd,\mm)$ be a compact $\RCD^*(K,N)$ space with $K \in\R$ and $N \in [1,\infty)$. Then for every $\delta>0$ there exists a constant $C_\delta>0$ depending on $K,N$, ${\rm Diam}(\X)$ and $\delta$ only such that the following holds.

For any  $u_0 \in L^1(\mm)$  non-negative and non-zero and $\eps\in(0,1)$ it holds
\begin{equation}
\label{eq:liyau}
\eps\Delta \log(\h_{\eps t}(u_0))\geq -C_\delta \qquad\forall t\geq \delta.
\end{equation}
\end{Theorem}
\begin{proof} We can, and will, assume $K<0$. Let  $C$ be the constant given by Theorem \ref{cor:1} (which only depends on $K,N$ and ${\rm Diam}(\X)$) and put
\[
\alpha(t) :=-KC\Big(1 + \frac{4}{t^2}\Big)>0.
\]
We shall prove that for $u_0$ as in the assumptions we have
\begin{equation}\label{eq:21}
\Delta\log u_t \geq -\sqrt{N\alpha(t)}\coth\Big(\sqrt{\frac{\alpha(t)}N}t\Big)\qquad\forall t>0.
\end{equation}
From this the thesis easily follows as the function $\phi(t,\eps):=\eps \sqrt{N\alpha(\eps t)}\coth\Big(\sqrt{\frac{\alpha(\eps t)}N}\eps t\Big)$ is decreasing in $t$ - as seen by direct computation - so that \eqref{eq:liyau} follows from \eqref{eq:21} and
\[
\lim_{\eps\downarrow 0}\phi(\delta,\eps)=\sqrt{\frac{-4KCN}{\delta^2}}\coth\Big(\sqrt{\frac{-4KC}N}\Big)<+\infty.
\]
Thus fix $u_0$ as in the statement and notice that $u_t\in\testipp \X$ for every $t>0$, so that  $f_t := \log u_t\in \testi \X$ for every $t>0$. Arguing as in the proof of Proposition \ref{pro:kyoto}  we see that  $(f_t)\in AC_{loc}((0,\infty),W^{1,2}(\X))$ with
\begin{equation}
\label{eq:hjb}
\frac{\d}{\d t} f_t = \Delta f_t + |\nabla f_t|^2, \quad \textrm{for a.e. }t > 0.
\end{equation}
Let $\eta>0$ to be fixed later and put $F_t:=\Delta f_{t+\eta}$. From the chain rules for the gradient and Laplacian it is readily verified that $(F_t)\in L^\infty_{loc}([0,\infty),L^{\infty}(\X))\cap L^\infty_{loc}([0,\infty),W^{1,2}(\X))$

Now, as in the proof of Lemma \ref{le:kyoto},  the trivial estimate
\[
\|\Delta f\|_{W^{-1,2}}=\sup_{\|g\|_{W^{1,2}}=1}\int g\Delta f\,\d\mm=\sup_{\|g\|_{W^{1,2}}=1}-\int\la\nabla g,\nabla f\ra\,\d\mm\leq \|f\|_{W^{1,2}}
\]
grants that $\Delta:D(\Delta)\to L^2$ can be uniquely extended to a linear bounded functional, still denoted by  $\Delta$, from $W^{1,2}(\X)$ to $W^{-1,2}(\X)$. It is then clear that  $(F_t)\in AC_{loc}([0,\infty),W^{-1,2}(\X))$.

We want to show that $(F_t)$ is a weak supersolution of \eqref{eq:16} for an appropriate choice of the parameters and to this aim we fix $h\in \testip \X$ and notice that
\[
\frac{\d}{\d t}\int hF_t\,\d\mm=\frac\d{\d t}\int \Delta h f_{t+\eta}\,\d\mm\stackrel{\eqref{eq:hjb}}=\int\Delta h(F_t + |\nabla f_{t+\eta}|^2)\,\d\mm.
\]
Using first the Bochner inequality \eqref{eq:bochlap} and then the gradient estimate \eqref{eq:23} we obtain
\[
\begin{split}
\frac{\d}{\d t}\int hF_t\,\d\mm&\geq \int \Delta h F_t+h\Big(2\la\nabla f_{t+\eta},\nabla F_t\ra+\frac2NF_t^2+2K|\nabla f_{t+\eta}|^2\Big)\,\d\mm\\
&\geq  \int \Delta h F_t+h\Big(2\la\nabla f_{t+\eta},\nabla F_t\ra+\frac2NF_t^2+ 2KC\big(1+\frac1{\eta^2}\big) \Big)\,\d\mm,
\end{split}
\]
thus indeed $(F_t)$ is a weak supersolution of \eqref{eq:16} for
\[
a_0:=\frac2N\qquad a_1:=0\qquad a_2(\eta):=2KC\big(1+\frac1{\eta^2}\big)\qquad v_t:=2f_{t+\eta}.
\]
Noticing that $\alpha_2(\eta)<0$, it is trivial to check that the function
\[
y_t:=-\sqrt{-\frac{ a_2(\eta)N}2}\coth\Big(\sqrt{-\frac{2 a_2(\eta)}N}(t+t_0)\Big)
\]
is the only solution of
\[
y'_t = \frac{2}{N}y^2_t +a_2(\eta)
\]
with $y_0=-\sqrt{-\frac{ a_2(\eta)N}2}\coth\Big(\sqrt{-\frac{2 a_2(\eta)}N}t_0\Big)$. Now recall that $F_0=\Delta f_\eta\in L^\infty$, so that choosing $t_0>0$ sufficiently small we have that $ F_0\geq y_0$ $\mm$-a.e..

Defining $G_t(x):=y_t$ it is then clear that $(G_t)$ is a weak (sub)solution of \eqref{eq:16}, and since $F_0\geq G_0$ holds $\mm$-a.e.\ and, as already argued in the proof of Proposition \ref{pro:kyoto}, $\Delta v_t\in L^\infty_{loc}([0,\infty),L^\infty(\X))$, Proposition \ref{pro:3} grants that for any $t>0$ it holds $F_t\geq G_t$ $\mm$-a.e., that is:
\[
\Delta\log(u_{t+\eta})\geq -\sqrt{-\frac{\alpha_2(\eta)N}2}\coth\Big(\sqrt{-\frac{2\alpha_2(\eta)}N}(t+t_0)\Big)\geq -\sqrt{-\frac{\alpha_2(\eta)N}2}\coth\Big(\sqrt{-\frac{2\alpha_2(\eta)}N}t\Big).
\] 
Picking $\eta:=t$ we obtain (an equivalent version of) \eqref{eq:21}.
\end{proof}
\section{The Schr\"odinger problem: properties of the solutions}\label{sec:5}
\subsection{The setting}
Let us fix once for all the assumptions and notations which  we shall use from now on.
\begin{Setting}\label{set}
 $(\X,\d,\mathfrak{m})$ is a compact $\RCD^*(K,N)$ space with $K \in\R$ and $N \in [1,\infty)$. $D<\infty$ is the diameter of $\X$ and  $\mu_0=\rho_0\mm$ and $\mu_1=\rho_1\mm$ are two absolutely continuous Borel probability measures with bounded densities.
 
 For any $\eps>0$ we consider the couple  $(f^{\varepsilon},g^{\varepsilon})$ given by Theorem \ref{thm:5} normalized in such a way that 
\[
\int\log(\h_{\frac\eps2} f^\eps)\rho_1\,\d\mm=0,
\]
then we  set $\rho^\eps_0:=\rho_0$, $\rho^\eps_1:=\rho_1$, $\mu^\eps_0:=\mu_0$, $\mu^\eps_1:=\mu_1$ and 
\[
\left\{\begin{array}{l}
f^{\varepsilon}_t := \h_{\varepsilon t/2}f^{\varepsilon} \\
\\
\varphi_t^{\varepsilon} := \varepsilon\log f_t^{\varepsilon}\\
\\
\text{for }t\in(0,1]
\end{array}
\right.\qquad\qquad
\left\{\begin{array}{l}
g^{\varepsilon}_t := \h_{\varepsilon(1-t)/2}g^{\varepsilon} \\
\\
\psi_t^{\varepsilon} := \varepsilon\log g_t^{\varepsilon}\\
\\
\text{for }t\in[0,1)
\end{array}
\right.\qquad\qquad
\left\{\begin{array}{l}
\rho^{\varepsilon}_t := f^{\varepsilon}_t g^{\varepsilon}_t \\
\\
\mu^{\varepsilon}_t := \rho^{\varepsilon}_t\mm\\
\\
\vartheta^{\varepsilon}_t := \frac12({\psi^{\varepsilon}_t - \varphi^{\varepsilon}_t})\\
\\
\text{for }t\in(0,1)
\end{array}
\right.
\]
\end{Setting}
The following proposition collects the basic properties of the functions just defined and the respective `PDEs' solved:
\begin{Proposition}\label{pro:7} With the same assumptions and notation as in Setting \ref{set}, the following holds.

All the functions are well defined and belong to $\testi \X$ and for any $\eps>0$ all the curves $(f^\eps_t),(g^\eps_t),(\varphi^\eps_t),(\psi^\eps_t),(\rho^\eps_t),(\vartheta^\eps_t)$ belong to $AC_{loc}(I,W^{1,2}(\X))$, where $I$ is the respective domain of definition (for $(\rho^\eps_t)$ we pick $I=(0,1)$) and their time derivatives are given by the following expressions for a.e.\ $t\in[0,1]$:
\begin{align*}
\frac{\d}{\d t}f^{\varepsilon}_t &= \frac{\varepsilon}{2}\Delta f^{\varepsilon}_t&&\qquad&\frac{\d}{\d t}g^{\varepsilon}_t &= -\frac{\varepsilon}{2}\Delta g^{\varepsilon}_t\\
\frac{\d}{\d t}\varphi^{\varepsilon}_t & =  \frac{1}{2}|\nabla\varphi^{\varepsilon}_t|^2 + \frac{\varepsilon}{2}\Delta\varphi^{\varepsilon}_t&&\qquad&-\frac{\d}{\d t}\psi^{\varepsilon}_t & =  \frac{1}{2}|\nabla\psi^{\varepsilon}_t|^2 + \frac{\varepsilon}{2}\Delta\psi^{\varepsilon}_t\\
\frac{\d}{\d t}\rho^{\varepsilon}_t &+ {\rm div}(\rho^{\varepsilon}_t\nabla\vartheta^{\varepsilon}_t) = 0&&\qquad&\frac{\d}{\d t}\vartheta^{\varepsilon}_t& + \frac{|\nabla\vartheta^{\varepsilon}_t|^2}{2} = -\frac{\varepsilon^2}{8}\Big(2\Delta\log\rho^{\varepsilon}_t + |\nabla\log\rho^{\varepsilon}_t|^2\Big).
\end{align*}
Moreover, for every $\eps>0$ we have:
\begin{itemize}
\item[i)]
\begin{equation}
\label{eq:boundbase}
\sup_{t\in C}\|h^\eps_t\|_{L^\infty}+\Lip(h^\eps_t)+\|\Delta h^\eps_t\|_{W^{1,2}}+\|\Delta h^\eps_t\|_{L^\infty}<\infty
\end{equation}
where $(h^\eps_t)$ is equal to any of $(f^\eps_t),(g^\eps_t),(\varphi^\eps_t),(\psi^\eps_t),(\rho^\eps_t),(\vartheta^\eps_t)$ and $C$ is a compact subset of the respective domain of definition (for $(\rho^\eps_t)$ we pick $I=(0,1)$),
\item[ii)] $\mu^\eps_t$ is a probability measure for every $t\in[0,1]$ and $(\rho^\eps_t)\in C([0,1],L^2(\X))$,
\item[iii)]  we have $f^\eps_t\to f^\eps$  and  $g^\eps_t\to g^\eps$ in $L^2(\X)$ as $t\downarrow0$ and $t\uparrow 1$ respectively,
\end{itemize}
Finally,  if we further assume $\rho_0,\rho_1 \in \testipp{\X}$, then all the above curves can be extended to curves in $AC([0,1],W^{1,2}(\X))$ and we can take $C=I$ in \eqref{eq:boundbase}.
\end{Proposition}
\begin{proof} Recalling \eqref{eq:regflow} we see that $f^\eps_{t_0}\in\testipp \X$ for any $t_0>0$. Then the maximum principle for the heat flow, the fact that it is a contraction in $W^{1,2}(\X)$ and the Bakry-\'Emery gradient estimates \eqref{eq:be} together with the Sobolev-to-Lipschitz property grant that \eqref{eq:boundbase} holds for $(f_t^\eps)$. The same arguments apply to $(g^\eps_t)$ and also show that for given $\eps>0$, both $(f^\eps_t)$ and $(g^\eps_t)$ are, locally in $t$, uniformly bounded from below by a positive constant. Then the bound \eqref{eq:boundbase} for $(\varphi^\eps_t),(\psi^\eps_t),(\rho^\eps_t),(\vartheta^\eps_t)$ follows from the chain rules for the gradient and Laplacian and the fact that $\log$ is smooth on $(0,\infty)$ and for the same reason these curves belong to $AC_{loc}(I,L^2(\X))$.

The equations for $\frac\d{\d t}\varphi^\eps_t$ and $\frac\d{\d t}\psi^\eps_t$ are easily derived, for $\frac\d{\d t}\rho^\eps_t$ we notice that $\eps\log\rho^\eps_t=\varphi^\eps_t+\psi^\eps_t$ and thus
\[
\begin{split}
\frac\d{\d t}\rho^\eps_t&=\rho^\eps_t\frac\d{\d t}\log\rho^\eps_t=\rho^\eps_t\frac1\eps\Big(\frac{|\nabla\varphi^\eps_t|^2}{2}-\frac{|\nabla\psi^\eps_t|^2}{2}+\frac\eps2\Delta\varphi^\eps_t-\frac\eps2\Delta\psi^\eps_t\Big)\\
&=\rho^\eps_t\Big(-\la\nabla\vartheta^\eps_t,\nabla\log\rho^\eps_t\ra-\Delta\vartheta^\eps_t\Big)=-\la\nabla\vartheta^\eps_t,\nabla\rho^\eps_t\ra-\rho^\eps_t\Delta\vartheta^\eps_t=-{\rm div}(\rho^\eps_t\nabla\vartheta^\eps_t)
\end{split}
\]
and for $\frac\d{\d t}\vartheta^\eps_t$ we observe that
\[
\begin{split}
\frac\d{\d t}\vartheta^\eps_t+\frac{|\nabla\vartheta^\eps_t|^2}2&=-\frac{|\nabla\psi^\eps_t|^2}{4}-\frac{\eps }4\Delta\psi^\eps_t-\frac{|\nabla\varphi^\eps_t|^2}4-\frac{\eps}4\Delta\varphi^\eps_t+\frac{|\nabla\psi^\eps_t|^2}8+\frac{|\nabla\varphi^\eps_t|^2}8-\frac{\la\nabla\psi^\eps_t,\nabla\varphi^\eps_t\ra}4\\
&=-\frac{\eps^2}4\Delta\log\rho^\eps_t-\frac18\Big(|\nabla\psi^\eps_t|^2+|\nabla\varphi^\eps_t|^2-2\la\nabla\varphi^\eps_t,\nabla\psi^\eps_t\ra\Big)\\
&=-\frac{\varepsilon^2}{8}\Big(2\Delta\log\rho^{\varepsilon}_t + |\nabla\log\rho^{\varepsilon}_t|^2\Big).
\end{split}
\]
The fact that $(\varphi^\eps_t),(\psi^\eps_t),(\rho^\eps_t),(\vartheta^\eps_t)$ are absolutely continuous with values in $W^{1,2}(\X)$ is then a direct consequence of the expressions for their derivatives and the bound \eqref{eq:boundbase} in conjunction with \eqref{eq:reggrad}.

It is clear that $\rho^\eps_t\geq 0$ for every $\eps,t$, hence the identity
\[
\int\rho^\eps_t\,\d\mm=\int \h_{\varepsilon t/2}f^\eps\h_{\varepsilon(1-t)/2}g^{\varepsilon}\,\d\mm=\int f^\eps\h_{\eps/2} g^\eps\,\d\mm=\int \rho^\eps_0\,\d\mm=1
\]
shows that $\mu^\eps_t\in\prob \X$. 

Due to the continuity of $[0,\infty)\ni t\mapsto\h_th\in L^2(\X)$ for every $h\in L^2(\X)$, the claimed 

The claimed continuities in $L^2$ for the $f$'s and $g$'s  follow from the continuity in $L^2$ of $[0,\infty)\ni t\mapsto \h_th$ for every $h\in L^2$. Then for what concerns the $\rho$'s, we need to check that for every $\eps>0$ we have
\begin{equation}
\label{eq:contin0}
\rho_0=f^\eps\h_{\eps/2}g^\eps\qquad\qquad\qquad\rho_1=g^\eps\h_{\eps/2}f^\eps.
\end{equation}
As already noticed in the proof of Theorem \ref{thm:5}, these are equivalent to the fact that $f^\eps\otimes g^\eps\,\hR^{\eps/2}$ is a transport plan from $\mu_0$ to $\mu_1$; hence, \eqref{eq:contin0} holds by the very choice of $(f^\eps,g^\eps)$ made.

Finally, the last claim follows recalling that   the last part of Theorem \ref{thm:5} grants that  $f^\eps,g^\eps\in\testipp \X$ and then arguing as before.
\end{proof}

Using the terminology adopted in the literature (see \cite{Leonard14}) we shall refer to:
\begin{itemize}
\item $\varphi^{\varepsilon}_t$ and $\psi^{\varepsilon}_t$ as Schr\"odinger potentials, in connection with Kantorovich ones;
\item $(\mu^{\varepsilon}_t)_{t \in [0,1]}$ as entropic interpolation, in analogy with displacement one.
\end{itemize}

\subsection{Uniform estimates for the densities and the potentials}
We start collecting information about quantities which remain bounded as $\eps\downarrow0$.
\begin{Proposition}[uniform $L^{\infty}$ bound on the densities]\label{pro:6} With the same assumptions and notations as in Setting \ref{set} the following holds.

There exists a constant $M > 0$ which only depends on $K,N,D$ such that
\begin{equation}
\label{eq:linftybound}
\|\rho_t^{\varepsilon}\|_{L^{\infty}(\mathfrak{m})} \leq M\max\{\|\rho_0\|_{L^{\infty}(\mathfrak{m})}, \|\rho_1\|_{L^{\infty}(\mathfrak{m})} \}
\end{equation}
for every $t \in [0,1]$ and for every $\varepsilon > 0$.
\end{Proposition}

\begin{proof}
Fix $\varepsilon > 0$. We know from Proposition \ref{pro:7}  that $(\rho_t^\eps)\in C([0,1],L^2(\X))\cap AC_{loc}((0,1),L^2(\X))$, thus for many $p>1$ the function
 $E_p : [0,1] \to[0,\infty)$ defined by
\[
E_p(t) := \int (\rho^\eps_t)^p\,\d\mm,
\]
belongs to $C([0,1])\cap AC_{loc}((0,1))$. An application of the dominated convergence theorem grants that its derivative can be computed passing the limit inside the integral, obtaining
\[
\begin{split}
\frac{\d}{\d t}E_p(t) &=p\int (\rho^\eps_t)^{p-1}\frac{\d}{\d t}\rho^\eps_t\d\mm=-p\int (\rho^\eps_t)^{p-1}{\rm div}(\rho^\eps_t\nabla\vartheta_t)\,\d\mm\\
&=p\int \langle\nabla (\rho^\eps_t)^{p-1},\nabla\vartheta^\eps_t\rangle \rho^\eps_t\,\d\mm=(p-1)\int \langle\nabla (\rho^\eps_t)^{p},\nabla\vartheta^\eps_t\rangle \,\d\mm=-(p-1)\int (\rho^\eps_t)^{p}\Delta\vartheta^\eps_t \,\d\mm.
\end{split}
\]
Now notice that $\vartheta^\eps_t = \psi^\eps_t - \frac{\varepsilon}{2}\log\rho^\eps_t$ to get
\begin{equation}\label{eq:18}
\ddt E_p(t) = -(p-1)\int(\rho_t^\eps)^p\Delta\psi^\eps_t\d\mm + \frac{\varepsilon}{2}(p-1) \int (\rho^\eps_t)^p\Delta\log\rho^\eps_t\d\mm.
\end{equation}
Choosing $\delta:=\frac12$ in \eqref{eq:liyau} we get the existence of a constant $C>0$ depending on $K,N,D$ only such that $\Delta\psi^\eps_t\geq -C$ for any $t\in[0,\frac12]$, thus  we have
\[
-(p-1)\int (\rho_t^\eps)^p\Delta\psi^\eps_t\d\mm \leq -C(p-1)\int(\rho^\eps_t)^p\d\mm, \quad \forall t \in [0,1/2].
\]
On the other hand,
\[
\int (\rho^\eps_t)^p\Delta\log\rho^\eps_t\d\mm = -p\int(\rho^\eps_t)^{p-1}\langle \nabla \rho^\eps_t,\nabla \log\rho^\eps_t \rangle \d\mm = -p\int(\rho^\eps_t)^{p-2}|\nabla \rho^\eps_t|^2\d\mm \leq 0,
\]
so that  plugging these two inequalities into \eqref{eq:18} we obtain $E_p' \leq -C(p-1)E_p$ for all $t \in [0,1/2]$, whence by Gronwall's inequality
\[
E_p(t) \leq E_p(0)e^{-C(p-1)} , \quad \forall t \in [0,1/2].
\]
Passing to  the $p$-th roots and observing that, being  $\mm$  a probability measure, we have $\|h\|_{L^p}\uparrow\|h\|_{L^\infty}$ as $p\to\infty$, we obtain 
\[
\|\rho^\eps_t\|_{L^{\infty}} \leq e^{-C}\|\rho_0\|_{L^{\infty}}, \quad \forall t \in [0,1/2].
\]
Switching the roles of $\rho_0$ and $\rho_1$ we get the analogous control for $t\in[\frac12,1]$, whence the conclusion holds with $M := e^{-C}$.
\end{proof}

\begin{Proposition}[Uniform Lipschitz and Laplacian controls  for the potentials]\label{pro:1} With the same assumptions and notations as in Setting \ref{set} the following holds.

For all $\delta \in\, (0,1)$ there exists $C_\delta>0$ which only depends on $K,N,D,\delta$ such that
\begin{subequations}
\begin{align}
\label{eq:lipcontr}
\Lip(\varphi^\eps_t)&\leq C_\delta\\
\label{eq:lapcontr}
\Delta\varphi^\eps_t&\geq -C_\delta\\
\label{eq:lapcontr2}
\|\Delta\varphi^\eps_t\|_{L^1(\mm)}&\leq C_\delta
\end{align}
\end{subequations}
for every $t\in[\delta,1]$ and $\eps\in(0,1)$. Analogous bounds hold for the $\psi_t^\eps$'s in the time interval $[0,1-\delta]$.

If moreover $\rho_0,\rho_1 \in \testipp{\X}$, then we can take $\delta=0$ in the Lipschitz estimate \eqref{eq:lipcontr} above.
\end{Proposition}
\begin{proof} Fix $\delta\in(0,1)$ and notice that the bound \eqref{eq:ham}  yields
\[
\||\nabla\varphi^\eps_t|\|_{L^\infty}=\eps\||\nabla\log\h_{\frac{\eps t}2f^\eps|}\|_{L^\infty}\leq C\qquad\forall t\in[\delta,1],\ \eps\in(0,1).
\]
Thus recalling the Sobolev-to-Lipschitz property \eqref{eq:sobtolip} we obtain the bound \eqref{eq:lipcontr}. The bound \eqref{eq:lapcontr} is a restatement of \eqref{eq:liyau} and \eqref{eq:lapcontr2} comes from \eqref{eq:lapcontr} and the trivial identity
\[
\int|\Delta\varphi^\eps_t|\,\d\mm=\int \Delta\varphi^\eps_t\,\d\mm+2\int(\Delta\varphi^\eps_t)^-\,\d\mm=2\int(\Delta\varphi^\eps_t)^-\,\d\mm.
\]
The bounds for $\psi^\eps_t$ are obtained in the same way.

For the last part of the statement, notice that we have just proved that $\sup_{\eps\in(0,1)}\Lip(\psi^\eps_0)<\infty$ which together with the identity $\varphi^\eps_0+\psi^\eps_0=\eps\log\rho_0$ and the assumption on $\rho_0$ ensures that $\sup_{\eps\in(0,1)}\Lip(\varphi^\eps_0)<\infty$. The claim then follows from Proposition \ref{pro:kyoto} and a symmetric argument provides the conclusion for the $\psi^\eps_t$'s.
\end{proof}

\subsection{The entropy along entropic interpolations}

L\'eonard computed in \cite{Leonard13} the first and second derivatives of the relative entropy along entropic interpolations: here our first goal is to verify that his computations are fully justifiable in our setting. As we shall see later on, these formulas will be the crucial tool for showing that the acceleration of the entropic interpolation goes to 0 in a suitable sense.

\begin{Proposition}\label{pro:5}
With the same assumptions and notations as in Setting \ref{set} the following holds.

For any $\eps>0$ the map $t\mapsto H(\mu^{\varepsilon}_t \,|\, \mm)$ belongs to $C([0,1])\cap C^2(0,1)$ and for every $t\in(0,1)$ it holds
\begin{subequations}
\begin{align}
\label{eq:firstder}
\frac{\d}{\d t}H(\mu^{\varepsilon}_t \,|\, \mm) &= \int\la\nabla \rho^\eps_t,\nabla\vartheta^\eps_t\ra\,\d\mm=\frac1{2\eps}\int\big(|\nabla\psi^\eps_t|^2-|\nabla\varphi^\eps_t|^2\big)\rho^\eps_t\,\d\mm,\\
\label{eq:secondder}
 \frac{\d^2}{\d t^2}H(\mu^{\varepsilon}_t \,|\, \mm) & =\int\rho^\eps_t\,\d\big(\Ggamma_2(\vartheta^\eps_t)+\tfrac{\eps^2}4\Ggamma_2(\log(\rho^\eps_t))\big)= \frac{1}{2}\int \rho^\eps_t\,\d\big(\Ggamma_2(\varphi^{\varepsilon}_t) + \Ggamma_2(\psi^{\varepsilon}_t)\big).
\end{align}
\end{subequations}
If in addition $\rho_0,\rho_1\in\testipp \X$, then $t\mapsto H(\mu^{\varepsilon}_t \,|\, \mm)$ belongs to $C^2([0,1])$ and the above formulas are valid for any $t\in[0,1]$.
\end{Proposition}
\begin{proof}The continuity of $[0,1]\ni t\mapsto H(\mu^{\varepsilon}_t \,|\, \mm)$ is a direct consequence of the fact that $(\rho^\eps_t)\in C([0,1],L^2(\X))$ (Proposition \ref{pro:7}) and the equality of the two expressions for both the first and second derivative follows from $\vartheta^\eps_t=\frac{\psi^\eps_i-\varphi^\eps_t}2$, $\eps\log\rho^\eps_t=\psi^\eps_t+\varphi^\eps_t$ and the fact that $\Ggamma_2(\cdot)$ is a quadratic form.

Now fix $\eps>0$ and recall from Proposition \ref{pro:7} that $(\rho^\eps_t)\in AC_{loc}((0,1),L^2(\X))$ and that it is, locally in $t\in(0,1)$, uniformly bounded in $L^\infty$. Therefore for $u(z):=z\log z$ we have that $(0,1)\ni t\mapsto u(\rho^\eps_t)\in L^{2}(\X)$ is absolutely continuous. In particular, so is $\int u(\rho^\eps_t)\,\d\mm$ and it is then clear that
\[
\ddt H(\mu^{\varepsilon}_t \,|\, \mm) =\ddt\int  u(\rho^\eps_t)\,\d\mm=\int(\log(\rho^\eps_t)+1)\ddt\rho_t^\eps\,\d\mm,\qquad{\rm a.e.}\ t.
\]
Using the formula for $\ddt\rho^\eps_t$ provided by Proposition \ref{pro:7} we then get
\[
\begin{split}
\ddt H(\mu^{\varepsilon}_t \,|\, \mm) &=-\int (\log(\rho^\eps_t)+1){\rm div}(\rho_t^\eps\nabla\vartheta^\eps_t)\d\mm=\int\langle\nabla\log(\rho^\eps_t),\nabla\vartheta^\eps_t\rangle\rho^\eps_t\d\mm=\int\la\nabla \rho^\eps_t,\nabla\vartheta^\eps_t\ra\d\mm,
\end{split}
\]
thus establishing the formula for the first derivative at least for a.e.\ $t$. Now notice that since $(\rho^\eps_t),(\vartheta^\eps_t)\in AC_{loc}((0,1),W^{1,2})$ (by Proposition \ref{pro:7}), the rightmost term is an absolutely continuous function of time. In particular $t\mapsto H(\mu^{\varepsilon}_t \,|\, \mm) $ is $C^1$, the last formula holds for any $t\in(0,1)$, $t\mapsto \ddt H(\mu^{\varepsilon}_t \,|\, \mm) $ is absolutely continuous and  for a.e.\ $t$ it holds
\[
\frac{\d^2}{\d t^2} H(\mu^{\varepsilon}_t \,|\, \mm) =\ddt \int\la\nabla \rho^\eps_t,\nabla\vartheta^\eps_t\ra\,\d\mm=\int \langle\nabla\ddt \rho^\eps_t,\nabla\vartheta^\eps_t\rangle+\langle\nabla \rho^\eps_t,\nabla\ddt\vartheta^\eps_t\rangle\,\d\mm.
\]
Thus from the formulas for $\ddt\rho^\eps_t,\ddt\vartheta^\eps_t$ provided in Proposition \ref{pro:7} we obtain 
\[
\begin{split}
\frac{\d^2}{\d t^2}& H(\mu^{\varepsilon}_t \,|\, \mm)\\
&=\int-\la\nabla({\rm div}(\rho^\eps_t\nabla\vartheta^\eps_t)),\nabla\vartheta^\eps_t\ra +\langle\nabla\rho^\eps_t,\nabla\big(-\tfrac12|\nabla\vartheta^\eps_t|^2-\tfrac{\eps^2}{4}\Delta\log(\rho^\eps_t)-\tfrac{\eps^2}8|\nabla\log(\rho^\eps_t)|^2\big)\rangle\,\d\mm\\
&=\int -\rho^\eps_t\la\nabla\vartheta^\eps_t,\nabla\Delta\vartheta^\eps_t\ra+\tfrac12\Delta\rho^\eps_t|\nabla\vartheta^\eps_t|^2-\tfrac{\eps^2}{4}\langle\underbrace{\nabla\rho^\eps_t}_{=\rho^\eps_t\nabla\log\rho^\eps_t},\nabla \Delta\log(\rho^\eps_t)\rangle+\tfrac{\eps^2}8\Delta\rho^\eps_t|\nabla\log(\rho^\eps_t)|^2\,\d\mm\\
&=\int\rho^\eps_t\,\d\big(\Ggamma_2(\vartheta^\eps_t)+\tfrac{\eps^2}4\Ggamma_2(\log(\rho^\eps_t))\big)
\end{split}
\]
for a.e.\ $t$, so that \eqref{eq:secondder} is satisfied for a.e.\ $t$.

To obtain $C^2$ regularity and that the formula for the second derivative is valid for any $t$ it is sufficient to check that both $t\mapsto \int \rho^\eps_t\,\d\Ggamma_2(\varphi^\eps_t)$ and $t\mapsto \int \rho^\eps_t\,\d\Ggamma_2(\psi^\eps_t)$ are continuous. We have
\[
\int  \rho^\eps_t\,\d\Ggamma_2(\varphi^\eps_t)=-\int\frac12\la\nabla\rho^\eps_t,\nabla(|\nabla\varphi^\eps_t|^2)\ra\,\d\mm-\int\rho^\eps_t\la\nabla\varphi^\eps_t,\nabla\Delta\varphi^\eps_t\ra\,\d\mm
\]
and recalling the regularity properties stated in Proposition \ref{pro:7}, we see that the first integral is continuous because $(\rho_t^\eps)\in AC_{loc}((0,1),W^{1,2}(\X))$ and $(|\nabla\varphi^\eps_t|^2)\in C((0,1),L^{2}(\X))\cap L^\infty_{loc}((0,1),W^{1,2}(\X))$ while the continuity of the second comes from $(\rho^\eps_t)\in C((0,1),L^2(\X))\cap L^\infty_{loc}((0,1),L^\infty(\X))$, $(\varphi^\eps_t)\in AC_{loc}((0,1),W^{1,2}(\X))$ and, as it is readily verified, $(\Delta\varphi^\eps_t)\in C((0,1),L^{2}(\X))\cap L^\infty_{loc}((0,1),W^{1,2}(\X))$.

The last claim follows by the same arguments and the additional regularity ensured by the last part of Proposition \ref{pro:7}.
\end{proof}
As a first consequence of the formulas just obtained, we show that some quantities remain bounded as $\eps\downarrow0$:
\begin{Lemma}[Bounded quantities]\label{lem:7} With the same assumptions and notations of Setting \ref{set} we have
\begin{equation}
\label{eq:boundenergia}
\sup_{\eps\in(0,1)}\iint_0^1\Big(|\nabla\vartheta^\eps_t|^2+{\eps^2}|\nabla\log\rho^\eps_t|^2\Big)\rho^\eps_t\,\d t\,\d\mm<\infty\\
\end{equation}
and for any $\delta\in(0,\frac12)$
\begin{subequations}
\begin{align}
\label{eq:bhess}
& \sup_{\varepsilon \in\, (0,1)} \iint_\delta^{1-\delta}\Big(|\H{\vartheta^\eps_t}|_\HS^2+  {\eps^2} |\H{\log\rho_t^{\varepsilon}}|^2_{\HS}\Big)\rho^{\varepsilon}_t\,\d t\,\d\mm < \infty,  \\
 \label{eq:blap}
 &\sup_{\varepsilon \in\, (0,1)} \iint_\delta^{1-\delta}\Big(|\Delta{\vartheta^\eps_t}|^2+ {\eps^2} |\Delta{\log\rho_t^{\varepsilon}}|^2\Big)\rho^{\varepsilon}_t\,\d t\,\d\mm < \infty.  
\end{align}
\end{subequations}
If $\rho_0,\rho_1\in\testipp \X$, then we can take $\delta=0$.
\end{Lemma}
\begin{proof} We start with \eqref{eq:boundenergia} and recall that Proposition \ref{pro:1} grants that  
\[
\sup_{\eps\in(0,1)}\sup_{t\in[\frac12,1]}\Lip(\varphi^\eps_t)<\infty\qquad\sup_{\eps\in(0,1)}\sup_{t\in[0,\frac12]}\Lip(\psi^\eps_t)<\infty,
\] 
so that trivially 
\begin{equation}
\label{eq:partefacile}
\sup_{\eps\in(0,1)}\iint_{\frac12}^1|\nabla\varphi^\eps_t|^2\rho^\eps_t\,\d t\,\d\mm+\iint_0^{\frac12}|\nabla\psi^\eps_t|^2\rho^\eps_t\,\d t\,\d\mm<\infty.
\end{equation}
Now notice that \eqref{eq:firstder} gives
\[
\begin{split}
\iint_0^{\frac12}|\nabla\varphi^\eps_t|^2\rho^\eps_t\,\d t\,\d\mm&=\iint_0^{\frac12}|\nabla\psi^\eps_t|^2\rho^\eps_t\,\d t\,\d\mm-2\eps\int_0^{\frac12}\ddt H(\mu^\eps_t\,|\,\mm)\,\d t\\
&=\iint_0^{\frac12}|\nabla\psi^\eps_t|^2\rho^\eps_t\,\d t\,\d\mm+2\eps\big(H(\mu_0\,|\,\mm)-H(\mu^\eps_t\,|\,\mm)\big)
\end{split}
\]
so that taking into account the non-negativity of the relative entropy and \eqref{eq:partefacile} we see that the right hand side is uniformly bounded for $\eps\in(0,1)$. Using again \eqref{eq:partefacile} we deduce that
\[
\sup_{\eps\in(0,1)}\iint_0^{1}|\nabla\varphi^\eps_t|^2\rho^\eps_t\,\d t\,\d\mm<\infty.
\]
A symmetric argument provides the analogous bound for $(\psi^\eps_t)$ and thus recalling that $\vartheta^\eps_t=\frac12(\psi^\eps_t-\varphi^\eps_t)$ and $\eps\log\rho^\eps_t=\psi^\eps_t+\varphi^\eps_t$ we obtain \eqref{eq:boundenergia}.

Now use the fact that $\vartheta^{\varepsilon}_t = -\varphi^{\varepsilon}_t + \frac{\varepsilon}{2}\log\rho^{\varepsilon}_t$  in conjunction with  \eqref{eq:firstder} to get
\[
\begin{split}
\ddt H(\mu^\eps_t\,|\,\mm)\restr{t=\delta}&=- \int  \langle \nabla\rho_{\delta}^\eps, \nabla\varphi_{\delta}^{\varepsilon} \rangle \d\mathfrak{m} + \frac{\varepsilon}{2} \int \langle \nabla\rho_{\delta}^\eps, \nabla\log\rho_{\delta}^{\varepsilon} \rangle \d\mathfrak{m}\\
&=\int  \rho^\eps_{\delta} \Delta\varphi_{\delta}^{\varepsilon}  \d\mathfrak{m} + \frac{\varepsilon}{2} \int\frac{| \nabla\rho_{\delta}^\eps|^2}{\rho^\eps_\delta} \d\mathfrak{m}\geq \int  \rho_{\delta}^\eps \Delta\varphi_{\delta}^{\varepsilon}  \d\mathfrak{m}.
\end{split}
\]
Recalling the bound \eqref{eq:lapcontr} we get that for some constant $C_\delta$ independent on $\eps$ it holds
\[
\ddt H(\mu^\eps_t\,|\,\mm)\restr{t=\delta}\geq -C_\delta\qquad\forall \eps\in(0,1)
\]
and an analogous argument starting from  $\vartheta^\eps_t=\psi^\eps_t-\frac\eps2\log\rho^\eps_t$  yields
$
\ddt H(\mu^\eps_t\,|\,\mm)\restr{t=1-\delta}\leq C_\delta$ for every $\eps\in(0,1)$.
Therefore
\[
\sup_{\eps\in(0,1)}\int_\delta^{1-\delta}\frac{\d^2}{\d t^2} H(\mu^\eps_t\,|\,\mm)=\sup_{\eps\in(0,1)}\bigg(\ddt H(\mu^\eps_t\,|\,\mm)\restr{t=1-\delta}-\ddt H(\mu^\eps_t\,|\,\mm)\restr{t=\delta}\bigg)<\infty.
\]
The bounds \eqref{eq:bhess} and \eqref{eq:blap} then come from this last inequality used in conjunction with \eqref{eq:boundenergia} and the Bochner inequality written as in \eqref{eq:bochhess} and \eqref{eq:bochlap} respectively.

For the last claim we recall  that under the further regularity assumptions on $\rho_0,\rho_1$ we have that
\[
\Big|\ddt H(\mu^\eps_t\,|\,\mm)\restr{t=0}\Big|=\Big|\int\la\nabla\rho_0,\nabla\vartheta^\eps_0\ra\,\d\mm\Big|\leq \frac12\Lip(\rho_0)\big(\Lip(\varphi^\eps_0)+\Lip(\psi^\eps_0)\big)
\]
and using the uniform Lipschitz bounds given in the last part of  Proposition \ref{pro:1} we obtain that $\sup_{\eps\in(0,1)}\big|\ddt H(\mu^\eps_t\,|\,\mm)\restr{t=0}\big|<\infty$. A similar argument provides a  uniform bound on $\big|\ddt H(\mu^\eps_t\,|\,\mm)\restr{t=1}\big|$ and then we conclude as before.
\end{proof}
With the help of the previous lemma we can now prove that some crucial quantities vanish in the limit $\eps\downarrow0$; as we shall see in the proof of our main theorem \ref{thm:main}, this is what we will need to prove that the acceleration of the entropic interpolations goes to 0 as $\eps$ goes to zero. 
\begin{Lemma}[Vanishing quantities]\label{le:vanish}With the same assumptions and notations of Setting \ref{set}, for any $\delta\in(0,\frac12)$ we have
\begin{subequations}
\begin{align}
\label{eq:b1}
&\lim_{\varepsilon \downarrow 0}\varepsilon^2 \iint_\delta^{1-\delta} \rho^{\varepsilon}_t |\Delta\log\rho^{\varepsilon}_t|  \,\d t\,  \d\mm= 0, \\
\label{eq:b2}
&\lim_{\varepsilon \downarrow 0}\varepsilon^2 \int_\delta^{1-\delta} \rho^{\varepsilon}_t |\nabla\log\rho^{\varepsilon}_t|^2 \,\d t\,  \d\mm = 0, \\
\label{eq:b3}
&\lim_{\varepsilon \downarrow 0}\varepsilon^2 \iint_\delta^{1-\delta} \rho^{\varepsilon}_t |\Delta\log\rho^{\varepsilon}_t| |\nabla\log\rho^{\varepsilon}_t| \,\d t\,  \d\mm = 0, \\
\label{eq:b4}
 &\lim_{\varepsilon \downarrow 0} \varepsilon^2 \iint_\delta^{1-\delta}\rho^{\varepsilon}_t |\nabla\log\rho^{\varepsilon}_t|^3 \,\d t\,  \d\mm= 0.
\end{align}
\end{subequations}
If in addition $\rho_0,\rho_1\in\testipp \X$, then we can take $\delta=0$.
\end{Lemma}
\begin{proof}
For \eqref{eq:b1} we notice that
\[
\varepsilon^2 \iint_\delta^{1-\delta} \rho^{\varepsilon}_t |\Delta\log\rho^{\varepsilon}_t|\,\d t\,  \d\mm\leq\eps\sqrt{1-2\delta}\sqrt{\varepsilon^2 \iint_\delta^{1-\delta} \rho^{\varepsilon}_t |\Delta\log\rho^{\varepsilon}_t|^2  \,\d t\,  \d\mm}
\]
and the fact that, by \eqref{eq:blap}, the last square root is uniformly bounded in $\eps\in(0,1)$. 

For \eqref{eq:b2} we start from the identity $\rho^{\varepsilon}_t|\nabla\log\rho^{\varepsilon}_t|^2 = -\rho^{\varepsilon}_t \Delta \log \rho^{\varepsilon}_t + \Delta\rho^{\varepsilon}_t$,  use the fact that $\int\Delta\rho^{\varepsilon}_t\,\d\mm=0$ to get
\[
\varepsilon^2\iint_{\delta}^{1-\delta}\rho^{\varepsilon}_t|\nabla\log\rho^{\varepsilon}_t|^2 \,\d t\,  \d\mm = -\varepsilon^2\iint_{\delta}^{1-\delta}\rho^{\varepsilon}_t \Delta\log\rho^{\varepsilon}_t\,\d t\,  \d\mm \leq \varepsilon^2 \iint_{\delta}^{1-\delta}  \rho^{\varepsilon}_t|\Delta\log\rho^{\varepsilon}_t|\,\d t\,  \d\mm
\]
and then conclude by \eqref{eq:b1}. 

For \eqref{eq:b3} we observe that
\[
\begin{split}
\varepsilon^2 \iint_\delta^{1-\delta} \rho^{\varepsilon}_t &|\Delta\log\rho^{\varepsilon}_t| |\nabla\log\rho^{\varepsilon}_t| \,\d t\,  \d\mm\\
&\leq \sqrt{\varepsilon^2 \iint_\delta^{1-\delta} \rho^{\varepsilon}_t |\Delta\log\rho^{\varepsilon}_t|^2 \,\d t\,  \d\mm}\sqrt{\varepsilon^2 \iint_\delta^{1-\delta} \rho^{\varepsilon}_t |\nabla\log\rho^{\varepsilon}_t|^2 \,\d t\,  \d\mm},
\end{split}
\]
and use the fact that the first square root in the right hand side is bounded (by \eqref{eq:blap}) and the second one goes to 0 (by \eqref{eq:b2}). 

To prove \eqref{eq:b4} we start again from   the identity $\rho^{\varepsilon}_t|\nabla\log\rho^{\varepsilon}_t|^2 = -\rho^{\varepsilon}_t \Delta \log \rho^{\varepsilon}_t + \Delta\rho^{\varepsilon}_t$ to get
\[
 \iint_\delta^{1-\delta} \rho^{\varepsilon}_t |\nabla\log\rho^{\varepsilon}_t|^3 \,\d t\,\d\mm=-\iint_\delta^{1-\delta} \rho^{\varepsilon}_t\Delta\log\rho^{\varepsilon}_t |\nabla\log\rho^{\varepsilon}_t| \,\d t\,\d\mm +  \iint_\delta^{1-\delta} \Delta\rho^{\varepsilon}_t |\nabla\log\rho^{\varepsilon}_t|\,\d t\,\d\mm.
\]
After a multiplication by $\eps^2$ we see that the first integral on the right-hand side vanishes as $\varepsilon \downarrow 0$ thanks to \eqref{eq:b3}. For the second we start noticing that an application of the dominated convergence theorem ensures that
\begin{equation}
\label{eq:inpiu}
\iint_\delta^{1-\delta}\Delta\rho^{\varepsilon}_t |\nabla\log\rho^{\varepsilon}_t| \,\d t\,  \d\mm =\lim_{\eta\downarrow 0}\iint_\delta^{1-\delta}\Delta\rho^{\varepsilon}_t \sqrt{\eta+|\nabla\log\rho^{\varepsilon}_t|^2} \,\d t\,  \d\mm,
\end{equation}
then we observe that for every $\eta>0$ the map $z\mapsto \sqrt{\eta+z}$ is in $C^1([0,\infty))$ and since $|\nabla\log\rho^{\varepsilon}_t|^2\in W^{1,2}(\X)$ we deduce that $\sqrt{\eta+|\nabla\log\rho^{\varepsilon}_t|^2} \in W^{1,2}(\X)$ as well. Thus by the chain rule for gradients and the Leibniz rule \eqref{eq:leibh} it holds
\[
\begin{split}
\bigg|\iint_\delta^{1-\delta}\Delta\rho^{\varepsilon}_t& \sqrt{\eta+|\nabla\log\rho^{\varepsilon}_t|^2} \,\d t\,  \d\mm\bigg|\\
&=\bigg|\iint_\delta^{1-\delta}\frac{\rho^\eps_t}{2\sqrt{\eta+|\nabla\log\rho^\eps_t|^2}}\langle\nabla\log\rho^\eps_t,\nabla|\nabla\log\rho^\eps_t|^2\rangle \,\d t\,  \d\mm\bigg|\\
&=\bigg|\iint_\delta^{1-\delta}\frac{\rho^\eps_t}{\sqrt{\eta+|\nabla\log\rho^\eps_t|^2}}\H{\log\rho^{\varepsilon}_t}(\nabla\log\rho^{\varepsilon}_t,\nabla\log\rho^{\varepsilon}_t) \,\d t\,  \d\mm\bigg|\\
&\leq\iint_\delta^{1-\delta}\rho^\eps_t|\H{\log\rho^{\varepsilon}_t}|_{\HS}\,|\nabla\log\rho^{\varepsilon}_t|\,\d t\,  \d\mm
\end{split}
\]
and being this true for any $\eta>0$, from \eqref{eq:inpiu} we obtain
\[
\begin{split}
\varepsilon^2\bigg|\iint_\delta^{1-\delta}\Delta\rho^{\varepsilon}_t |\nabla\log\rho^{\varepsilon}_t| \,\d t\,  \d\mm\bigg|  
& \leq  \varepsilon^2\iint_\delta^{1-\delta} \rho^{\varepsilon}_t |\H{\log\rho^{\varepsilon}_t}|_{\HS}|\nabla\log\rho^{\varepsilon}_t|\,\d t\,  \d\mm \\ 
& \leq  \sqrt{\varepsilon^2\iint_\delta^{1-\delta} \rho^{\varepsilon}_t |\H{\log\rho^{\varepsilon}_t}|^2_{\HS} \,\d t\,  \d\mm }\\
&\qquad\qquad\qquad\qquad\qquad\times \sqrt{\varepsilon^2\iint_\delta^{1-\delta} \rho^{\varepsilon}_t |\nabla\log\rho^{\varepsilon}_t|^2 \,\d t\,  \d\mm} .
\end{split}
\]
In this last expression the first square root is uniformly bounded in $\varepsilon\in(0,1)$ by \eqref{eq:bhess}, while the second one vanishes as $\varepsilon \downarrow 0$ thanks to \eqref{eq:b2}. 

The last claim follows from the fact that under the stated additional regularity properties of $\rho_0,\rho_1$ we can take $\delta=0$ in \eqref{eq:bhess}, \eqref{eq:blap}. Then we argue as before.
\end{proof}

\section{From entropic to displacement interpolations}\label{sec:6}

\subsection{Compactness}

Starting from the uniform estimates discussed in Section \ref{sec:5}, let us first prove that when we pass to the limit as $\varepsilon \downarrow 0$, up to subsequences Schr\"odinger potentials and entropic interpolations converge in a suitable sense to limit potentials and interpolations. 
\begin{Proposition}[Compactness]\label{lem:6} With the same assumptions and notations as in Setting \ref{set} the following holds.

For any sequence $\eps_n\downarrow0$ there exists a subsequence, not relabeled, so that:
\begin{itemize}
\item[(i)] the curves $(\mu^{\eps_n}_t)$ uniformly converge in $(\prob{\X},W_2)$ to a limit curve  $(\mu_t)$ which belongs to $AC([0,1],(\prob{\X},W_2))$. Moreover, there is $M>0$ so that
\begin{equation}
\label{eq:bdens}
\mu_t\leq M\mm\qquad\forall t\in[0,1]
\end{equation}
and setting $\rho_t:=\frac{\d\mu_t}{\d\mm}$ it holds
\begin{equation}\label{eq:66}
\rho^{\varepsilon_n}_t \stackrel{\ast}{\rightharpoonup} \rho_t \quad \textrm{in } L^{\infty}(\mm)\qquad\forall t\in[0,1].
\end{equation}
\item[(ii)] the curves $(\varphi^{\eps_n}_t),(\psi^{\eps_n}_t)$ converge locally uniformly on $I$ with values in $ L^1(\X)$ to limit curves $(\varphi_t),(\psi_t)\in  AC_{loc}(I,L^1(\X))$ with $\Lip(\varphi_t),\Lip(\psi_t)$ locally bounded for $t\in I$, where $I:=(0,1]$ for the $\varphi$'s and $I:=[0,1)$ for the $\psi$'s. Moreover for every $t\in(0,1)$ it holds 
\begin{equation}\label{eq:63}
\begin{split}
\varphi_t+\psi_t&\leq 0\qquad \text{ on }\X,\\
\varphi_t +\psi_t&=0 \qquad \text{ on }\supp(\mu_t).
\end{split}
\end{equation}
Similarly, the curves $(\vartheta^{\eps_n}_t)$ converge  in $(0,1)$ to the limit curve $t\mapsto\vartheta_t:=\frac12(\psi_t-\varphi_t)$ in the same sense as above.

If  we further assume that $\rho_0,\rho_1 \in \testipp{\X}$, we can pick $I:=[0,1]$ for both the $\varphi$'s and the $\psi$'s and \eqref{eq:63} holds for all $t\in[0,1]$.
\end{itemize}
\end{Proposition}
\begin{proof}

\noindent{\bf (i)} Fix $\eps\in(0,1)$; we want to apply Theorem \ref{thm:GH} to   $(\mu^\eps_t)$ and $(\nabla\vartheta^\eps_t)$. The continuity of $t\mapsto\rho^\eps_t\in L^2(\X)$ granted by Proposition \ref{pro:7} yields weak continuity of $(\mu_t)$ and the uniform $L^\infty$-bound \eqref{eq:linftybound} gives \eqref{eq:bih}. From the bound \eqref{eq:boundenergia} it follows \eqref{eq:bkh} and from the formula for $\frac\d{\d t}\rho^\eps_t$ given in Proposition \ref{pro:7} and again the $L^2$-continuity of $(\rho^\eps_t)$ on $[0,1]$  it easily follows that $(\mu_t)$ and $(\vartheta^\eps_t)$ solve the continuity equation in the sense of Theorem \ref{thm:GH}. The conclusion of such theorem     ensures that $(\mu^\eps_t)$ is $W_2$-absolutely continuous with 
\[
\int_0^1|\dot\mu^\eps_t|^2\,\d t=\iint_0^1|\nabla\vartheta^\eps_t|^2\rho^\eps_t\,\d t\,\d\mm.
\]
The bound \eqref{eq:boundenergia} grants that the right hand side is uniformly bounded in $\eps\in(0,1)$ and since $(\prob X,W_2)$ is compact, this  is sufficient to ensure the compactness of the family $\{(\mu^\eps_t)\}_\eps$ in $C([0,1],(\prob X,W_2))$ and, by the lower semicontinuity of the kinetic energy, the fact that any limit curve $(\mu_t) $ is absolutely continuous. The bound \eqref{eq:bdens} is then a direct consequence of the uniform bound \eqref{eq:linftybound} and the convergence property \eqref{eq:66} comes from the weak convergence of the measures and the uniform bound on the densities.

\noindent{\bf (ii)} From the formula for $\ddt\varphi^\eps_t$ provided in Proposition \ref{pro:7} we obtain
\[
\|\varphi^{\varepsilon}_t - \varphi^{\varepsilon}_s\|_{L^1(\mm)} \leq \iint_t^s   \frac{|\nabla\varphi^{\varepsilon}_r|^2}{2} + \frac{\varepsilon}{2}|\Delta\varphi^{\varepsilon}_r|\,\d r\,\d\mm\qquad\forall\eps>0,\ \forall t,s\in(0,1],\ t<s.
\]
Thus for  $\delta\in(0,1)$ the estimates \eqref{eq:lipcontr} and \eqref{eq:lapcontr2} give
\begin{equation}
\label{eq:lipl1}
\|\varphi^{\varepsilon}_t - \varphi^{\varepsilon}_s\|_{L^1(\mm)} \leq C'_\delta|s-t |\qquad\forall\eps\in(0,1),\ \forall t,s\in[\delta,1],\ t<s.
\end{equation}
Now notice that  for $h\in {\rm LIP}(X)$ and $\mu\in\prob \X$, integrating in $y$ w.r.t.\ $\mu$ the trivial inequality $h(x)\leq h(y)+D\Lip(h)$ yields
\[
h(x)^+\leq\Big(\int h\,\d\mu+D\Lip(h)\Big)^+\leq \Big|\int h\,\d\mu\Big|+D\Lip(h)
\]
and since a similar bound can be obtained for $h(x)^-$ we get
\begin{equation}
\label{eq:l1li}
\|h\|_{L^\infty(\mm)}\leq  \Big|\int h\,\d\mu\Big|+D\Lip(h)\leq \|h\|_{L^1(\mu)}+D\Lip(h).
\end{equation}
Choosing $\mu:=\mu_1$ and $h:=\varphi^\eps_1$ and recalling that the normalization chosen for $(f^\eps,g^\eps)$ in Setting \ref{set} reads as $\int \varphi^\eps_1\,\d\mu_1=0$, we deduce that $\{\varphi^\eps_1\}_{\eps\in(0,1)}$ is uniformly bounded in $L^\infty(\mm)$. Using this information together with \eqref{eq:lipl1} and \eqref{eq:l1li} with $\mu:=\mm$ we conclude that
\[
\sup_{\eps\in(0,1)}\sup_{t\in[\delta,1]}\|\varphi^\eps_t\|_{L^\infty(\mm)}<\infty.
\]
By Ascoli-Arzel\`a's theorem, this uniform bound and the equi-Lipschitz continuity in space given by \eqref{eq:lipcontr} together with the equi-Lipschitz continuity in time given by \eqref{eq:lipl1} give compactness in $C([\delta,1],L^1(\X))$; it is clear then  that any limit curve $(\varphi_t)$ belongs to ${\rm LIP}([\delta,1],L^1(\X))$ and  that $\sup_{t\in[\delta,1]}\Lip(\varphi_t)<\infty$. A diagonalization argument and the arbitrariness of $\delta\in(0,1)$ then provide the  required results on $(0,1]$.

The argument for the $\psi^\eps_t$'s follows the same lines provided we are able to show that for some $t\in[0,1)$ the functions $\psi^\eps_t$ are uniformly bounded. To see this, observe that from the estimate \eqref{eq:linftybound} it follows that 
\[
0\leq H(\mu^\eps_t\,|\,\mm)\leq \log M\qquad\forall \eps\in(0,1),\ t\in[0,1],
\]
thus multiplying the identity
\begin{equation}
\label{eq:k4}
\varphi^\eps_t+\psi^\eps_t=\eps\log\rho^\eps_t\qquad\forall t\in(0,1)
\end{equation}
by $\rho^\eps_t$ and integrating we get
\begin{equation}
\label{eq:k5}
0\leq \int\varphi^\eps_t+\psi^\eps_t\,\d\mu^\eps_t\leq\eps \log M\qquad\forall t\in(0,1).
\end{equation}
Since we know that $\varphi^\eps_{1/2}$ is uniformly bounded, this yields a uniform control on $\int \psi^\eps_{1/2}\,\d\mu^\eps_{1/2}$ and then we can proceed as before starting from \eqref{eq:l1li} with $h:=\psi^\eps_{1/2}$ and $\mu:=\mu^\eps_{1/2}$.

The claim for the $(\vartheta^\eps_t)$ is now obvious.

Finally, to prove the first in \eqref{eq:63} we pass to the limit in \eqref{eq:k4} recalling the uniform bound \eqref{eq:linftybound}, then passing to the limit in \eqref{eq:k5} (by uniform convergence of functions and weak convergence of measures) we deduce that
\[
\int\varphi_t+\psi_t\,\d\mu_t=0,
\]
which forces the second in \eqref{eq:63}. 

For the last claim, start recalling that under the additional regularity assumption by Proposition \ref{pro:7} we know that $(\varphi^\eps_t),(\psi^\eps_t)\in AC([0,1],W^{1,2})$ for every $\eps>0$. Then notice that for every $t_0,t_1\in[0,1]$ and $\eps\in(0,1)$ the uniform Lipschitz estimates granted by the last part of Proposition \ref{pro:1} ensure that
\[
\Big|\int \varphi^\eps_{t_1}-\varphi^\eps_{t_0}\,\d\mm\Big|=\Big|\iint_{t_0}^{t_1}\frac{|\nabla\varphi^\eps_t|^2}{2}+\frac\eps2\Delta\varphi^\eps_t\,\d t\,\d\mm\Big|=\iint_{t_0}^{t_1}\frac{|\nabla\varphi^\eps_t|^2}{2}\,\d t\,\d\mm\leq |t_1-t_0| C
\]
for some $C>0$ independent on $\eps$. Thus from \eqref{eq:l1li} with $h:=\varphi_{t_1}-\varphi_{t_0}$ and $\mu:=\mm$ we deduce that $\|\varphi_{t_1}-\varphi_{t_0}\|_{L^\infty(\mm)}$ is uniformly bounded in $t_0,t_1\in[0,1]$ and the conclusion follows along the same lines used before. Similarly for $(\psi^\eps_t)$.
\end{proof}

\subsection{Identification of the limit curve and potentials}
We now show that the limit interpolation is the geodesic from $\mu_0$ to $\mu_1$ and the limit potentials are Kantorovich potentials. We shall make use of the following simple lemma valid on general metric measure spaces: 
\begin{Lemma}\label{le:dermiste}
Let $(\Y,\sfd_\Y,\mm_\Y)$ be a complete separable metric measure space endowed with a non-negative measure $\mm_\Y$ which is finite on bounded sets and assume that $W^{1,2}(\Y)$ is separable. Let $\ppi$ be a test plan and $f\in W^{1,2}(\Y)$. Then $t\mapsto \int f\circ\e_t\,\d\ppi$ is absolutely continuous and 
\begin{equation}
\label{eq:bder1}
\Big|\frac\d{\d t}\int f\circ\e_t\,\d\ppi \Big|\leq \int |\d f|(\gamma_t)|\dot\gamma_t|\,\d\ppi(\gamma)\qquad{\rm a.e.}\ t\in[0,1],
\end{equation}
where the exceptional set can be chosen to be independent on $f$.

Moreover, if $(f_t)\in AC([0,1],L^2(\Y))\cap L^\infty([0,1],W^{1,2}(\Y))$, then the map $t\mapsto\int f_t\circ\e_t\,\d\ppi$ is also absolutely continuous and 
\[
\frac\d{\d s}\Big(\int f_s\circ\e_s\,\d\ppi\Big)\restr{s=t}=\int \big(\frac\d{\d s}f_s\restr{s=t}\big)\circ\e_t\,\d\ppi+\frac\d{\d s}\Big(\int f_t\circ\e_s\,\d\ppi\Big)\restr{s=t}\qquad {\rm a.e.}\ t\in[0,1].
\]
\end{Lemma}
\begin{proof} The absolute continuity of $t\mapsto \int f\circ\e_t\,\d\ppi$ and the bound \eqref{eq:bder1} are  trivial consequences of the definitions of test plans and Sobolev functions. The fact that the exceptional set can be chosen independently on $f$ follows from the separability of $W^{1,2}(\Y)$ and standard approximation procedures, carried out, for instance, in \cite{Gigli14}.

For the second part, we start noticing that the second derivative in the right hand side exists for a.e.\ $t$ thanks to what we have just proved, so that the claim makes sense. The  absolute continuity follows from the fact that for any $t_0,t_1\in[0,1]$, $t_0<t_1$ it holds
\[
\begin{split}
\Big|\int f_{t_1}\circ\e_{t_1}-f_{t_0}\circ\e_{t_0}\,\d\ppi\Big|&\leq \Big|\int f_{t_1}\circ\e_{t_1}-f_{t_1}\circ\e_{t_0}\,\d\ppi\Big|+\Big|\int f_{t_1}-f_{t_0}\,\d(\e_{t_0})_*\ppi\Big|\\
&\leq \iint_{t_0}^{t_1}|\d f_{t_1}|(\gamma_t)|\dot\gamma_t|\,\d t\,\d\ppi(\gamma)+\iint_{t_0}^{t_1}\Big|\ddt f_t\Big|\,\d t\,\d(\e_{t_0})_*\ppi
\end{split}
\]
and our assumptions on $(f_t)$ and $\ppi$. Now fix a point $t$ of differentiability for $(f_t)$ and observe that the fact that $\frac{f_{t+h}-f_t}{h}$ strongly converges in $L^2(\Y)$ to $\ddt f_t$ and $(\e_{t+h})_*\ppi$ weakly converges to $(\e_t)_*\ppi$ as $h\to 0$ and the densities are equibounded is sufficient to get
\[
\lim_{h\to 0}\int\frac{f_{t+h}-f_t}{h}\circ\e_{t+h}\,\d\ppi=\int \ddt f_t\circ\e_t\,\d\ppi=\lim_{h\to 0}\int\frac{f_{t+h}-f_t}{h}\circ\e_{t}\,\d\ppi.
\]
Hence the conclusion comes dividing by $h$ the trivial identity
\[
\begin{split}
\int f_{t+h}\circ\e_{t+h}-f_{t}\circ\e_{t}\,\d\ppi=&\int f_{t}\circ\e_{t+h}-f_{t}\circ\e_{t}\,\d\ppi+\int f_{t+h}\circ\e_{t}-f_{t}\circ\e_{t}\,\d\ppi+\\
&\qquad+\int (f_{t+h}-f_t)\circ\e_{t+h}-(f_{t+h}-f_t)\circ\e_{t}\,\d\ppi
\end{split}
\]
and letting $h\to 0$.
\end{proof}
We now prove that in the limit the potentials evolve according to the Hopf-Lax semigroup (recall  formula \eqref{eq:hli}).

\begin{Proposition}[Limit curve and potentials]\label{pro:9} With the same assumptions and notations as in Setting \ref{set} the following holds.

The limit curve $(\mu_t)$ given by Proposition \ref{lem:6} is unique (i.e.\ independent on the sequence $\eps_n\downarrow0$) and is the only $W_2$-geodesic connecting $\mu_0$ to $\mu_1$.

Any limit curve $(\varphi_t)$ given by Proposition \ref{lem:6} is in $AC_{loc}((0,1],C(\X))\cap L^\infty_{loc}((0,1],W^{1,2}(\X))$ and for any $ t_0,t_1\in(0,1]$, $t_0<t_1$ we have
\begin{subequations}
\begin{align}
\label{eq:hl1}
-\varphi_{t_1}&=Q_{t_1-t_0}(-\varphi_{t_0})\\
\label{eq:costointermedio}
\int \varphi_{t_0}\,\d\mu_{t_0}- \int \varphi_{t_1}\,\d\mu_{t_1}&= \frac1{2(t_1-t_0)}W_2^2(\mu_{t_0},\mu_{t_1})
\end{align}
\end{subequations}
and $-(t_1-t_0)\varphi_{t_1}$ is a Kantorovich potential from $\mu_{t_1}$ to $\mu_{t_0}$.
Similarly,  any limit curve  $(\psi_t)$ given by Proposition \ref{lem:6} belongs to  $AC_{loc}([0,1),C(\X))\cap L^\infty_{loc}([0,1),W^{1,2}(\X))$  and for every $ t_0,t_1\in[0,1)$,  $t_0<t_1$ we have
\begin{subequations}
\begin{align}
\label{eq:hl2}
-\psi_{t_0}&= Q_{t_1-t_0}(-\psi_{t_1})\\
\label{eq:costointermedio2}
\int \psi_{t_1}\,\d\mu_{t_1}- \int \psi_{t_0}\,\d\mu_{t_0}&= \frac1{2(t_1-t_0)}W_2^2(\mu_{t_0},\mu_{t_1})
\end{align}
\end{subequations}
and $-(t_1-t_0)\psi_{t_0}$ is a Kantorovich potential from $\mu_{t_0}$ to $\mu_{t_1}$.

Finally,  if we further assume that $\rho_0,\rho_1\in\testipp \X$ then the claimed properties of $(\varphi_t)$, $(\psi_t)$ hold for $t_0,t_1\in[0,1]$.
\end{Proposition}
\begin{proof}

\noindent{\bf Inequality $\leq$ in \eqref{eq:hl1}}. Pick $x,y \in \X$, $r>0 $,  define
\[
\nu^r_x := \frac{1}{\mm(B_r(x))}\mm\restr{B_r(x)} \qquad\qquad\qquad\nu^r_y := \frac{1}{\mm(B_r(y))}\mm\restr{B_r(y)}
\]
and $\ppi^r$ as the only lifting of the only $W_2$-geodesic from $\nu^r_x$ to $\nu^r_y)$ (recall point $(i)$ of Theorem \ref{thm:bm}). Let $\eps\in(0,1)$ and $0<t_0<t_1\leq 1$ and apply Lemma \ref{le:dermiste} to  $\ppi^r$ and  $t\mapsto\varphi^\eps_{(1-t)t_0+tt_1}\in W^{1,2}(\X)$ to get
\[
\begin{split}
\ddt\int \varphi^\eps_{(1-t)t_0+tt_1}\circ\e_t\,\d\ppi^r\geq \int(t_1-t_0)\frac{\d}{\d s}\varphi^\eps_s\restr{s={(1-t)t_0+tt_1}}(\gamma_t)-|\d \varphi^\eps_{(1-t)t_0+tt_1}|(\gamma_t)|\dot\gamma_t|\,\d\ppi^r(\gamma).
\end{split}
\]
Thus recalling the expression for $\ddt \varphi^\eps_t$ and using Young's inequality we obtain
\[
\ddt\int \varphi^\eps_{(1-t)t_0+tt_1}\circ\e_t\,\d\ppi^r\geq \int\eps\frac{t_1-t_0}2\Delta\varphi^\eps_{(1-t)t_0+tt_1}(\gamma_t)-\frac1{2(t_1-t_0)}|\dot\gamma_t|^2\,\d\ppi^r(\gamma).
\]
Integrating in time  and recalling that $\ppi^r$ is optimal we get
\[
\int\varphi^\eps_{t_1}\,\d\nu^r_y-\int\varphi_{t_0}^\eps\,\d\nu^r_x\geq -\frac1{2(t_1-t_0)}W_2^2(\nu^r_y,\nu^r_x)+\iint_0^1\eps\frac{t_1-t_0}2\Delta\varphi^\eps_{(1-t)t_0+tt_1}\circ\e_t\,\d t\,\d\ppi^r.
\]
Let  $\eps\downarrow0$ along the sequence $(\eps_n)$ for which $(\varphi^{\eps_n}_t)$ converges to our given $(\varphi_t)$ as in Proposition \ref{lem:6} and use the uniform bound \eqref{eq:lapcontr2} and the fact that $\ppi^r$ has bounded compression to deduce that
\[
\int\varphi_{t_1}\,\d\nu^r_y-\int\varphi_{t_0}\,\d\nu^r_x\geq -\frac1{2(t_1-t_0)}W_2^2(\nu^r_y,\nu^r_x)
\]
and finally letting $r\downarrow 0$  we conclude from the arbitrariness of $x\in\X$ that
\begin{equation}
\label{eq:o1}
-\varphi_{t_1}(y)\leq Q_{t_1-t_0}(-\varphi_{t_0})(y)\qquad\forall y\in \X.
\end{equation} 

\noindent{\bf Inequality $\geq$ in \eqref{eq:hl1}}. To prove the opposite inequality we fix again $0<t_0<t_1\leq 1$ and  apply Theorem \ref{thm:RLF} to the vector fields $((t_1-t_0)\nabla\varphi^\eps_{(1-t)t_1+tt_0})$: the bound \eqref{eq:lapcontr} ensures that the theorem is applicable and we obtain existence of the regular Lagrangian flow $F^\eps$. Put $\ppi^\eps:=(F^\eps_\cdot)_*\mm$, where $F^\eps_\cdot:\X\to C([0,1],\X)$ is the $\mm$-a.e.\ defined map which sends $x$ to $F_t^\eps(x)$, and observe that the  bound \eqref{eq:quantm} and the identity \eqref{eq:quants}  provided by Theorem \ref{thm:RLF} coupled with the estimates \eqref{eq:lipcontr}, \eqref{eq:lapcontr}  on $\nabla\varphi^\eps_t,\Delta\varphi^\eps_t$ ensure that $\ppi^\eps$ is a test plan with
\begin{equation}
\label{eq:uniformtest}
\sup_{\eps\in(0,1)}\iint_0^1|\dot\gamma_t|^2\,\d t\,\d\ppi^\eps(\gamma)<\infty\qquad\text{and}\qquad (\e_t)_*\ppi^\eps\leq C\mm\qquad\forall t\in[0,1],\ \eps\in(0,1),
\end{equation}
for some $C<\infty$. Thus by Lemma \ref{le:dermiste} applied to  $\ppi^\eps$ and $t\mapsto \varphi^\eps_{(1-t)t_1+tt_0}$  we obtain
\[
\begin{split}
\ddt\int &\varphi^\eps_{(1-t)t_1+tt_0}\circ\e_t\,\d\ppi^\eps\\
&= \int(t_0-t_1)\frac{\d}{\d s}\varphi^\eps_s\restr{s={(1-t)t_1+tt_0}}\circ\e_t\,\d\ppi^\eps+\frac\d{\d s}\int\varphi^\eps_{(1-t)t_1+tt_0}\circ\e_s\,\d\ppi^\eps\restr{s=t}\\
&= \int\Big( \frac{t_0-t_1}{2}|\d\varphi^\eps_{(1-t)t_1+tt_0}|^2 +\eps\frac{t_0-t_1}{2}\Delta\varphi^\eps_{(1-t)t_1+tt_0} +(t_1-t_0)|\d\varphi^\eps_{(1-t)t_1+tt_0}|^2\Big)\circ\e_t\,\d\ppi^\eps\\
&= \int\Big( \frac{t_1-t_0}{2}|\d\varphi^\eps_{(1-t)t_1+tt_0}|^2 +\eps\frac{t_0-t_1}{2}\Delta\varphi^\eps_{(1-t)t_1+tt_0} \Big)\circ\e_t\,\d\ppi^\eps.
\end{split}
\]
Integrating in time and recalling \eqref{eq:quants} we deduce
\begin{equation}
\label{eq:k6}
\int \varphi_{t_0}^\eps\circ\e_1- \varphi_{t_1}^\eps\circ\e_0\,\d\ppi^\eps=\iint_0^1\frac1{2(t_1-t_0)}|\dot\gamma_t|^2+\eps\frac{t_0-t_1}{2}\Delta\varphi^\eps_{(1-t)t_1+tt_0}(\gamma_t)\,\d t\,\d\ppi^\eps(\gamma).
\end{equation}
Now, as before, we let $\eps\downarrow0 $  along the sequence $(\eps_n)$ for which $(\varphi^{\eps_n}_t)$ converges to our given $(\varphi_t)$ as in Proposition \ref{lem:6}: the first in \eqref{eq:uniformtest} grants that $(\ppi^\eps)$ is tight in $\prob{C([0,1],\X)}$ (because $\gamma\mapsto\int_0^1|\dot\gamma_t|^2\,\d t$ has compact sublevels) and thus  up to pass to a subsequence, not relabeled, we can assume that $(\ppi^{\eps_n})$ weakly converges to some $\ppi\in\prob{C([0,1],\X)}$. The second in \eqref{eq:uniformtest} and the bound \eqref{eq:lapcontr2} grant that the term with the Laplacian in \eqref{eq:k6} vanishes in the limit and thus taking into account the lower semicontinuity of the 2-energy we deduce that
\[
\int \varphi_{t_0}\circ\e_1- \varphi_{t_1}\circ\e_0\,\d\ppi\geq \frac1{2(t_1-t_0)}\iint_0^1|\dot\gamma_t|^2\,\d t\,\d\ppi\geq  \frac1{2(t_1-t_0)}\int\sfd^2(\gamma_0,\gamma_1)\,\d\ppi(\gamma).
\]
Now notice that \eqref{eq:o1} implies that 
\begin{equation}
\label{eq:hjcurve}
\frac{\sfd^2(\gamma_0,\gamma_1)}{2(t_1-t_0)}\geq \varphi_{t_0}(\gamma_1)-\varphi_{t_1}(\gamma_0)
\end{equation} 
for any curve $\gamma$, hence the above gives
\[
\int \varphi_{t_0}\circ\e_1- \varphi_{t_1}\circ\e_0\,\d\ppi\geq  \frac1{2(t_1-t_0)}\int\sfd^2(\gamma_0,\gamma_1)\,\d\ppi(\gamma)\geq \int \varphi_{t_0}\circ\e_1- \varphi_{t_1}\circ\e_0\,\d\ppi
\]
thus forcing the inequalities to be equalities. In particular, equality in \eqref{eq:hjcurve} holds for $\ppi$-a.e.\ $\gamma$ and since $(\e_0)_*\ppi=\mm$, this is the same as to say that for $\mm$-a.e.\ $y\in \X$ equality holds in \eqref{eq:o1}. Since both sides of $\eqref{eq:o1}$ are continuous in $y$, we conclude that equality holds for any $y\in \X$. 

\noindent{\bf Other properties of $\varphi_t$}. 
The fact that $(\varphi_t)\in AC_{loc}((0,1],C(\X))\cap L^\infty_{loc}((0,1],W^{1,2}(\X))$ is a direct consequence of \eqref{eq:hl1} and Proposition \ref{pro:11}.

Up to extract a further subsequence - not relabeled - we can assume that the curves $(\mu^{\eps_n}_t)$ converge to a limit curve $(\mu_t)$ as in Proposition \ref{lem:6}. We claim that for any $t_0,t_1\in(0,1]$, $t_0<t_1$ it holds 
\begin{equation}
\label{eq:perpot}
 -\int \varphi_{t_1}\,\d\mu_{t_1}+\int \varphi_{t_0}\,\d\mu_{t_0}\geq \frac1{2(t_1-t_0)}W_2^2(\mu_{t_0},\mu_{t_1}).
\end{equation}
To see this, start noticing that from  Proposition \ref{pro:7} it is clear that $t\mapsto\int\varphi^\eps_t\rho^\eps_t\,\d\mm$ is in $C((0,1])\cap AC_{loc}((0,1))$ and that it holds
\[
\begin{split}
-\frac\d{\d t}\int\varphi^\eps_t\rho^\eps_t\,\d\mm=\int\Big(-\frac{|\nabla\varphi^\eps_t|^2}2-\frac\eps2\Delta\varphi^\eps_t-\la\nabla\varphi^\eps_t,\nabla\vartheta^\eps_t \ra\Big)\rho^\eps_t\,\d\mm\qquad{\rm a.e.}\ t\in(0,1).
\end{split}
\]
Integrating and recalling that $\varphi^\eps_t=\frac\eps2\log\rho^\eps_t-\vartheta^\eps_t$ we deduce
\[
 -\int \varphi^\eps_{t_1}\,\d\mu^\eps_{t_1}+\int \varphi^\eps_{t_0}\,\d\mu^\eps_{t_0}=\iint_{t_0}^{t_1}\Big(\frac{|\nabla\vartheta^\eps_t|^2}2-\frac{\eps^2}8|\nabla\log\rho^\eps_t|^2-\frac\eps2\Delta\varphi^\eps_t\Big)\rho^\eps_t\,\d t\,\d\mm.
\]
As already noticed in the proof of point $(i)$ of Proposition \ref{lem:6}, $(\mu^\eps_t)$ and $(\nabla\vartheta^\eps_t)$ satisfy the assumptions of Theorem \ref{thm:GH},   thus from such theorem we deduce that
\[
\iint_{t_0}^{t_1}\frac{|\nabla\vartheta^\eps_t|^2}2\rho^\eps_t\,\d t\,\d\mm=\frac12\int_{t_0}^{t_1}|\dot\mu^\eps_t|^2\,\d t\geq\frac1{2(t_1-t_0)}W_2^2(\mu^\eps_{t_0},\mu^\eps_{t_1}).
\]
Therefore
\[
 -\int \varphi^\eps_{t_1}\,\d\mu^\eps_{t_1}+\int \varphi^\eps_{t_0}\,\d\mu^\eps_{t_0}\geq\frac1{2(t_1-t_0)}W_2^2(\mu^\eps_{t_0},\mu^\eps_{t_1})+\iint_{t_0}^{t_1}\Big(-\frac{\eps^2}8|\nabla\log\rho^\eps_t|^2-\frac\eps2\Delta\varphi^\eps_t\Big)\rho^\eps_t\,\d t\,\d\mm.
\]
We now pass to the limit in $\eps=\eps_n\downarrow0$: the left hand side trivially converges to the left hand side of \eqref{eq:perpot} while $W_2^2(\mu^{\eps_n}_{t_0},\mu^{\eps_n}_{t_1})\to W_2^2(\mu_{t_0},\mu_{t_1})$, the contribution of the term with $\Delta\varphi^\eps_t$ vanishes in the limit by \eqref{eq:linftybound} and \eqref{eq:lapcontr2}, while the one with $|\nabla\log\rho^\eps_t|$ vanishes by \eqref{eq:b2}. Hence \eqref{eq:perpot} is proved.

Now notice that  \eqref{eq:hl1} can be rewritten as
\[
-(t_1-t_0)\varphi_{t_1}=\big((t_1-t_0)\varphi_{t_0}\big)^c,
\]
so that in particular  $-(t_1-t_0)\varphi_{t_1}$ is $c$-concave and $(-(t_1-t_0)\varphi_{t_1})^c\geq (t_1-t_0)\varphi_{t_0}$. Hence both \eqref{eq:costointermedio} and the fact that $-(t_1-t_0)\varphi_{t_1}$ is a Kantorovich potential follow from
\[
\begin{split}
 \frac12W_2^2(\mu_{t_0},\mu_{t_1})&\geq \int -(t_1-t_0)\varphi_{t_1}\,\d\mu_{t_1}+\int (-(t_1-t_0)\varphi_{t_1})^c\,\d\mu_{t_0}\\
 &\geq \int -(t_1-t_0)\varphi_{t_1}\,\d\mu_{t_1}+\int (t_1-t_0)\varphi_{t_0}\,\d\mu_{t_0}\stackrel{\eqref{eq:perpot}}\geq \frac12W_2^2(\mu_{t_0},\mu_{t_1})
 \end{split}
\]
The claims about $(\psi_t)$ are proved in the same way.

In the case of additional regularity of $\rho_0,\rho_1$, taking into account the fact that $(\varphi_t)\in C([0,1],L^1(\mm))$and $\sup_t\Lip(\varphi_t)<\infty $  (from Proposition \ref{lem:6}), it is easy to see that we can pass to the limit in $t_0\downarrow0$ to get that \eqref{eq:hl1} holds even for $t_0=0$ - see for instance the arguments used in Proposition \ref{pro:v0} below. Then the other properties follow from the arguments already used in this step and Proposition \ref{pro:11}.

\noindent{\bf $(\mu_t)$ is a geodesic}.   Let $[t_0,t_1]\subset(0,1)$, pick $t\in[0,1]$ and put $t_0':=(1-t)t_1+tt_0$. We know that $-(t_1-t_0)\varphi_{t_1}$ and $-t(t_1-t_0)\varphi_{t_1}$ are Kantorovich potentials from $\mu_{t_1}$ to $\mu_{t_0}$ and from $\mu_{t_1}$ to $\mu_{t_0'}$ respectively and thus by point $(ii)$ of Theorem \ref{thm:bm} we deduce
\[
W_2^2(\mu_{t_0},\mu_{t_1})=\int|\d((t_1-t_0)\varphi_{t_1})|^2\,\d\mu_{t_1}=\frac1{t^2}\int|\d((t_1-t_0')\varphi_{t_1})|^2\,\d\mu_{t_1}=\frac{(t_1-t_0)^2}{(t_1-t_0')^2}W_2^2(\mu_{t_1},\mu_{t_0'}).
\]
Swapping the roles of $t_0,t_1$ and using the $\psi$'s in place of the $\varphi$'s we then get
\[
W_2(\mu_{t_1'},\mu_{t_0'})=\frac{t_1'-t_0'}{t_1-t_0}W_2(\mu_{t_1},\mu_{t_0})\qquad\forall [t_0',t_1']\subset [t_0,t_1]\subset (0,1).
\]
This grants that the restriction of $(\mu_t)$ to any interval $[t_0,t_1]\subset(0,1)$ is a constant speed geodesic. Since $(\mu_t)$ is continuous on the whole $[0,1]$, this gives the conclusion. Since in this situation the $W_2$-geodesic connecting $\mu_0$ to $\mu_1$ is unique (recall point $(i)$ of Theorem \ref{thm:bm}), by the arbitrariness of the subsequences chosen we also proved the uniqueness of the limit curve $(\mu_t)$.
\end{proof}
\begin{Remark}[The vanishing viscosity limit]{\rm
The part of this last proposition concerning the properties of the $\varphi^\eps_t$'s is valid in a context wider than the one provided by Schr\"odinger problem: we could restate the result by saying that if $(\varphi^\eps_t)$ solves
\begin{equation}
\label{eq:hjvv}
\ddt\varphi^\eps_t=\frac12{|\nabla\varphi_t^\eps|^2}+\frac\eps2\Delta\varphi^\eps_t
\end{equation}
and $\varphi^\eps_0$ uniformly converges to some  $\varphi_0$, then $\varphi^\eps_t$ uniformly converges to $\varphi_t:=-Q_t(-\varphi_0)$.

In this direction, it is worth recalling that in \cite{AF14} and \cite{GS15} it has been developed a theory of viscosity solutions for some first-order Hamilton-Jacobi equations on metric spaces. This theory applies in particular to the equation 
\begin{equation}
\label{eq:hjv}
\ddt\varphi_t=\frac12\lip(\varphi_t)^2
\end{equation}
whose only viscosity solution is given by the formula $\varphi_t:=-Q_t(-\varphi_0)$.

Therefore, we have just proved that if one works not only on a metric space, but on a metric measure space which is a $\RCD^*(K,N)$ space, then the solutions of the viscous approximation \eqref{eq:hjvv} converge to the unique viscosity solution of \eqref{eq:hjv}, in accordance with the classical case.
}\fr\end{Remark}
\begin{Remark}{\rm
It is not clear whether the `full' families $\varphi^\eps_t,\psi^\eps_t$ converge as $\eps\downarrow0$ to a unique limit. This is related to the non-uniqueness of the  Kantorovich potentials in the classical optimal transport problem.
}\fr\end{Remark}

 the following lemma. It could be directly deduced from the results obtained by Cheeger in \cite{Cheeger00}, however, the additional regularity assumptions on both the space and the function allow for a `softer' argument based on the metric Brenier's theorem, which we propose:
\begin{Lemma}\label{lem:12}
Let $(\Y,\sfd_\Y,\mathfrak{m}_\Y)$ be a  $\RCD^*(K,N)$ space, possibly not compact, with $K \in \mathbb{R}$ and $N \in [1,\infty)$ and let $\phi : \X \to \R\cup\{-\infty\}$ be a $c$-concave function not identically $-\infty$. Let $\Omega$ be the interior of the set $\{\phi>-\infty\}$. Then $\phi$ is locally Lipschitz on $\Omega$ and
\[
\lip\,\phi = |\d\phi|, \quad \mm\textrm{-a.e. on }\Omega.
\]
\end{Lemma}
\begin{proof} Lemma 3.3 in \cite{GigliRajalaSturm13} grants that $\phi$ is locally Lipschitz on $\Omega$ and that $\partial^c\phi(x)\neq\emptyset$ for every $x\in \Omega$. The same lemma also grants that for $K\subset \Omega$ compact, the set $\cup_{x\in K}\partial^c\phi(x)$ is bounded. Recalling that $\partial^c\phi$ is the set of $(x,y)\in \Y^2$ such that
\[
\phi(x)+\phi^c(y)=\frac12\sfd^2(x,y)
\]
and that $\phi,\phi^c$ are upper semicontinuous, we see that $\partial^c\phi$ is closed. Hence for $K\subset \Omega$ compact the set $\cup_{x\in K}\partial^c\phi(x)$ is compact and not empty and thus  by the Kuratowski--Ryll-Nardzewski  Borel selection theorem we deduce the existence of a Borel map $T:\Omega\to \Y$ such that $T(x)\in\partial^c\phi(x)$ for every $x\in \Omega$.  

Pick  $\mu\in\probt \Y$ with $\supp(\mu)\subset\subset\Omega$ and $\mu\leq C\mm$ for some $C>0$ and set  $\nu := T_{*}\mu$. By construction, $\mu,\nu$ have both bounded support,  $T$ is an optimal map and $\phi$ is a Kantorovich potential from $\mu$ to $\nu$. 

Hence point $(iii)$ of Theorem \ref{thm:bm}  applies and since $\lip\,\phi=\max\{|D^+\phi|,|D^-\phi|\}$, by the arbitrariness of $\mu$ to conclude it is sufficient to show that $|D^+\phi|=|D^-\phi|$ $\mm$-a.e.. This easily follows from the fact that $\mm$ is doubling and $\phi$ Lipschitz, see Proposition 2.7 in \cite{AmbrosioGigliSavare11}.
\end{proof}
With this said, we can now show that the energies of the Schr\"odinger potentials converge to the energy of the limit ones:
\begin{Proposition}\label{thm:10} With the same assumptions and notations as in Setting \ref{set} the following holds.

Let  $\eps_n\downarrow 0$ be a sequence such that $(\varphi^{\eps_n}_t), (\psi^{\eps_n}_t)$ converge to limit curves $(\varphi_t),(\psi_t)$ as in Proposition \ref{lem:6}. Then for every $\delta\in(0,1)$ we have
\begin{equation}
\label{eq:energialimite}
\begin{split}
\lim_{n \to \infty}\iint_{\delta}^1|\d \varphi^{\varepsilon_n}_t|^2\,\d t\,\d\mm &= \iint_{\delta}^1|\d\varphi_t|^2\,\d t\,\d\mm,\\
\lim_{n \to \infty}\iint_0^{1-\delta}|\d \psi^{\varepsilon_n}_t|^2\,\d t\,\d\mm &= \iint_0^{1-\delta}|\d\psi_t|^2\,\d t\,\d\mm.
\end{split}
\end{equation}
If we further assume that $\rho_0,\rho_1 \in \testipp{\X}$, $\delta$ can be chosen equal to 0.
\end{Proposition}
\begin{proof}
Fix $\delta \in\, (0,1)$ and notice that from the formula for $\ddt\varphi^\eps_t$ we get
\[
\int\varphi^\eps_1-\varphi^\eps_\delta\,\d\mm=\frac12\iint_\delta^1|\d \varphi^\eps_t|^2+\eps\Delta\varphi^\eps_t\,\d t\,\d\mm.
\]
Choosing $\eps:=\eps_n$, letting $n\to\infty$ and using the uniform bound \eqref{eq:lapcontr2} we obtain that
\begin{equation}
\label{eq:k7}
\lim_{n\to\infty}\frac12\iint_\delta^1|\d \varphi^{\eps_n}_t|^2\,\d t\,\d\mm=\lim_{n\to\infty}\int\varphi^{\eps_n}_1-\varphi^{\eps_n}_\delta\,\d\mm=\int\varphi_1-\varphi_\delta\,\d\mm.
\end{equation}
Combining \eqref{eq:55} and \eqref{eq:hl1} we see that for any $x \in \X$ it holds
\[
\ddt \varphi_t(x) = \frac{1}{2}\big((\lip\,\varphi_t)(x)\big)^2\qquad a.e.\ t\in[0,1].
\]
By Fubini's theorem we see that the same identity holds for $\mathscr{L}^1 \times \mm$-a.e.\ $(t,x) \in [\delta,1] \times \X$.  The identity \eqref{eq:hl1}  also grants that $\varphi_t$ is a multiple of a $c$-concave function, thus the thesis of Lemma \ref{lem:12} is valid for $\varphi_t$  and recalling that $(\varphi_t)\in AC_{loc}((0,1],L^1(\X))$ we deduce that  
\[
\int\varphi_1-\varphi_\delta\,\d\mm=\int_\delta^1 \ddt\int\varphi_t\d\mm\,\d t = \iint_{\delta}^1\frac{|\d\varphi_t|^2}{2}\,\d t\,\d\mm,
\]
which together with \eqref{eq:k7} gives the first in \eqref{eq:energialimite}. The second is proved in the same way.

For the last statement we simply recall that from Proposition \ref{lem:6} we know that under the additional assumptions on $\rho_0,\rho_1$ we have that $(\varphi^{\eps_n}_t),(\psi^{\eps_n}_t)$ converge  to $(\varphi_t),(\psi_t)$ respectively in $C([0,1],L^1(\X))$. Then we argue as above.
\end{proof}
As a direct consequence of the limit \eqref{eq:energialimite} and the equi-Lipschitz bounds \eqref{eq:lipcontr} we obtain:
\begin{Corollary}\label{cor:convd} With the same assumptions and notations as in Setting \ref{set} the following holds.

Let  $\eps_n\downarrow 0$ be a sequence such that $(\varphi^{\eps_n}_t), (\psi^{\eps_n}_t)$ converge to limit curves $(\varphi_t),(\psi_t)$ as in Proposition \ref{lem:6}. Then for every $\delta\in(0,1)$ we have
\begin{equation}
\label{eq:limited}
\begin{array}{rlll}
(\d\varphi^{\eps_n}_t)\quad&\to\quad (\d\varphi_t)&&\text{ in }\quad L^2([\delta,1],L^2(T^*\X))\\
(\d\psi^{\eps_n}_t)\quad&\to\quad (\d\psi_t)&&\text{ in }\quad L^2([0,1-\delta],L^2(T^*\X))\\
(\d\varphi^{\eps_n}_t\otimes\d\varphi^{\eps_n}_t)\quad&\to\quad (\d\varphi_t\otimes\d\varphi_t)&& \text{ in }\quad L^2([\delta,1],L^2((T^*)^{\otimes 2}\X))\\
(\d\psi^{\eps_n}_t\otimes\d\psi^{\eps_n}_t)\quad&\to\quad (\d\psi_t\otimes\d\psi_t)&&\text{ in }\quad L^2([0,1-\delta],L^2((T^*)^{\otimes 2}\X))\\
(\d\varphi^{\eps_n}_t\otimes\d\psi^{\eps_n}_t)\quad&\to\quad (\d\varphi_t\otimes\d\psi_t)&&\text{ in }\quad L^2([\delta,1-\delta],L^2((T^*)^{\otimes 2}\X))
\end{array}
\end{equation}
If we further assume that $\rho_0,\rho_1 \in \testipp{\X}$, $\delta$ can be chosen equal to 0.
\end{Corollary}
\begin{proof} Start noticing that the closure of the differential grants that $\d\varphi^{\eps_n}_t\weakto \d\varphi_t$ in $L^2(T^*\X)$ for all $t\in(0,1]$. This and the fact that $(\d\varphi^{\eps_n}_t)$ is equibounded in $L^2([\delta,1],L^2(T^*\X))$, as a  direct consequence of \eqref{eq:lipcontr}, are sufficient to ensure that $(\d\varphi^{\eps_n}_t)\weakto (\d\varphi_t)$ in $L^2([\delta,1],L^2(T^*\X))$. Given that the first in \eqref{eq:energialimite} grants convergence of the $L^2([\delta,1],L^2(T^*\X))$-norms, we deduce strong convergence. This establishes the first limit.

Now observe that for every $\omega\in L^2([\delta,1],L^2(T^*\X))$ the fact that $|\d\varphi^{\eps_n}_t|$ is uniformly bounded in $L^\infty([\delta,1]\times \X)$ and the strong $L^2$-convergence just proved ensure that $\la \d\varphi^{\eps_n}_t,\omega_t\ra\to \la \d\varphi_t,\omega_t\ra$ in $L^2([\delta,1]\times \X)$. It follows that for any $\omega_1,\omega_2\in L^2([\delta,1],L^2(T^*\X))$ we have
\[
\iint_\delta^1\la \d\varphi^{\eps_n}_t,\omega_{1,t}\ra\la \d\varphi^{\eps_n}_t,\omega_{2,t}\ra\,\d t\,\d\mm\quad\to\quad\iint_\delta^1\la \d\varphi_t,\omega_{1,t}\ra\la \d\varphi_t,\omega_{2,t}\ra\,\d t\,\d\mm
\]
and thus to conclude it remains to prove that
\[
\iint_\delta^1|\d\varphi^{\eps_n}_t\otimes \d\varphi^{\eps_n}_t|_\HS^2\,\d t\,\d\mm\quad\to\quad\iint_\delta^1|\d\varphi_t\otimes \d\varphi_t|_\HS^2\,\d t\,\d\mm.
\]
Since $|v\otimes v|_\HS^2=|v|^4$ this is a direct consequence of the fact that  $|\d\varphi^{\eps_n}_t|$ is uniformly bounded and converges to $|\d\varphi_t|$ in $L^2([\delta,1]\times \X)$. Hence also the third limit is established.

The other claims  follow by analogous arguments and the last statement follows from the equi-Lipschitz continuity of the $\varphi^\eps_t,\psi^\eps_t$'s that holds in this case (from the last part of Proposition \ref{pro:1}) and the fact that we can take $\delta=0$ in the first in \eqref{eq:energialimite}.
\end{proof}

\bigskip

The estimates that we have for the functions $\varphi$'s tell nothing about their regularity as $t\downarrow0$ and similarly little we know so far about the $\psi$'s for $t\uparrow1$. We now see in which sense limit functions $\varphi_0,\psi_1$ exist. This is not needed  for the proof of our main result, but we believe it is relevant in its own.

Thus let us fix $\eps_n\downarrow0$ so that $\varphi^{\eps_n}_t\to \varphi_t$ for $t\in(0,1]$ and $\psi^{\eps_n}_t\to \psi_t$ for $t\in[0,1)$ as in Proposition \ref{lem:6}.  Then  define the functions $\varphi_0,\psi_1:\X\to\R\cup\{-\infty\}$ as
\begin{equation}
\label{eq:varphi0}
\begin{split}
\varphi_0(x):=\inf_{t\in(0,1]}\varphi_t(x)=\lim_{t\downarrow0}\varphi_t(x),\\
\psi_1(x):=\inf_{t\in[0,1)}\psi_t(x)=\lim_{t\uparrow1}\psi_t(x).
\end{split}
\end{equation}
Notice that the fact that the $\inf$ are equal to the stated limits is a consequence of formulas \eqref{eq:hl1}, \eqref{eq:hl2}, which directly imply that for every $x\in \X$ the maps $t\mapsto\varphi_t(x)$ and $t\mapsto\psi_{1-t}(x)$ are non-decreasing.

The main properties of $\varphi_0,\psi_1$ are collected in the following proposition:
\begin{Proposition}\label{pro:v0} With the same assumptions and notations as in Setting \ref{set}, and for $\varphi_0,\psi_1$ defined by \eqref{eq:varphi0} the following holds.

\begin{itemize}
\item[i)] The functions $-\varphi_t$ (resp. $-\psi_t$) $\Gamma$-converge to $-\varphi_0$ (resp. $-\psi_1$) as $t\downarrow0$ (resp. $t\uparrow 1$).

\item[ii)] For every $t\in(0,1]$ we have
\[
Q_t(-\varphi_0)=-\varphi_t\qquad\qquad Q_t(-\psi_1)=-\psi_{1-t}.
\]
\item[iii)] It holds
\[
\varphi_0(x)=
\left\{\begin{array}{ll}
-\psi_0(x)&\quad\text{if }x\in\supp(\rho_0)\\
-\infty&\quad\text{otherwise}
\end{array}
\right.
\qquad\quad
\psi_1(x)=
\left\{\begin{array}{ll}
-\varphi_1(x)&\quad\text{if }x\in\supp(\rho_1)\\
-\infty&\quad\text{otherwise}
\end{array}
\right.
\]
\item[iv)] We have
\[
\int\varphi_0\rho_0\,\d\mm+\int\psi_1\rho_1\,\d\mm=\frac12W_2^2(\mu_0,\mu_1).
\]
\item[v)] Define $\varphi^\eps_0$ on $\supp(\rho_0)$ as $\varphi^\eps_0:=\eps\log(f^\eps)$ and let $\eps_n\downarrow0$ be such that $\varphi^{\eps_n}_t,\psi^{\eps_n}_t$ converge to $\varphi_t,\psi_t$ as $n\to\infty$ as in Proposition \ref{lem:6}. 

Then the functions $\rho_0\varphi^{\eps_n}_0$,  set to be 0 on $\X\setminus\supp(\rho_0)$, converge to $\rho_0\varphi_0$  in $L^\infty(\mm)$  as $n\to\infty$. 

With the analogous definition of $\rho_1\psi^{\eps_n}_1$ we have that these converge to $\rho_1\psi_1$ in $L^\infty(\mm)$ as $n\to\infty$. 
\end{itemize}
\end{Proposition}
\begin{proof} We shall prove the claims for $\varphi_0$ only, as those for $\psi_1$ follow along similar lines.\\
\noindent{\bf (i)} For the $\Gamma-\lims$ inequality we simply observe that by definition $-\varphi_0(x)=\lim_{t\downarrow0}-\varphi_t(x)$. To prove the $\Gamma-\limi$ inequality, use the fact that $-\varphi_t\geq-\varphi_s$ for $0<t\leq s$ and the continuity of $\varphi_s$:  for given $(x_t)$ converging to $x$ we have
\[
\limi_{t\downarrow0}-\varphi_t(x_t)\geq\limi_{t\downarrow0}-\varphi_s(x_t)=-\varphi_s(x)\qquad\forall s>0.
\]
The conclusion follows letting $s\downarrow0$.

\noindent{\bf (ii)} This claim follows from the general properties of $\Gamma$-convergence; we quickly report the argument. From $-\varphi_0\geq-\varphi_s$ we deduce that 
\[
Q_t(-\varphi_0)\geq Q_t(-\varphi_s)\stackrel{\eqref{eq:hl1}}=-\varphi_{t+s}\qquad\forall s\in(0,1]
\]
and thus letting $s\downarrow 0$ and using the continuity of $t\mapsto \varphi_t\in C(\X)$ we obtain $Q_t(-\varphi_0)\geq -\varphi_t$. For the opposite inequality  fix $x\in \X$, a sequence $t_n\downarrow0$ and find $y_n\in \X$ such that  $Q_t(-\varphi_{t_n})(x)=\frac{\sfd^2(x,y_n)}{2t}-\varphi_{t_n}(y_n)$. By compactness, up to pass to a subsequence we can assume that $y_n\to y$ for some $y\in \X$, so that taking into account the $\Gamma-\limi$ inequality previously proved we get
\[
\frac{\sfd^2(x,y)}{2t}-\varphi_{0}(y)\leq \limi_{n\to\infty}\frac{\sfd^2(x,y_n)}{2t}-\varphi_{t_n}(y_n)=\limi_{n\to\infty}Q_t(-\varphi_{t_n})(x)\stackrel{\eqref{eq:hl1}}=\limi_{n\to\infty}-\varphi_{t_n+t}(x)=-\varphi_t(x)
\]
which shows that  $Q_t(-\varphi_0)(x)\leq -\varphi_t(x)$, as desired.

\noindent{\bf (iii)} For any $t\in(0,1]$ we have
\[
\varphi_0\leq \varphi_t\stackrel{\eqref{eq:63}}\leq-\psi_t
\]
so that letting $t\downarrow0$ and using the continuity of $[0,1)\ni t\mapsto\psi_t\in C(\X)$ we deduce that 
\[
\varphi_0\leq -\psi_0\qquad{\rm on}\ \X.
\]
Now notice that the fact that $-\varphi_0\leq \Gamma-\limi(-\varphi_t)$ implies that
\begin{equation}
\label{eq:magg}
\varphi_0(\gamma_0)\geq\lims_{t\downarrow 0}\varphi_t(\gamma_t)\qquad\forall \gamma\in C([0,1],\X).
\end{equation}
Let $\ppi$ be the lifting  of the $W_2$-geodesic $(\mu_t)$ (recall point $(i)$ of Theorem \ref{thm:bm}); taking into account that the evaluation maps $\e_t:C([0,1],\X)\to \X$ are continuous and that $\supp(\ppi)$ is a compact subset of $C([0,1],\X)$ it is easy to see that for every  $\gamma\in\supp(\ppi)$ and $t\in[0,1]$ we have $\gamma_t\in \supp(\mu_t)$ and viceversa for every $x\in\supp(\mu_t)$ there is $\gamma\in\supp(\ppi)$ with $\gamma_t=x$.  

Thus let $x\in\supp(\rho_0)=\supp(\mu_0)$ and find $\gamma\in \supp(\ppi)$ with $\gamma_0=x$: from the fact that $\gamma_t\in\supp(\mu_t)$ and \eqref{eq:63} we get
\[
\varphi_0(x)\stackrel{\eqref{eq:magg}}\geq\lims_{t\downarrow 0}\varphi_t(\gamma_t)=\lims_{t\downarrow0}-\psi_t(\gamma_t)=-\psi_0(x).
\]
Thus to conclude it remains to   prove that $\varphi_0=-\infty$ outside $\supp(\rho_0)$ and to this aim we shall use the Gaussian estimates \eqref{eq:gaussest}. Let $v_\eps:=\inf_y\mm(B_{\sqrt\eps}(y))$ and  $V_\eps:=\sup_y\mm(B_{\sqrt\eps}(y))$. We start claiming that 
\begin{equation}
\label{eq:volneg}
\lim_{\eps\downarrow0}\eps\log( v_\eps)=\lim_{\eps\downarrow0}\eps\log (V_\eps)=0.
\end{equation}
Indeed on one side since $\mm(\X)=1$ we have $\eps\log( v_\eps)\leq \eps\log( V_\eps)\leq 0$ for every $\eps>0$. On the other one, letting $C$ be the doubling constant of $(\X,\sfd,\mm)$ we have
\[
\mm(B_{\sqrt\eps}(y))\geq C^{\log_2(D/\sqrt\eps)+1}\mm(\X)= C^{\log_2(D/\sqrt\eps)+1}\qquad\forall y\in \X.
\]
Thus $v_\eps\geq C^{\log_2(D/\sqrt\eps)+1}$ from which it  follows that $\limi_{\eps}\eps\log(v_\eps)\geq 0$ and thus \eqref{eq:volneg} is proved. Now use the first inequality in \eqref{eq:gaussest} and the fact that $f^\eps\otimes g^\eps \hR^\eps$ is a probability measure (by construction - recall our Setting \ref{set}) to obtain
\begin{equation}
\label{eq:integrfg}
\int f^\eps\,\d\mm\int g^\eps \,\d\mm\leq C_1V_\eps e^{\frac{C_1D^2}\eps}\int f^\eps (x)g^\eps (y)\,\d \hR^\eps(x,y)= C_1V_\eps e^{\frac{C_1D^2}\eps}.
\end{equation}
Observing that by construction we have $\supp(f^\eps)=\supp(\rho_0)$ for every $\eps>0$,  the second in \eqref{eq:gaussest} yields
\[
\begin{split}
f^\eps_t(x)&=\h_{\eps t/2}f^\eps(x)=\int f^\eps(y)\hr_{\eps t/2}(x,y)\,\d\mm(y)\leq\frac{C_2}{v_{\eps t/2}}e^{-\frac{\sfd^2(x,\supp(\rho_0))}{3\eps t}}\int f^\eps\,\d\mm,\\
g^\eps_t(x)&=\h_{\eps (1-t)/2}g^\eps(x)=\int g^\eps(y)\hr_{\eps(1-t)/2}(x,y)\,\d\mm(y)\leq\frac{C_2}{v_{\eps (1-t)/2}}\int g^\eps\,\d\mm,
\end{split}
\]
for every $t\in(0,1)$ and thus coupling these bounds with \eqref{eq:integrfg} we obtain
\[
\rho^\eps_t(x)=f^\eps_t(x)g^\eps_t(x)\leq \frac{C_1C_2^2V_\eps}{v_{\eps (1-t)/2}v_{\eps t/2}}e^{\frac{C_1D^2}\eps}e^{-\frac{\sfd^2(x,\supp(\rho_0))}{3\eps t}}\qquad\forall x\in\X,\ t\in(0,1).
\]
Therefore recalling \eqref{eq:volneg} we obtain
\begin{equation}
\label{eq:altdens}
\lims_{\eps\downarrow0}\eps\log(\rho^\eps_t(x))\leq C_1D^2-\frac{\sfd^2(x,\supp(\rho_0))}{3t}\qquad\forall x\in\X,\ t\in(0,1).
\end{equation}
Now let  $\eps_n\downarrow0$ be the sequence such that $\varphi^{\eps_n}_t,\psi^{\eps_n}_t$ converge to $\varphi_t,\psi_t$ as in Proposition \ref{lem:6} and put  $S:=\sup_{\eps\in(0,1),t\in[0,1/2]}\|\psi^\eps_t\|_{L^\infty}<\infty$. The inequality
\[
\varphi_t(x)=\lim_{n\to\infty}\varphi^{\eps_n}_t(x)\leq \lims_{n\to\infty}\eps_n\log(\rho^{\eps_n}_t(x))-\lim_{n\to\infty}\psi^{\eps_n}_t(x)\stackrel{\eqref{eq:altdens}}\leq S+C_1D^2-\frac{\sfd^2(x,\supp(\rho_0))}{3t}
\]
shows that if $x\notin\supp(\rho_0)$ we have $\varphi_0(x)=\lim_{t\downarrow0}\varphi_t(x)=-\infty$, as desired.

\noindent{\bf (iv)} By the point $(iii)$ just proven we have
\[
\int\varphi_0\rho_0\,\d\mm+\int\psi_1\rho_1\,\d\mm=-\int\psi_0\rho_0\,\d\mm-\int\varphi_1\rho_1\,\d\mm
\]
so that taking into account the weak continuity of $t\mapsto\mu_t$ and the uniform continuity of $t\mapsto \varphi_t$ (resp. $t\mapsto\psi_t$) for $t$ close to 1 (resp. close to 0) we get
\[
\begin{split}
\int\varphi_0\rho_0\,\d\mm+\int\psi_1\rho_1\,\d\mm&\stackrel{\phantom{\eqref{eq:63}}}=\lim_{t\downarrow0}-\int\psi_t\rho_t\,\d\mm-\int\varphi_{1-t}\rho_{1-t}\,\d\mm\\
&\stackrel{\eqref{eq:63}}=\lim_{t\downarrow0}\int\varphi_t\rho_t\,\d\mm-\int\varphi_{1-t}\rho_{1-t}\,\d\mm\stackrel{\eqref{eq:costointermedio}}=\frac12W_2^2(\mu_0,\mu_1).
\end{split}
\] 
\noindent{\bf (v)} Since $\rho_0\in L^\infty(\mm)$, we also have $\rho_0\log(\rho_0)\in L^\infty(\mm)$. The claim then follows from the identity $\rho_0\varphi^\eps_0=\eps\rho_0\log\rho_0-\rho_0\psi^\eps_0$, the uniform convergence of $\psi^{\eps_n}_0$ to $\psi_0$ as $n\to\infty$ and the fact that $\psi_0=-\varphi_0$ on $\supp(\rho_0)$.
\end{proof}

\begin{Remark}[Entropic and transportation cost]{\rm For $\eps>0$  the \emph{entropic cost} from $\rho_0$ to $\rho_1$ is defined as
\[
I_\eps(\rho_0,\rho_1):=\inf H(\ggamma\,|\,\hR^\eps),
\]
the infimum being taken among all transport plans $\ggamma$ from $\mu_0:=\rho_0\mm$ to $\mu_1:=\rho_1\mm$. Hence  with our notation
\[
I_\eps(\rho_0,\rho_1)=H\big(f^\eps\otimes g^\eps \hR^\eps\,|\,\hR^\eps\big)=\frac1\eps\int\varphi^\eps_0\oplus\psi^\eps_1 f^\eps\otimes g^\eps \,\d \hR^\eps=\frac1\eps\Big(\int\varphi^\eps_0\rho_0\,\d\mm+\int\psi^\eps_1\rho_1\,\d\mm\Big)
\]
and by $(iv),(v)$ of the previous proposition we get
\[
\lim_{\eps\downarrow0}\eps \,I_\eps(\rho_0,\rho_1)=\frac12W_2^2(\mu_0,\mu_1).
\]
In other words, after the natural rescaling the entropic cost converges to the quadratic transportation cost, thus establishing another link between the Schr\"odinger problem and the transport one.

We emphasize that although this argument is new, the result is not,  not even on $\RCD^*(K,N)$ spaces: L\'eonard proved in \cite{Leonard12} that the same limit holds in a very abstract setting provided the heat kernel satisfies the appropriate large deviation principle $\eps\log \hr^\eps_t(x,y)\sim-\frac{\sfd^2(x,y)}{2}$. Since recently such asymptotic behavior for the heat kernel on $\RCD^*(K,N)$ spaces has been proved by Jiang-Li-Zhang in \cite{JLZ15}, L\'eonard's result applies. Thus in this remark we simply wanted to show an alternative proof of such limiting property. 
}\fr\end{Remark}
\subsection{Proof of the main theorem}
We start with the following simple continuity statement:
\begin{Lemma}\label{le:percon} With the same assumptions and notation as in Setting \ref{set}, let $t\mapsto\mu_t=\rho_t\mm$ be the $W_2$-geodesic from $\mu_0$ to $\mu_1$ and  $(\varphi_t)_{t\in(0,1]}$ and $(\psi_t)_{t\in[0,1)}$ any couple of limit functions given by Proposition \ref{lem:6}.

Then the maps
\[
\begin{array}{rll}
(0,1]\ni t\quad&\mapsto\quad \rho_t\,\d \varphi_t&\in L^2(T^*\X)\\
\ [0,1)\ni t\quad&\mapsto\quad \rho_t\,\d \psi_t&\in L^2(T^*\X)\\
(0,1]\ni t\quad&\mapsto\quad \rho_t\,\d \varphi_t\otimes\d\varphi_t&\in L^2((T^*)^{\otimes2}\X)\\
\ [0,1)\ni t\quad&\mapsto\quad \rho_t\,\d \psi_t\otimes\d\psi_t&\in L^2((T^*)^{\otimes2}\X)
\end{array}
\]
are all continuous w.r.t.\ the strong topologies.
\end{Lemma}
\begin{proof}
By Lemma \ref{lem:8} we know that for any $p<\infty$ we have $\rho_s\to\rho_t$ in $L^p(\mm)$ as $s\to t$ and thus in particular $\sqrt{\rho_s}\to\sqrt{\rho_t}$ as $s\to t$. The closure of the differential and the fact that $\varphi_s\to \varphi_t$ weakly in $W^{1,2}(\X)$ as $s\to t>0$ (as a consequence of $(\varphi_t)\in C((0,1],C(\X))\cap L^\infty_{loc}((0,1),W^{1,2}(\X))$, see Proposition \ref{pro:9}) grant that $\d\varphi_s\to\d\varphi_t$ weakly in $L^2(T^*\X)$. Together with the previous claim about the densities and the fact that the latter are uniformly bounded in $L^\infty(\mm)$, this is sufficient to conclude that $t\mapsto \sqrt{\rho_t}\d\varphi_t\in L^2(T^*\X)$ is weakly continuous.

We now claim that $t\mapsto \sqrt{\rho_t}\d\varphi_t\in L^2(T^*\X)$ is strongly continuous and to this aim we   show that their $L^2(T^*\X)$-norms are constant. To see this, recall that by Proposition \ref{pro:9} we know that for $t\in(0,1]$ the function $-(1-t)\psi_t$ is a Kantorovich potential from $\mu_t$ to $\mu_1$ while from \eqref{eq:63} and the locality of the differential we get that $|\d \varphi_t|=|\d\psi_t|$ $\mu_t$-a.e., thus by point $(iii)$ in Theorem \ref{thm:bm} we have that
\[
\int|\d\varphi_t|^2\rho_t\,\d\mm=\frac1{(1-t)^2}\int |\d(1-t)\psi_t|^2\rho_t\,\d\mm=\frac1{(1-t)^2}W_2^2(\mu_t,\mu_1)=W_2^2(\mu_0,\mu_1).
\]
Multiplying the $\sqrt{\rho_t}\d\varphi_t$'s by $\sqrt{\rho_t}$ and using again the $L^2(\mm)$-strong continuity of $\sqrt{\rho_t}$ and the uniform $L^\infty(\mm)$-bound we conclude that $t\mapsto {\rho_t}\d\varphi_t\in L^2(T^*\X)$ is strongly continuous, as desired.

To prove the strong continuity of $t\mapsto  \rho_t\,\d \varphi_t\otimes\d\varphi_t\in  L^2((T^*)^{\otimes2}\X)$ we argue as in Corollary \ref{cor:convd}: the strong continuity of $t\mapsto \sqrt{\rho_t}\d\varphi_t\in L^2(T^*\X)$ and the fact that these are, locally in $t\in(0,1]$, uniformly bounded, grant both that $t\mapsto \|\rho_t\d\varphi_t\otimes\d\varphi_t\|_{L^2((T^*)^{\otimes 2}\X)}$ is continuous and that $t\mapsto \rho_t\d\varphi_t\otimes\d \varphi_t\in L^2((T^*)^{\otimes 2}\X)$ is weakly continuous.

The claims about the $\psi_t$'s follow in the same way.
\end{proof}

We now have all the tools needed to prove our main result. Notice that we shall not make explicit use of Theorem \ref{thm:1i} but rather reprove it for (the restriction to $[\delta,1-\delta]$ of) entropic interpolations.
\begin{Theorem}\label{thm:main}
Let $(\X,\sfd,\mathfrak{m})$ be a compact $\RCD^*(K,N)$ space with $K \in \mathbb{R}$ and $N \in [1,\infty)$. Let $\mu_0,\mu_1\in\prob \X$ be  such that $\mu_0,\mu_1\leq C\mm$ for some $C>0$ and let $(\mu_t)$  be the unique $W_2$-geodesic connecting $\mu_0$ to $\mu_1$. Also, let $h\in H^{2,2}(\X)$.

Then the map 
\[
[0,1]\ni \ t\quad \mapsto\quad \int h\,\d\mu_t\ \in\R
\]
belongs to $C^2([0,1])$ and the following formulas hold for every $t\in[0,1]$:
\begin{equation}
\label{eq:derivate}
\begin{split}
\ddt \int h\,\d\mu_t&=\int\la\nabla h,\nabla\phi_t\ra\,\d\mu_t,\\
\frac{\d^2}{\d t^2}\int h\,\d\mu_t&=\int \H h(\nabla\phi_t,\nabla\phi_t)\,\d\mu_t,
\end{split}
\end{equation}
where $\phi_t$ is any function such that for some $s\neq t$, $s\in[0,1]$, the function $-(s-t)\phi_t$ is a Kantorovich potential from $\mu_t$ to $\mu_s$.
\end{Theorem}
\begin{proof} For the given $\mu_0,\mu_1$ introduce the notation of Setting \ref{set} and then find $\eps_n\downarrow0$ such that $(\varphi^{\eps_n}_t),(\psi^{\eps_n}_t)$ converge to limit curves $(\varphi_t),(\psi_t)$ as in Proposition \ref{lem:6}. 

By Lemma \ref{le:gradpot} we know that the particular choice of the $\phi_t$'s as in the statement does not affect the right hand sides in \eqref{eq:derivate}, we shall therefore prove that such formulas hold for the choice $\phi_t:=\psi_t$, which is admissible thanks to Proposition \ref{pro:9} whenever $t<1$. The case $t=1$ can be achieved swapping the roles of $\mu_0,\mu_1$ or, equivalently, with the choice $\phi_t=-\varphi_t$ which is admissible for $t>0$.

Fix $h\in H^{2,2}(\X)$ and for $t\in[0,1]$ set
\[
I_n(t):=\int h\,\d\mu^{\eps_n}_t\qquad\qquad I(t):=\int h\,\d\mu_t.
\]
The bound \eqref{eq:linftybound} grants that the $I_n$'s are uniformly bounded and the convergence in \eqref{eq:66} that $I_n(t)\to I(t)$ for any $t\in[0,1]$.

Since $(\rho^{\eps_n}_t)\in AC_{loc}((0,1),W^{1,2}(\X))$ we have that $I_n\in AC_{loc}((0,1))$ and, recalling the formula for $\ddt\rho^\eps_t$ given by Proposition \ref{pro:7}, that
\begin{equation}
\label{eq:der1}
\ddt I_n(t)=\int h\ddt\rho^{\eps_n}_t\,\d\mm=-\int h\,{\rm div}(\rho^{\eps_n}_t\nabla\vartheta^{\eps_n}_t)=\int \la\nabla h,\nabla\vartheta_t^{\eps_n}\ra\rho^{\eps_n}_t\,\d\mm.
\end{equation}
The fact that  $\vartheta_t=\frac{\psi_t-\varphi_t}2$ and  the bounds \eqref{eq:lipcontr} and \eqref{eq:linftybound} ensure that $\big|\ddt I_n(t)\big|$ is uniformly bounded in $n$ and $t\in[t_0,t_1]\subset(0,1)$ and the convergence properties \eqref{eq:limited} and  \eqref{eq:66} grant that 
\[
\iint_{t_0}^{t_1} \la\nabla h,\nabla\vartheta_t^{\eps_n}\ra\rho^{\eps_n}_t\,\d t\,\d\mm\quad\to\quad\iint_{t_0}^{t_1}  \la\nabla h,\nabla\vartheta_t\ra\rho_t\,\d t\,\d\mm.
\]
This is sufficient to pass  to the limit in the distributional formulation of $\ddt I_n(t)$ and taking into account that $I\in C([0,1])$ we have just proved that $I\in AC_{loc}((0,1))$ with
\begin{equation}
\label{eq:der11}
\ddt I(t)=\int \la\nabla h,\nabla\vartheta_t\ra\rho_t\,\d\mm
\end{equation}
for a.e.\ $t\in[0,1]$. Recalling that $\vartheta_t=\frac{\psi_t-\varphi_t}2$, \eqref{eq:63} and the locality of the differential we see that
\begin{equation}
\label{eq:gradt}
\nabla\vartheta_t=\nabla\psi_t\quad\rho_t\mm\ae\qquad\forall t\in[0,1),
\end{equation} 
and thus by Lemma \ref{le:percon} we see that the right hand side of \eqref{eq:der11} has a continuous representative in $t\in[0,1)$, which then implies  that $I\in C^1([0,1))$ and that the first in \eqref{eq:derivate} holds for any $t\in[0,1)$.

For the second derivative we assume for a moment that $h\in\testi \X$. Then we recall  that $(\rho^{\eps_n}_t),(\vartheta^{\eps_n}_t)\in AC_{loc}((0,1),W^{1,2}(\X))$ and consider the rightmost side of \eqref{eq:der1} to get that $\ddt I_n(t)\in AC_{loc}((0,1))$ and
\[
\frac{\d^2}{\d t^2}I_n(t)=\int\langle\nabla h,\nabla\ddt \vartheta^{\eps_n}_t\rangle\rho^{\eps_n}_t+\la\nabla h,\nabla \vartheta^{\eps_n}_t\ra\ddt\rho^{\eps_n}_t\,\d\mm
\]
for a.e.\ $t$, so that defining the `acceleration' $a^\eps_t$ as
\[
a^\eps_t:= -\Big(\frac{\varepsilon^2}{4} \Delta\log\rho^{\varepsilon}_t + \frac{\varepsilon^2}{8}|\nabla\log\rho^{\varepsilon}_t|^2\Big)
\]
and recalling the formula for $\ddt\vartheta^\eps_t$ given by Proposition \ref{pro:7} we have
\[
\begin{split}
\frac{\d^2}{\d t^2}I_n(t)&=\int\langle\nabla h,\nabla\Big(-\frac12|\nabla\vartheta^{\eps_n}_t|^2+ a^{\eps_n}_t\Big)\rangle\rho^{\eps_n}_t-\la\nabla h,\nabla \vartheta^{\eps_n}_t\ra{\rm div}(\rho^{\eps_n}_t\nabla\vartheta^{\eps_n}_t)\,\d \mm\\
&=\int\Big(-\frac12\langle\nabla h,\nabla |\nabla\vartheta^{\eps_n}_t|^2\rangle +\langle\nabla(\la\nabla h,\nabla \vartheta^{\eps_n}_t\ra),\nabla\vartheta^{\eps_n}_t\rangle+\la\nabla h,\nabla a^{\eps_n}_t\ra\Big)\rho^{\eps_n}_t\,\d\mm\\
(\text{by \eqref{eq:leibh}})\quad&=\int \H h(\nabla\vartheta^{\eps_n}_t,\nabla\vartheta^{\eps_n}_t)\rho^{\eps_n}_t\,\d\mm-\int \big(\Delta  h+\la\nabla h,\nabla\log\rho^{\eps_n}_t\ra\big) a^{\eps_n}_t\rho^{\eps_n}_t\,\d\mm.
\end{split}
\]
Since $\H h\in L^2(T^{*\otimes 2}\X)$ and  $\vartheta^\eps_t=\frac{\psi^\eps_t-\varphi^\eps_t}2$, by the limiting properties \eqref{eq:limited} and \eqref{eq:66} we know that 
\[
\int \H h(\nabla\vartheta^{\eps_n}_t,\nabla\vartheta^{\eps_n}_t)\rho^{\eps_n}_t\,\d\mm\quad\stackrel{n\to\infty}\to \quad\int \H h(\nabla\vartheta_t,\nabla\vartheta_t)\rho_t\,\d\mm\qquad\text{ in $L^1_{loc}(0,1)$}
\]
and since $|\nabla h|,\Delta h\in L^\infty(\X)$, by Lemma \ref{le:vanish} we deduce that
\[
\int \big(\Delta  h+\la\nabla h,\nabla\log\rho^{\eps_n}_t\ra\big) a^{\eps_n}_t\rho^{\eps_n}_t\,\d\mm\quad\to\quad 0\qquad\text{ in $L^1_{loc}(0,1)$}.
\]
Hence we can pass to the limit in the distributional formulation of $\frac{\d^2}{\d t^2}I_n$ to obtain that $\ddt I\in AC_{loc}((0,1))$ and
\begin{equation}
\label{eq:der2}
\frac{\d^2}{\d t^2}I(t)=\int \H h(\nabla\vartheta_t,\nabla\vartheta_t)\rho_t\,\d\mm
\end{equation}
for a.e.\ $t$. Using again \eqref{eq:gradt} and Lemma \ref{le:percon} we conclude that the right hand side of \eqref{eq:der2} is continuous on $[0,1)$, so that $I\in C^2([0,1))$ and the second in \eqref{eq:derivate} holds for every $t\in[0,1)$.

It remains to remove the assumption that $h\in \testi \X$. Thus pick $h\in H^{2,2}(\X)$ and put $h_s:=\h_sh$. As $s\downarrow0$ we clearly have $h_s\to h$ in $W^{1,2}(\X)$ and $\Delta h_s\to\Delta h$ in $L^2(\X)$. Thus the bound \eqref{eq:heslap} grants that $h_s\to h$ in $W^{2,2}(\X)$. By \eqref{eq:regflow} we know that $h_s\in\testi\X$ for every $s>0$, thus the conclusion of the theorem hold for the $h_s$'s.

Now notice that we can choose the $\phi_t$'s to be uniformly Lipschitz (e.g.\ by taking $\phi_t:=\psi_t$ for $t\geq 1/2$ and $\phi_t:=-\varphi_t$ for $t<1/2$). The uniform $L^\infty$ estimates \eqref{eq:linftyrcd}, the equi-Lipschitz continuity of $\phi_t$ and the $L^2$-convergence of ${h_k},\nabla {h_k},\H {h_k}$ to ${h},\nabla {h},\H {h}$ respectively grant that as $s\downarrow0$ we have that
\[
\begin{split}
\int h_s\,\d\mu_t&\qquad\to\qquad\int h\,\d\mu_t\\
\int\la\nabla h_s,\nabla\phi_t\ra\,\d\mu_t&\qquad\to\qquad \int\la\nabla h,\nabla\phi_t\ra\,\d\mu_t\\
\int \H {h_s}(\nabla\phi_t,\nabla\phi_t)\,\d\mu_t&\qquad\to\qquad \int \H {h}(\nabla\phi_t,\nabla\phi_t)\,\d\mu_t
\end{split}
\]
uniformly in $t\in[0,1]$. This is sufficient to conclude.
\end{proof}

\def\cprime{$'$} \def\cprime{$'$}

\end{document}